\documentclass[11pt,letterpaper]{amsart}

\usepackage[T1]{fontenc}	
\usepackage{graphicx}
\usepackage{amssymb,centernot}
\usepackage{upgreek}
\usepackage{mathrsfs}
\usepackage[usenames,dvipsnames]{xcolor}
\usepackage{enumerate}							
\usepackage[inline,shortlabels]{enumitem}
\usepackage{verbatim}								
\usepackage{dsfont}									

\usepackage{pdflscape}
\usepackage{afterpage}
\usepackage[normalem]{ulem}
\usepackage{soul}
\usepackage{anyfontsize}
\usepackage[
						colorlinks,				
						breaklinks,unicode,
						hypertexnames=false,
						citecolor=OliveGreen,
						linkcolor=Maroon
						]{hyperref} 

\usepackage{multirow}

\usepackage{xypic}
\usepackage{mathtools}
\usepackage{xspace}

\usepackage[a4paper,top=3.2cm,bottom=2.7cm,left=3cm,right=3cm, bindingoffset=0mm]{geometry}
\usepackage{booktabs}

\usepackage{tikz}

\usetikzlibrary{calc, positioning, shapes.geometric, patterns, cd}

\newlength{\hatchspread}
\newlength{\hatchthickness}
\newlength{\hatchshift}
\newcommand{\hatchcolor}{}
\tikzset{hatchspread/.code={\setlength{\hatchspread}{#1}},
         hatchthickness/.code={\setlength{\hatchthickness}{#1}},
         hatchshift/.code={\setlength{\hatchshift}{#1}},
         hatchcolor/.code={\renewcommand{\hatchcolor}{#1}}}
\tikzset{hatchspread=3pt,
         hatchthickness=0.4pt,
         hatchshift=0pt,
         hatchcolor=black}
\pgfdeclarepatternformonly[\hatchspread,\hatchthickness,\hatchshift,\hatchcolor]
   {custom north west lines}
   {\pgfqpoint{\dimexpr-2\hatchthickness}{\dimexpr-2\hatchthickness}}
   {\pgfqpoint{\dimexpr\hatchspread+2\hatchthickness}{\dimexpr\hatchspread+2\hatchthickness}}
   {\pgfqpoint{\dimexpr\hatchspread}{\dimexpr\hatchspread}}
   {
    \pgfsetlinewidth{\hatchthickness}
    \pgfpathmoveto{\pgfqpoint{0pt}{\dimexpr\hatchspread+\hatchshift}}
    \pgfpathlineto{\pgfqpoint{\dimexpr\hatchspread+0.15pt+\hatchshift}{-0.15pt}}
    \ifdim \hatchshift > 0pt
      \pgfpathmoveto{\pgfqpoint{0pt}{\hatchshift}}
      \pgfpathlineto{\pgfqpoint{\dimexpr0.15pt+\hatchshift}{-0.15pt}}
    \fi
    \pgfsetstrokecolor{\hatchcolor}
    \pgfusepath{stroke}
   }

\pgfdeclarepatternformonly[\hatchspread,\hatchthickness,\hatchshift,\hatchcolor]
   {custom north east lines}
   {\pgfqpoint{\dimexpr-2\hatchthickness}{\dimexpr-2\hatchthickness}}
   {\pgfqpoint{\dimexpr\hatchspread+2\hatchthickness}{\dimexpr\hatchspread+2\hatchthickness}}
   {\pgfqpoint{\dimexpr\hatchspread}{\dimexpr\hatchspread}}
   {
    \pgfsetlinewidth{\hatchthickness}
    \pgfpathmoveto{\pgfqpoint{\dimexpr\hatchshift-0.15pt}{-0.15pt}}
    \pgfpathlineto{\pgfqpoint{\dimexpr\hatchspread+0.15pt}{\dimexpr\hatchspread-\hatchshift+0.15pt}}
    \ifdim \hatchshift > 0pt
      \pgfpathmoveto{\pgfqpoint{-0.15pt}{\dimexpr\hatchspread-\hatchshift-0.15pt}}
      \pgfpathlineto{\pgfqpoint{\dimexpr\hatchshift+0.15pt}{\dimexpr\hatchspread+0.15pt}}
    \fi
    \pgfsetstrokecolor{\hatchcolor}
    \pgfusepath{stroke}
   }

\newcommand{\Alpha}{\mathrm{A}}

\newcommand{\Kappa}{\mathrm{K}}

\newcommand{\rep}[1]{\hat #1}	

\usepackage{xparse}

\ExplSyntaxOn
\NewDocumentCommand{\makeabbrev}{mmm}
 {
  \yoruk_makeabbrev:nnn { #1 } { #2 } { #3 }
 }

\cs_new_protected:Npn \yoruk_makeabbrev:nnn #1 #2 #3
 {
  \clist_map_inline:nn { #3 }
   {
    \cs_new_protected:cpn { #2 } { #1 { ##1 } }
   }
 }
 \ExplSyntaxOff

\makeabbrev{\textbf}{tbf#1}{a,b,c,d,e,f,g,h,i,j,k,l,m,n,o,p,q,r,s,t,u,v,w,x,y,z,A,B,C,D,E,F,G,H,I,J,K,L,M,N,O,P,Q,R,S,T,U,V,W,X,Y,Z}

\makeabbrev{\textbf}{bf#1}{a,b,c,d,e,f,g,h,i,j,k,l,m,n,o,p,q,r,s,t,u,v,w,x,y,z,A,B,C,D,E,F,G,H,I,J,K,L,M,N,O,P,Q,R,S,T,U,V,W,X,Y,Z}

\makeabbrev{\textsf}{tsf#1}{a,b,c,d,e,f,g,h,i,j,k,l,m,n,o,p,q,r,s,t,u,v,w,x,y,z,A,B,C,D,E,F,G,H,I,J,K,L,M,N,O,P,Q,R,S,T,U,V,W,X,Y,Z}

\makeabbrev{\mathsf}{mss#1}{a,b,c,d,e,f,g,h,i,j,k,l,m,n,o,p,q,r,s,t,u,v,w,x,y,z,A,B,C,D,E,F,G,H,I,J,K,L,M,N,O,P,Q,R,S,T,U,V,W,X,Y,Z}

\makeabbrev{\mathfrak}{mf#1}{a,b,c,d,e,f,g,h,i,j,k,l,m,n,o,p,q,r,s,t,u,v,w,x,y,z,A,B,C,D,E,F,G,H,I,J,K,L,M,N,O,P,Q,R,S,T,U,V,W,X,Y,Z}

\makeabbrev{\mathrm}{mrm#1}{a,b,c,d,e,f,g,h,i,j,k,l,m,n,o,p,q,r,s,t,u,v,w,x,y,z,A,B,C,D,E,F,G,H,I,J,K,L,M,N,O,P,Q,R,S,T,U,V,W,X,Y,Z}

\makeabbrev{\mathbf}{mbf#1}{a,b,c,d,e,f,g,h,i,j,k,l,m,n,o,p,q,r,s,t,u,v,w,x,y,z,A,B,C,D,E,F,G,H,I,J,K,L,M,N,O,P,Q,R,S,T,U,V,W,X,Y,Z}

\makeabbrev{\mathcal}{mc#1}{A,B,C,D,E,F,G,H,I,J,K,L,M,N,O,P,Q,R,S,T,U,V,W,X,Y,Z}

\makeabbrev{\mathbb}{mbb#1}{A,B,C,D,E,F,G,H,I,J,K,L,M,N,O,P,Q,R,S,T,U,V,W,X,Y,Z}

\makeabbrev{\mathscr}{ms#1}{A,B,C,D,E,F,G,H,I,J,K,L,M,N,O,P,Q,R,S,T,U,V,W,X,Y,Z}

\makeabbrev{\mathrm}{#1}{
Id,id,ran,rk,diag,stab,ann,conv,pr,ev,tr,End,Hom,sgn,im,op,can,fin,ext,red,tot,
rot,usc,lsc,Lip,LocLip,lip,bSymLip,osc,AC,loc,uloc,spec,coz,z,
supp,Opt,Adm,Cpl,Geo,GeoSel,GeoOpt,GeoAdm,GeoCpl,reg,
bd,co,Ric,Exp,dExp,dist,seg,Seg,cut,fcut,Cut,SDiff,Iso,Isom,diam,cl,Homeo,Diff,Der,vol,dvol,inj,relint, Graph, sub,codim,
var,law,Var,Poi,Gam,pa,,iso,fs,inv,pqi,mix,
TestF,
}

\makeabbrev{\mathsf}{#1}{DP,CD,BE,MCP,Ent,wMTW,MTW,RCD,ncRCD,QCD,EVI,Irr,IH,SC,wFe,VA,UP,Curv,Alex,CAT}

\newcommand{\bLip}{\mathrm{Lip}_b}

\newcommand{\T}{\tau} 

\renewcommand{\complement}{\mathrm{c}}

\newcommand{\emparg}{{\,\cdot\,}}

\newcommand{\Ch}{\mathsf{Ch}}

\newcommand{\dom}[1]{\mathcal D(#1)}
\newcommand{\sem}[1]{\{#1\}_{t \ge 0}}

\newcommand{\comm}{\,\,\mathrm{,}\;\,}

\DeclareMathOperator{\eqdef}{\coloneqq}

\let\epsilon\varepsilon

\let\temp\phi
\let\phi\varphi
\let\varphi\temp

\newcommand{\longrar}{\longrightarrow}
\newcommand{\rar}{\rightarrow}

\newcommand{\diff}{\mathop{}\!\mathrm{d}}

\newcommand{\tabs}[1]{\big\lvert#1\big\rvert}	

\newcommand{\norm}[1]{\left\lVert#1\right\rVert}					
\newcommand{\set}[1]{\left\{#1\right\}}							
\newcommand{\paren}[1]{\left(#1\right)}							
\newcommand{\tparen}[1]{\big({#1}\big)}

\newcommand{\ttonde}[1]{\big({#1}\big)}
\newcommand{\quadre}[1]{\left[#1\right]}							
\newcommand{\class}[2][]{\left[#2\right]_{#1}}						
\newcommand{\tclass}[2][]{\big [#2\big]_{#1}}						

\newcommand{\tym}[1]{{\scriptscriptstyle{\times #1}}}
\newcommand{\otym}[1]{{\scriptscriptstyle{\otimes #1}}}
\newcommand{\osym}[1]{{\scriptscriptstyle{\odot #1}}}

\DeclareSymbolFont{symbolsC}{U}{pxsyc}{m}{n}
\SetSymbolFont{symbolsC}{bold}{U}{pxsyc}{bx}{n}
\DeclareFontSubstitution{U}{pxsyc}{m}{n}
\DeclareMathSymbol{\medcirc}{\mathbin}{symbolsC}{7}
\DeclareSymbolFont{symbolsZ}{OMS}{pxsy}{m}{n}
\SetSymbolFont{symbolsZ}{bold}{OMS}{pxsy}{bx}{n}
\DeclareFontSubstitution{OMS}{pxsy}{m}{n}

\newcommand{\seq}[1]{\paren{#1}}								

\newcommand{\Meas}{\mathscr M}

\newcommand{\pfwd}{\sharp}
\DeclareMathOperator*{\esssup}{ess\text-sup}

\DeclareMathOperator{\car}{\mathbf 1}

\DeclareMathOperator{\emp}{\varnothing} 
\newcommand{\N}{{\mathbb N}}
\newcommand{\EN}{\overline{\N}}
\newcommand{\R}{{\mathbb R}}

\DeclareMathOperator{\Z}{{\mathbb Z}}

\newcommand{\mrestr}[1]{\!\downharpoonright_{#1}}

\allowdisplaybreaks

\usetikzlibrary{shapes.misc}
\tikzset{cross/.style={cross out, draw=black, minimum size=2*(#1-\pgflinewidth), inner sep=0pt, outer sep=0pt},
cross/.default={4pt}}

\newcommand{\iref}[1]{\ref{#1}}

\newcommand{\comma}{\,\,\mathrm{,}\;\,}

\newcommand{\fstop}{\,\,\mathrm{.}}

\newcommand{\cdc}{\Gamma}

\usepackage{scrextend}						

\newcommand{\cpl}{q}
\newcommand{\QP}{{\mu}}

\newcommand{\e}{\varepsilon}
\newcommand{\dUpsilon}{{\mathbf \Upsilon}}

\newcommand{\U}{\dUpsilon}
\newcommand{\sine}{\mathsf{sine}}

\newcommand{\E}{\mathcal E}

\renewcommand{\1}{\mathbf 1}

\numberwithin{equation}{section}
\theoremstyle{plain}
\newtheorem{thm}{Theorem}[section]
\newtheorem*{thm*}{Theorem}

\newtheorem*{mthm*}{Main Theorem}

\newtheorem{prop}[thm]{Proposition}
\newtheorem{lem}[thm]{Lemma}
\newtheorem{cor}[thm]{Corollary}
\newtheorem*{cor*}{Corollary}

\theoremstyle{remark}
\newtheorem{defs}[thm]{Definition}
\newtheorem*{defs*}{Definition}

\theoremstyle{remark}

\newtheorem{rem}[thm]{Remark}
\newtheorem{ese}[thm]{Example}
\newtheorem{ass}[thm]{Assumption}

\newcommand{\proj}{{\sf proj}}
\renewcommand{\paragraph}[1]{\medskip\emph{#1}.}

\newcommand{\quot}{{\sf P}}
\newcommand{\CBE}{\mathrm{C}\beta\mathrm{E}}
\newcommand{\cquad}{\comma\quad}
\newcommand{\Fis}{\mathsf{F}}
\newcommand{\scolon}{\ ;}

\makeatletter
\@namedef{subjclassname@2020}{%
  \textup{2020} Mathematics Subject Classification}
\makeatother

\begin{document}
\title[Curvature bound of Dyson Brownian Motion]{Curvature Bound of Dyson Brownian Motion}

\author[K.~Suzuki]{Kohei Suzuki}
\address{Department of Mathematical Science, Durham University, Science Laboratories, South Road, DH1 3LE, United Kingdom}
\address{Theoretical Sciences Visiting Program, Okinawa Institute of Science and Technology Graduate University, Okinawa, 904-0495, Japan}
\thanks{\hspace{-5.5mm} Department of Mathematical Science, Durham University, United Kingdom / Okinawa Institute of Science and Technology Graduate University, Japan
\\
\hspace{2.0mm} E-mail: kohei.suzuki@durham.ac.uk
}

\keywords{\vspace{2mm} Dyson Brownian motion, log-gas, Ricci curvature bound}


\begin{abstract}
We construct a strongly local symmetric Dirichlet form on the configuration space~$\U$ whose symmetrising (thus also invariant) measure is $\sine_\beta$, which is the law of the sine $\beta$ ensemble for every $\beta>0$. 
For every $\beta>0$, this Dirichlet form satisfies the Bakry--\'Emery gradient estimate~$\BE(K, \infty)$ with $K=0$. This implies various functional inequalities, including the local Poincar\'e inequality, the local log--Sobolev inequality and the local hyper-contractivity. We then introduce an $L^2$-transportation-type extended distance $\bar\mssd_\U$ on $\U$, and  prove the dimension-free Harnack inequality and several Lipschitz regularisation estimates of the $L^2$-semigroup associated with the Dirichlet form in terms of $\bar\mssd_{\U}$.  As a result of $\BE(0,\infty)$, we obtain that  the dual semigroup on the space of probability measures over $\U$, endowed with a Benamou--Brenier-like extended distance $\mssW_{\E}$, satisfies the evolutional variation inequality with respect to the Bolzmann--Shannon entropy $\Ent_{\sine_\beta}$ associated with~$\sine_\beta$. 
Furthermore,  the dual semigroup is characterised as the unique $\mssW_{\E}$-gradient flow in the space of probability measures with respect to $\Ent_{\sine_\beta}$. These results provide quantitative estimates of the transition semigroup of the unlabelled infinite Dyson Brownian motion (DBM) with the inverse temperature $\beta$, and give a new perspective regarding the DBM as the $\mssW_{\E}$-gradient flow of the Bolzmann--Shannon entropy.
Finally, we provide a sufficient condition for $\BE(K, \infty)$ beyond $\sine_\beta$ and apply it to the infinite particle diffusion whose symmetrising measure is the law of the $1$-dimensional $(\beta,s)$-circular Riesz gas with $\beta>0$ and $0<s<1$. 
\end{abstract}

\maketitle

%

\setcounter{tocdepth}{1}
\makeatletter
\def\l@subsection{\@tocline{2}{0pt}{2.5pc}{5pc}{}}
\def\l@subsubsection{\@tocline{3}{0pt}{4.75pc}{5pc}{}}
\makeatother

\tableofcontents

\section{Introduction} \label{s:Int}

\paragraph{Infinite Dyson Brownian motion}
The interacting particle system mainly studied in this article can be formally described as the following  stochastic differential equation of infinitely many particles in $\R$:
\begin{align}  \label{d:DBMS}
\diff X_t^k=    \frac{\beta}{2} \lim_{r \to \infty}\sum_{\substack{i: i \neq k\\ |X_t^k- X_t^i|<r}} \frac{1}{X_t^k- X_t^i} \diff t + \diff B^k_t \cquad k \in \N \comma 
\end{align}
 where $(B_t^k: t\ge 0, k \in \N)$ is the family of infinitely many independent Brownian motions on~$\R$. The solution $\mathbb X_t = (X_t)_{k \in \N}$ to \eqref{d:DBMS} is called {\it infinite Dyson Brownian motion with inverse temperature $\beta>0$}, which is of  particular importance in relation to the random matrix theory. Over the last thirty-five years, the existence and the uniqueness of strong and weak solutions to~\eqref{d:DBMS} have been intensively studied, e.g., in~\cite{Dys62, Spo87, NagFor98, KatTan10, Osa96, Osa12, Osa13, Tsa16, OsaTan20, KawOsaTan22}.  In particular, the existence and the pathwise uniqueness of the strong solution to~\eqref{d:DBMS} have been proven with a suitable choice of initial conditions in~\cite{OsaTan20, KawOsaTan22} for $\beta=1,2,4$ and  in~\cite{Tsa16} for $\beta \ge 1$.
By mapping the solution~$\mathbb X_t$ via $(x_i)_{i \in \N} \mapsto \sum_{i=1}^\infty \delta_{x_i}$, it can be thought of as a diffusion process (i.e., a continuous-time strong Markov process with continuous trajectories) on {\it the configuration spcae~$\U=\U(\R)$} over $\R$ (i.e., the space of locally finite point measures on~$\R$) {endowed with the vague topology~$\tau_\mrmv$ (i.e., the topology induced by the duality of compactly supported continuous functions in $\R$).
 This diffusion process on~$\U$ is called {\it unlabelled solution} to~\eqref{d:DBMS} and denoted by $\mathsf X_t$. For~$\beta=1,2,4$, the solution~$\mathsf X_t$ has been identified with the diffusion process  associated with a particular Dirichlet form whose symmetrising measure~$\QP$ is the law of the sine $\beta$ ensemble, see~\cite[Thm.~24]{Osa12} and \cite[\S8]{Tsa16}.
 }

 \smallskip

\paragraph{Sine $\beta$ ensemble} 
Let $\beta>0$ and $\CBE_k$ be the law of {\it the circular $\beta$ ensemble} of~$k$-particle, which is a probability measure defined as 
$$\frac{1}{Z_{k, \beta}}\prod_{1 \le j<l \le k}\bigl| e^{i\theta_j}- e^{i\theta_l}\bigr|^\beta\diff \theta_1 \cdots\diff \theta_k\cquad Z_{k, \beta}=(2\pi)^k\frac{\Gamma\Bigl(\frac{\beta}{2}k+1\Bigr)}{\Gamma\Bigl(\frac{\beta}{2}+1\Bigr)^k}\comma $$
where $\diff \theta_i$ is the Lebesgue measure on $[-\pi, \pi]$. Let $\mathbb P_{k, \beta}$ be the push-forward measure of~$\CBE_k$ by the composition of the scaling map~$(\theta_1,\ldots, \theta_k) \to (\frac{k\theta_1}{2\pi}, \ldots, \frac{k\theta_k}{2\pi})$ and the symmetric quotient~$(\theta_1,\ldots, \theta_k) \mapsto \sum_{i=1}^k \delta_{\theta_i}$, which is a probability measure on $\U^k\bigl([-\frac{k}{2}, \frac{k}{2}]\bigr)$.
According to \cite[Dfn.\ 1.6]{KilSto09}, the law $\CBE$ of the {\it circular $\beta$ ensemble}  is defined as the weak limit~$\mathbb P_{\beta}$ of $\mathbb P_{k, \beta}$ with $k \to \infty$, which is a Borel probability measure on $\U(\R)$.
In \cite{ValVir09}, the Borel probability measure $\sine_\beta$ on $\U(\R)$ called {\it the law of the sine $\beta$ ensemble} (or {\it the sine $\beta$ point process}) was constructed by the limit of the laws of  the scaled Gaussian $\beta$-ensembles. These two measures $\mathbb P_\beta$ and $\sine_\beta$ are identical for every~$\beta>0$, see~\cite{Nak14}. When $\beta=1, 2,4$, $\sine_\beta$ was constructed and studied as determinantal or Pfaffian point processes before these works, see, e.g., \cite{Meh04}. In the rest of the article, we use the notation~$\sine_\beta$ instead of $\mathbb P_{\beta}$.

\paragraph{Bakry--\'Emery curvature bound}
In the seminal paper~\cite{BakEme85}, it was discovered that a complete Riemannian manifold~$(M, g)$ has a Ricci curvature lower bound by a constant  $K \in \R$, i.e., 
\begin{align}\label{e:RLB}
{\rm Ric}_x(v, v) \ge Kg_x(v, v) \cquad x \in M\cquad v \in T_xM 
\end{align} 
 if and only if {the following {\it $\cdc_2$-criterion} holds: for every compactly supported smooth function $u \in C_0^\infty(M)$
$$\cdc_2(u) \ge K\cdc(u) \comma$$
where $\cdc(u,v):=\langle \nabla u, \nabla v\rangle$ is the square gradient operator and $\cdc_2(u, v):=\frac{1}{2}(\Delta \cdc(u, v)- \cdc(\Delta u, v) - \cdc(u, \Delta v))$ is what is called {\it the $\cdc_2$-operator} with the Laplace--Beltrami operator~$\Delta$. 
Due to the existence of good test functions in the domain of~$\Delta$ supporting the essential self-adjointness in this case, the {\it $\cdc_2$-criterion} is equivalent to the following gradient estimate: 
\begin{align}\label{D:BE}|\nabla T_tu|^2 \le e^{-2Kt}T_t|\nabla u|^2 \tag*{$\BE(K,\infty)$} \comma \quad u \in W^{1,2}(M) \comma
\end{align}
where $\{T_t\}_{t \ge 0}$ is the heat semigroup and $W^{1,2}(M)$ is the $(1,2)$-Sobolev space~on~$M$, see e.g., \cite[Cor.~3.3.19]{BakGenLed14}.  We refer to this formula as {\it $\BE(K,\infty)$ gradient estimate}, or {\it $\BE(K,\infty)$ curvature-dimension condition} in this paper.}
The Bakry--\'Emery~gradient estimate~$\BE(K,\infty)$  has rich applications to functional inequalities such as the (local) Poincar\'e inequalities, the (local) log-Sobolev inequality, the (local) hyper-contractivity and many others.  Furthermore, this discovery opened a way to generalise the concept of lower Ricci curvature bound to singular spaces beyond manifolds such as metric measure spaces and infinite-dimensional spaces since the formulation~$\BE(K,\infty)$ requires only a weak (Sobolev) differentiable structure, which does not require {\it Ricci curvature tensors} nor a $C^2$-structure. This concept particularly fits  the framework of Dirichlet forms, where the square gradient operator is replaced by what is called {\it square field operator} (or {\it carr\'e du champ}), the heat semigroup is replaced by the $L^2$-semigroups associated with the Dirichlet form, and the $(1,2)$-Sobolev space is replaced by the domain of the Dirichlet form. 
 We refer the readers to, e.g., \cite{BakGenLed14} and \cite{Vil09} for  comprehensive references. 

\paragraph{Main results} In this paper, we construct a strongly local symmetric  Dirichlet form (in Dfn.~\ref{d:DFF})
\begin{align*}
\E^{\U, \QP}(u)=\frac{1}{2}\int_{\U}\cdc^{\U}(u) \diff \QP \comma \quad  u \in \dom{\E^{\U, \QP}} 
\end{align*}
with the square field $\cdc^{\U}$ and the symmetrising measure $\QP=\sine_\beta$ for arbitrary $\beta>0$ 
 such that  $(\E^{\U, \QP}, \dom{\E^{\U, \QP}})$ satisfies $\BE(0,\infty)$. We note that  the $\cdc_2$-criterion is not available in this case because, due to the long-range correlation of $\sine_\beta$,  there is no known space of test functions for $(\E^{\U, \QP}, \dom{\E^{\U, \QP}})$ on which the corresponding $L^2$-infinitesimal generator has a concrete expression. 
 We also remark that the quasi-invariance of $\QP$ is unknown, i.e., we do not know whether the push-forward measure of $\QP$ by the shifts induced by any compactly supported smooth vector fields in $\R$ is equivalent to $\QP$. Thus, the standard integration by parts argument does not work to construct Dirichlet forms  in this case. 
\begin{thm}[Thm.~\ref{t: main}] \label{t:intromain}
Let $\beta>0$ and $\QP=\sine_\beta$.  The Dirichlet form $(\E^{\U, \QP}, \dom{\E^{\U, \QP}})$ constructed in Dfn.~\ref{d:DFF} satisfies the Bakry--\'Emery gradient estimate $\BE(0,\infty)$. Namely,  for the $L^2$-semigroup $\{T_t^{\U, \QP}\}_{t \ge 0}$ associated with $(\E^{\U, \QP}, \dom{\E^{\U, \QP}})$, 
\begin{align*}
\cdc^{\U}\bigl(T_t^{\U, \mu} u\bigr) \le T_t^{\U, \mu} \cdc^{\U}(u) \comma \quad u \in \dom{\E^{\U, \QP}}\quad t  \ge 0 \fstop
\end{align*} 
{In the case $\beta=2$, the curvature lower bound $K=0$ is optimal. }
\end{thm}
Thm.~\ref{t:intromain} states that the configuration space $\U$ endowed with the Dirichlet form structure~$(\E^{\U, \QP}, \dom{\E^{\U, \QP}})$ can be seen as a {\it non-negatively curved space} in the sense of Bakry--\'Emery.
{The Dirichlet form~$(\E^{\U, \QP}, \dom{\E^{\U, \QP}})$ constructed in this paper is {\it a priori} different from  those constructed in~\cite{Osa96, Osa12} for $\beta=1,2,4$, but they are {\it a posteriori} identified, see Rem.~\ref{r:IDI2}. Therefore, the $L^2$-semigroup~$\{T_t^{\U, \QP}\}_{t \ge 0}$ coincides with the transition probability of the unlabelled solution~$\mathsf X_t$ to \eqref{d:DBMS} after excluding a set of capacity zero from $\U$.  In particular, for $\beta=1, 2, 4$, the form $(\E^{\U, \QP}, \dom{\E^{\U, \QP}})$ is quasi-regular (see~\S\ref{ss:DF} for the definition of the quasi-regularity).  
Furthermore, by taking a particular smaller domain $\mathcal F \subset \dom{\E^{\U, \QP}}$, the quasi-regularity also holds true for every $\beta>0$, see Cor.~\ref{c:QRB}. Hence, by, e.g., \cite[Thm.~3.5 p.103]{MaRoe90},  there exists an associated diffusion process on~$\U$ for every $\beta>0$ (cf.~$\beta \ge 1$ in~\cite{Tsa16}) whose transition semigroup is given by the semigroup associated with $(\E^{\U, \QP}, \mathcal F)$.}

\paragraph{Functional inequalities}
The Bakry--\'Emery gradient estimate $\BE(0,\infty)$ provides various functional inequalities regarding quantitative estimates of the semigroup.  We start with the local Poincar\'e inequality, which, in a sense,  provides a {\it local} spectral gap estimate for the corresponding particle dynamics~\eqref{d:DBMS}, see Rem.~\ref{r:LSG}. 
 \begin{cor}[Cor.~\ref{t:LPS}] \label{c:1}
Let $\beta>0$ and $\QP=\sine_\beta$.  Then, the local Poincar\'e inequality holds: for every~$u \in \dom{\E^{\U, \QP}}$ and $t \ge 0$,
\begin{align*}
&T^{\U, \QP}_tu^2- (T^{\U, \QP}_tu)^2 \le 2tT^{\U, \QP}_t\cdc^{\U}(u)  \comma 
\\
&T^{\U, \QP}_tu^2- (T^{\U, \QP}_tu)^2 \ge 2t\cdc^{\U} (T^{\U, \QP}_tu) \fstop
\end{align*}
\end{cor}
Suppose in addition that the form $(\E^{\U, \QP}, \dom{\E^{\U, \QP}})$ is quasi-regular (it is known, e.g., for $\beta=1,2,4$, see Rem.~\ref{r:IDI2}). 
Then, we have the following functional inequalities.
\begin{cor}[Cor.~\ref{p:PBE}, \ref{c:LLSI}, \ref{c:LHC}] \label{c:3-1}
Let $\beta>0$ and $\QP=\sine_\beta$.  Suppose that the form $(\E^{\U, \QP}, \dom{\E^{\U, \QP}})$ is quasi-regular.
Then, the following hold:
\begin{enumerate}[(a)]  
\item $(${\bf$p$-Bakry-\'Emery estimate}$)$
 The form $(\E^{\U, \QP}, \dom{\E^{\U, \QP}})$ satisfies $\BE_p(K,\infty)$ for every $1 \le p <\infty$:
 $$\cdc^{\U}(T_t^{\U, \QP}u)^{\frac{p}{2}} \le T_t^{\U, \QP}\bigl(\cdc^{\U}(u)^{\frac{p}{2}}\bigr) \cquad u \in \dom{\E^{\U, \QP}} \quad t \ge 0 \fstop$$
 \item $(${\bf local log-Sobolev inequality}$)$
 For every positive $u \in \dom{\E^{\U, \QP}}$ and $t \ge 0$,
\begin{align*}
&T^{\U, \QP}_t(u\log u)- T^{\U, \QP}_tu\log T^{\U, \QP}_t u \le t T^{\U, \QP}_t\biggl( \frac{\cdc^{\U}(u)}{u} \biggr) \comma
\\
&T^{\U, \QP}_t(u\log u)- T^{\U, \QP}_tu\log T^{\U, \QP}_t u \ge t \frac{\cdc^{\U}(T^{\U, \QP}_t u)}{T^{\U, \QP}_t u}  \fstop
\end{align*}
 \item $(${\bf local hyper-contractivity}$)$
 For every $t>0$, $0<s \le t$, and $1 < p <q<\infty$ so that 
 $$\frac{q-1}{p-1}=\frac{t}{s} \comma$$
 it holds that 
 $$\Bigl(T_s^{\U, \QP}(T_{t-s}^{\U, \QP}u)^q\Bigr)^{1/q} \le \Bigl(T_t^{\U, \QP}u^p\Bigr)^{1/p} \cquad u \ge 0 \fstop$$
 \end{enumerate}
\end{cor}

\smallskip

\paragraph{Extended distances in $\U$} A distance function that allows to take $+\infty$ is called {\it extended distance}. 
In this paper, we study two {\it extended} distances $\mssd_\U$ and $\bar\mssd_\U$ on $\U$, both of which stem from the optimal transport theory.
The {\it $L^2$-transportation extended distance} $\mssd_{\dUpsilon}$ is defined as
\begin{align*}
\mssd_{\dUpsilon}(\gamma,\eta)\eqdef \inf_{\cpl\in\Cpl(\gamma,\eta)} \paren{\int_{\R^{2}} \mssd^2(x,y) \diff\cpl(x,y)}^{1/2}\comma \qquad \inf{\emp}=+\infty \comma
\end{align*}
where $\mssd(x,y)=|x-y|$ is the standard Euclidean distance in~$\R$, and $\Cpl(\gamma,\eta)$ denotes the set of Radon measures on $\R^{2}$ whose first (resp.~second) marginal is $\gamma$ (resp.~$\eta$).
As a variant of $\mssd_\U$, we introduce the \emph{$L^2$-transportation-type}  \emph{extended distance}~$\bar{\mssd}_\U$, defined as 
\begin{align*}
\bar{\mssd}_\U(\gamma, \eta):=
\begin{cases}
\mssd_\U(\gamma, \eta) \quad &\text{if $\gamma_{B_r^c}=\eta_{B_r^c}$ for some $r>0$\ ,}
\\
+\infty \quad  &\text{otherwise} \comma
\end{cases}
\end{align*}
where  $\gamma_{B_r^c}:=\gamma\mrestr{B_r^c}$ denotes the configuration $\gamma$ restricted (as a measure) on the complement $B_r^c:=\R \setminus B_r$ of the interval $B_r:=[-r, r]$. 

As one can see from the definitions above, the function~$\mssd_\U$ as well as $\bar\mssd_\U$ could take $+\infty$ very often from the measure-theoretic viewpoint. Indeed, every metric ball is a set of measure zero with respect to $\QP$, which is similar to the Cameron--Martin distance for the Wiener space, see (d) Rem.~\ref{r:ped}.
However, if we see the distance $\gamma \mapsto \mssd_\U(\gamma, \Lambda)=\inf_{\eta \in \Lambda}\mssd_\U(\gamma, \eta)$ (similarly for $\bar\mssd_\U$) from a set $\Lambda \subset \U$, it recovers the finiteness and provides a non-trivial Lipschitz function. See Example~\ref{r:LT} for examples and counterexamples for Lipschitz functions with respect to $\mssd_\U$ and $\bar\mssd_\U$, where, interestingly, cylinder functions are not necessarily Lipschitz functions. 
We note that the function~$\bar\mssd_\U: \U^{\times 2} \to \R_+\cup \{+\infty\}$ is Borel measurable but not continuous nor lower semi-continuous with respect to the product vague topology~$\tau_\mrmv^{\times 2}$, see Rem.~\ref{r:ped}. Hence, the space~$\Lip_b(\U, \bar\mssd_\U)$ of bounded Lipschitz functions $u: \U \to \R$ with respect to~$\bar\mssd_\U$ does not necessarily consist of $\tau_\mrmv$-continuous functions nor even measurable functions with respect to the~$\QP$-completion~$\mathscr B(\tau_\mrmv)^\QP$  of the Borel $\sigma$-algebra~$\mathscr B(\tau_\mrmv)$. See~\cite[Example~3.4]{LzDSSuz21} for the existence of non-measurable Lipschitz functions. We, therefore, denote by $\Lip(\U, \bar\mssd_\U, \QP)$ the subspace of $\Lip(\U, \bar\mssd_\U)$ whose elements are~$\mathscr B(\tau_\mrmv)^\QP$-measurable. 

\paragraph{Lipschitz structure vs Dirichlet forms}
In the following, we relate the Lipschitz constant~$\Lip_{\bar\mssd_\U}(u)$ with respect to $\bar\mssd_\U$ and the square field $\cdc^{\U}(u)$ for $u \in \Lip_b(\U, \bar\mssd_\U, \QP)$.
\begin{thm}[Prop.~\ref{p:DF}] \label{t:1.5}
Let $\beta>0$ and $\QP=\sine_\beta$.  Then, the Rademacher-type property holds:
$$\Lip_b(\U, \bar\mssd_\U, \QP) \subset \dom{\E^{\U, \QP}} \comma \quad \cdc^{\U, \QP}(u) \le \Lip_{\bar\mssd_\U}(u)^2 \fstop$$
 \end{thm}

Suppose that the form $(\E^{\U, \QP}, \dom{\E^{\U, \QP}})$ is quasi-regular (e.g., $\beta=1,2,4$, see Rem.~\ref{r:IDI2}). Then, there exists a $\QP$-symmetric diffusion process~$\{(X_t, \mathbb P_\gamma): t\ge 0, \ \gamma \in \U\}$ so that 
$T_t^{\U, \QP}u(\gamma) = \mathbb E_\gamma[u(X_t)]$ for every $t \ge 0$, every bounded Borel function $u$ and quasi-every~$\gamma$ (i.e., the equality holds after excluding a set of capacity zero associated with the Dirichlet form, see \S\ref{ss:DF}).
  In particular, there exists a transition probability measure $P^{\U, \QP}_t(\gamma, \diff\eta)$ satisfying
\begin{align} \label{e:FE}
T_t^{\U, \QP}u(\gamma) = \int_{\U} u(\eta) P^{\U, \QP}_t(\gamma, \diff\eta) \qquad \text{quasi-every~$\gamma$} \fstop
\end{align}
Combining Thm.~\ref{t:1.5} with the local Poincar\'e inequality in Cor.~\ref{c:1}, we have the following exponential decay estimate of the transition probability $P_t^{\U, \QP}(\gamma, \diff \eta)$ in terms of $1$-Lipschitz functions with respect to~$\bar\mssd_\U$.
\begin{cor}[Cor.~\ref{c:TES}, exponential integrability]\label{c:3}
Let $\beta>0$ and $\QP =\sine_\beta$. Suppose that the form $(\E^{\U, \QP}, \dom{\E^{\U, \QP}})$ is quasi-regular.
If $u$ is a $\bar{\mssd}_\U$-Lipschitz $\QP$-measurable function with $\Lip_{\bar{\mssd}_\U}(u) \le 1$ and $|u(\gamma)|<+\infty$ $\QP$-a.e.~$\gamma$, then
$$\int_{\U} e^{s u(\eta)} P^{\U, \QP}_t(\gamma, \diff\eta)<+\infty \quad \text{$\QP$-a.e.~} \quad s<\sqrt{2/t} \fstop$$
\end{cor}

\paragraph{Curvature bound in terms of the metric $\bar\mssd_\U$}
In the case of Riemannian manifolds~$(M,g)$, the Ricci curvature lower bound~\eqref{e:RLB} is known to be equivalent to the dimension-free Harnack inequality (\cite[Thm.~2.3.3]{Wan14}): for $\alpha>1$ and every bounded Borel function $u \ge 0$ on $M$
\begin{align*}
(T_tu)^\alpha(x)\le T_tu^\alpha(y) \exp\Bigl\{ \frac{\alpha K}{2(\alpha-1)(1-e^{-2Kt})}\mssd_g(x, y)^2\Bigr\} \comma 
\end{align*}
where $\mssd_g$ is the geodesic distance induced by $g$. This provides a characterisation of \eqref{e:RLB}  in terms of the distance~$\mssd_g$ and the heat semigroup~$T_t$. 
In the following theorem, we prove the dimension-free Harnack inequality with $K=0$ in terms of~$\bar\mssd_\U$ and~$T_t^{\U, \QP}$. Furthermore, we prove the log-Harnack inequality, the Lipschitz contraction estimate and  the Lipschitz regularisation property. For a $\QP$-class $u$ of functions, we say that a function $v: \U \to \R$ is a {\it $\QP$-modification of $u$} if $u=v$ $\QP$-almost everywhere. 
\begin{thm}[Thm.~\ref{t:DFH}] \label{t:3}
Let $\beta>0$ and $\QP=\sine_\beta$.  Then, the following hold:
\begin{enumerate}[(a)]
\item Wang's dimension-free Harnack inequality: for every non-negative $u \in L^\infty(\U, \QP)$, $t>0$ and $\alpha>1$, there exists $\Omega \subset \U$ so that $\QP(\Omega)=1$ and 
$$(T^{\U, \QP}_tu)^\alpha(\gamma)\le T^{\U, \QP}_tu^\alpha(\eta) \exp\Bigl\{ \frac{\alpha }{4(\alpha-1){t}}\bar{\mssd}_\U(\gamma, \eta)^2\Bigr\} \comma \quad \text{$\gamma, \eta \in \Omega$} \ ;$$

\item Log-Harnack inequality: for any non-negative $u \in L^\infty(\U, \QP)$, $\e \in (0, 1]$ and $t>0$, there exists $\Omega \subset \U$ so that $\QP(\Omega)=1$ and 
$$T^{\U, \QP}_t\log (u+\e)(\gamma) \le \log (T^{\U, \QP}_tu(\eta) +\e)+ \frac{\bar{\mssd}_\U(\gamma, \eta)^2}{{4t}}\comma \quad \text{$\gamma, \eta \in \Omega$} \ ;$$
\item  Lipschitz contraction: for every $u \in \Lip_b(\U, \bar{\mssd}_\U, \QP)$ and $t>0$,
$T_t^{\U, \QP}u$ has a $\bar{\mssd}_\U$-Lipschitz $\QP$-modification~(denoted by the same symbol~$T_t^{\U, \QP}u$)
such that 
$$\Lip_{\bar{\mssd}_\U}({T}_t^{\U, \QP}u) \le \Lip_{\bar{\mssd}_\U}(u) \ ;$$
\item  {$L^\infty(\QP)$-to-$\Lip(\U, \bar{\mssd}_\U, \QP)$ {regularisation} property}: 
For $u \in L^\infty(\QP)$ and $t>0$, 
$T_t^{\U, \QP}u$ has a $\bar{\mssd}_\U$-Lipschitz $\QP$-modification~(denoted by the same symbol~$T_t^{\U, \QP}u$) 
such that
\begin{align*}
&\Lip_{\bar{\mssd}_\U}({T}_t^{\U, \QP} u) \le \frac{1}{\sqrt{2} t} \|u\|_{L^\infty(\QP)}  \fstop
\end{align*}
\end{enumerate}
 \end{thm}
 We note that the RHS of the dimension-free/log Harnack inequalities including the term~$\bar\mssd_\U(\gamma, \eta)$ is not trivial (i.e., $\bar\mssd_\U(\gamma, \eta) \not \equiv +\infty$ on $\Omega$) as long as $\QP$ is tail-trivial (e.g., $\beta=2$), see Rem.~\ref{r:UME}.

\paragraph{Dyson Brownian motions as a gradient flow}
Jordan, Kinderlehrer and Otto \cite{JorKinOtt98} discovered a class of  partial differential equations that can be realised as gradient flows in the space~$(\mathcal P_2, \mssW_2)$ of probability measures with finite second moment endowed with the $L^2$-Monge-Kantrovich-Rubinstein-Wasserstein distance~$\mssW_2$. In particular, the dual flow of the heat equation in the Euclidean space $\R^n$, where the corresponding diffusion process is the Brownian motion in $\R^n$,  is characterised as the $\mssW_2$-gradient flow
$$``\partial_t \nu=-\nabla_{\mssW_2}\Ent(\nu)"$$ of the Boltzmann-Shannon entropy $\Ent(\nu)=\int_{\R^n} \rho \log \rho \diff x$ with $\diff \nu = \rho \diff x$. Here,  {\it the $\mssW_2$-gradient flow} is defined as the energy dissipation equality:
\begin{align} \label{e:GFE}
\frac{\diff}{\diff t} \Ent({\nu_t}) = -|\dot\nu_t|^2 = -|{\sf D}^-_{\mssW_{2}} \Ent|^2(\nu_t) \quad\text{a.e.~$t>0$} \comma
\end{align}
where $|\dot\nu_r|$ denotes the metric speed of the curve $(\nu_r)$ and $|{\sf D}_{\mssW_{2}}^- \Ent|$ is the descending slope of $\Ent$ with respect to $\mssW_{2}$, see \S\ref{s:MS}. This relates {\it the Brownian motion, the Boltzmann--Shannon entropy}, and {\it the optimal transport distance $\mssW_2$} in a single equation~\eqref{e:GFE}, and brought a new perspective of the Brownian motion  as a {\it steepest descent} of the Boltzmann-Shannon entropy with respect to~$\mssW_2$. 

Exploiting Thm.~\ref{t:intromain}, we can extend this perspective to the case of infinite Dyson Brownian motions in terms of the Boltzmann-Shannon entropy $\Ent_\QP(\nu)=\int_{\U} \rho \log \rho \diff \QP$ for $\diff\nu=\rho\cdot \diff\QP$ (simply written as $\nu=\rho\cdot\QP$) associated with $\mu=\sine_\beta$ for $\beta>0$ and a Benamou--Brenier-like extended distance~$\mssW_{\E}$.  Let $\mathcal P(\U)$ be the space of all Borel probability measures in $\U$ and $\mathcal P_\QP(\U)=\{\nu \in \mathcal P(\U): \nu \ll \QP\}$. 
For~$\nu, \sigma \in \mathcal P_\QP(\U)$, we define $\mssW_{\E}$ as 
\begin{align*} 
\mssW_{\E}(\nu, \sigma)^2:=\inf\biggl\{ \int_0^1 \|\rho'_t\|^2 \diff t : (\rho_t) \in \mathsf{CI}(\E^{\U, \QP}) \comma \nu=\rho_0\cdot \QP\comma \sigma=\rho_1\cdot \QP\biggr\}\comma
\end{align*}
where $(\rho_t) \in \mathsf{CI}(\E^{\U, \QP})$ satisfies a  {\it continuity inequality}, and $\|\rho'_t\|$ is the {\it modulus of verocity}, see Dfn.~\ref{d:CI}.
If there is no $(\rho_t) \in \mathsf{CI}(\E^{\U, \QP})$ connecting $\nu$ and $\sigma$, we define $\mssW_{\E}(\nu, \sigma)=+\infty$. The extended distance $\mssW_\E$ can be thought of as an intrinsic distance on~$\mathcal P_\QP(\U)$ induced by $(\E^{\U, \QP}, \dom{\E^{\U, \QP}})$ as it is determined only by the Dirichlet form data. 
 Let $\dom{\Ent_\QP}:=\{\nu \in \mathcal P(\U): \Ent_\QP(\nu)<+\infty\}$ be the domain of $\Ent_\QP$. Let $t\mapsto \mathcal T_t^{\U, \QP}\nu$ be the dual flow of~$T_t^{\U, \QP}$ defined as
$$\mathcal T_t^{\U, \QP}\nu:=(T_t^{\U, \QP}\rho)\cdot  \mu \cquad \nu=\rho\cdot \mu \in \mathcal P(\U) \fstop$$
 \begin{cor}[{Cor.~\ref{t:mainc}, \ref{c:GC}, \ref{c:GF}}] \label{intro:mainc} Let $\mu=\sine_\beta$  with $\beta>0$.
\begin{enumerate}[(a)]
 \item Evolutional variation inequality: For every $ \nu, \sigma\in \dom{\Ent_\QP}$ with $\mssW_{\E}(\nu, \sigma)<+\infty$, the curve $t \mapsto \mathcal T^{\U, \QP}_t \sigma \in (\mathcal P(\U), \mssW_{\E})$ is locally absolutely continuous, $\Ent_\QP(\mathcal T_t^{\U, \QP}\sigma)<+\infty$, $\mssW_{\E}(\mathcal T_t^{\U, \QP}\sigma, \nu)<+\infty$ for every $t>0$, and 
\begin{align*} 
\frac{1}{2}\frac{\diff^+}{\diff t}{\mssW}_{\E}\bigl({\mathcal T^{\U, \QP}_t \sigma}, \nu \bigr)^2  \le \Ent_{\mu}({\nu}) - \Ent_{\mu}({\mathcal T^{\U, \QP}_t \sigma})\cquad t>0 \fstop 
\end{align*}
\item \label{cc:1} {Geodesic convexity}: The space~$(\dom{\Ent_\QP}, \mssW_{\E})$ is an extended geodesic metric space. Namely, for every pair $\nu, \sigma \in \dom{\Ent_\QP}$ with $\mssW_{\E}(\nu, \sigma)<+\infty$, there exists $\mssW_{\E}$-Lipschitz curve $\nu_\cdot: [0,1] \to (\dom{\Ent_\QP},\mssW_{\E})$ so that
\begin{align*} 
\nu_0=\nu \cquad \nu_1=\sigma \cquad \mssW_{\E}(\nu_t, \nu_s) =|t-s|\mssW_{\E}(\nu, \sigma) \cquad s, t \in [0, 1]\fstop
\end{align*}
Furthermore, the entropy~$\Ent_\QP$ is convex along every ${\mssW}_{\E}$-geodesic $(\nu_t)_{t \in [0,1]}$$:$
\begin{align*}
\Ent_\QP(\nu_t) \le (1-t)\Ent_{\QP}(\nu_0) + t \Ent_{\QP}(\nu_1) \cquad t \in [0,1] \fstop
\end{align*}
\item \label{cc:1-2} Gradient flow: The dual flow~$\bigl\{\mathcal T_t^{\U, \QP}\nu_0\bigr\}_{t \ge 0}$ is the unique solution to the $\mssW_{\E}$-gradient flow of $\Ent_{\QP}$ starting at $\nu_0$. Namely, for any $\nu_0 \in \dom{\Ent_\QP}$, the curve $[0, +\infty) \ni t \mapsto \nu_t=\mathcal T_t^{\U, \QP} \nu_0\in \dom{\Ent_\QP}$ is the unique solution to the energy equality starting at $\nu_0$:
\begin{align*}
\frac{\diff}{\diff t} \Ent_\QP({\nu_t}) = -|\dot\nu_t|^2 = -|{\sf D}^-_{\mssW_{\E}} \Ent_\QP|^2(\nu_t) \quad\text{a.e.~$t>0$} \fstop
\end{align*}
Here, $|\dot\nu_t|:=\lim_{s \to t}\frac{\mssW_{\E}(\nu_s, \nu_t)}{|s-t|}$ is the metric speed of $\nu_t$ and 
\begin{align*}
|\mathsf D_{\mssW_{\E}}^-\Ent_\QP|(\nu):=
\begin{cases} \displaystyle
\limsup_{\sigma \to \nu}\frac{\bigl(\Ent_{\QP}(\sigma)-\Ent_{\QP}(\nu)\bigr)^-}{\mssW_{\E}(\sigma, \nu)} &\text{if $\nu$ is not isolated} , 
\\
0 &\text{otherwise} \fstop
\end{cases}
\end{align*} 
\end{enumerate}
\end{cor}

\paragraph{Generalisation beyond $\sine_\beta$} At the end of this article, our results will be extended to $\BE(K,\infty)$ with $K \in \R$ for a general probability measure~$\QP$ in $\U=\U(\R)$ satisfying {\it conditional geodesic $K$-convexity}, see Thm.~\ref{t:GT}. We apply Thm.~\ref{t:GT} to prove $\BE(0,\infty)$ of the Dirichlet form $(\E^{\U, \QP}, \dom{\E^{\U, \QP}})$ associated with the law of the one-dimensional $(\beta, s)$-circular Riesz ensemble~$\QP=\QP_{\beta, s}$ with $\beta>0$ and $s \in(0,1)$, where the interaction potential is given by $g(x)=|x|^{-s}$ for $x \in \R$.  It was introduced in \cite[Thm.~1.8]{DerVas21} as a subsequential infinite-volume limit of the finite-volume Riesz gas.

\begin{cor}[Cor.~\ref{c:BRE}]
Let $\QP=\QP_{\beta, s}$ with $\beta>0$ and $ s \in (0,1)$ be the one-dimensional $(\beta, s)$-circular Riesz ensemble. 
Then, the corresponding Dirichlet form $(\E^{\U, \QP}, \dom{\E^{\U, \QP}})$ satisfies $\BE(0,\infty)$ for $\beta>0$ and $s \in (0,1)$.
 Furthermore, the statements of Thm.~\ref{t:intromain},~\ref{t:1.5},~\ref{t:3} and Cor.~\ref{c:1},\ref{intro:mainc}  hold in this case. 
\end{cor}

\subsection*{Comparison with the finite particle case}
It is a classical result in random matrix theory that the finite particle Dyson Brownian motion satisfies the $\cdc_2$-condition~$\cdc_2 \ge 0$, which can be immediately seen by the computation of the Hessian with the logarithmic interaction potential. The derivation from $\cdc_2 \ge 0$ to the $\BE(0,\infty)$ gradient estimate is the technical part. To apply a general theory, e.g., \cite[Cor.~3.3.19]{BakGenLed14}, one needs a good space of test functions in the domain of the infinitesimal generator supporting e.g., the {essential self-adjointness}. 
In the infinite particle case, however, the $\cdc_2$-criterion is not available because, due to the long range correlation of $\sine_\beta$, there is no known space of test functions on which the $L^2$-infinitesimal generator of $(\E^{\U, \QP},\dom{(\E^{\U, \QP}})$ is computable and has a concrete expression in the $L^2$ space. 
In the proof of Thm.~\ref{t:intromain}, we pay great attention to the construction of the domain~$\dom{\E^{\U, \QP}}$ that is large enough to lift $\BE(0,\infty)$ from the space of finite particles, by which we can avoid using the $\cdc_2$-criterion. To do so, we use the measurable extended distance $\bar\mssd_\U$ on $\U$ and construct the domain~$\dom{\E^{\U, \QP}}$ large enough to contain the $\QP$-measurable $\bar\mssd_\U$-Lipschitz algebra $\Lip_b(\U, \bar\mssd_\U, \QP)$ and have the Rademacher-type property $\cdc^{\U}(u) \le \Lip_{\bar\mssd_\U}(u)^2$. We then leverage recent developments of geometric analysis on metric measure spaces (such as the theory of $\RCD$ spaces) and the DLR equation for $\sine_\beta$ recently proven in \cite{DerHarLebMai20}.

\subsection*{Comparison with Literature}
To the author's best knowledge, this is the first article addressing the lower Ricci curvature bound on~$\U$ under the presence of interactions. 
{Even with a simpler interaction potential like compactly supported smooth pair potential with Ruelle condition,  no result regarding the curvature bound has been discussed so far.} 
In the non-interacting case where the symmetrising  measure~$\QP$ is the Poisson measure, it has been studied in~\cite{ErbHue15} in the case where the base space is Riemannian manifolds and in~\cite{LzDSSuz22} in the case where the base space is a general diffusion space.
{In~\cite{ErbHueJalMul23}, a specific entropy associated with the Poisson point process and an optimal transport distance have been introduced in the space of translation-invariant point processes. They established the evolutional variation inequality, the gradient flow property, the displacement convexity and the HWI inequality for the flows induced by independent Brownian particles starting at stationary measures.}
 In the case of finite particle systems, a variable Ricci curvature bound has been addressed in~\cite{GyuVon20} for Coulomb-type potentials. 

Up until now, only little is understood about the transition probability of interacting {\it infinite} particle diffusions, and, in particular almost nothing is known about {\it quantitative estimates}. 
The functional inequalities in  Cor.~\ref{c:1},~\ref{c:3-1},  Thm.~\ref{t:3}  and the exponential decay estimate of the transition semigroup in Cor.~\ref{c:3} contribute to this direction.
Furthermore, the dimension-free Harnack inequality in Thm.~\ref{t:3} provides quantitative estimates of the semigroup  in terms of the metric structure~$\bar\mssd_\U$, which could give a new approach to study the Dyson SDEs in a geometric manner.  We note that~\cite{KopStu21} provided an equivalence between a synthetic lower Ricci curvature bound (what is called $\RCD$ condition) and the Wang's dimension-free Harnack inequality in a framework of metric measure spaces. We cannot however apply their result to our setting because (a) we do not know if $(\U, \bar\mssd_\U, \QP)$ is an $\RCD$ space; (b) $(\U, \bar\mssd_\U, \QP)$ is not a metric measure space due to the fact that $\bar\mssd_\U$ does not generate the given topology~$\tau_\mrmv$ on~$\U$ and $\bar\mssd_\U$ takes $+\infty$ on sets of positive measure with respect to~$\QP$. We, therefore, prove the dimension-free Harnack inequality through a finite-particle approximation.

On the {\it qualitative side}, we revealed in Cor.~\ref{intro:mainc} that the infinite Dyson Brownian motion is the unique $\mssW_\E$-gradient flow of $\Ent_\QP$ associated with $\mu=\sine_\beta$, which provides a new perspective of the Dyson Brownian motion as a steepest descent of the Boltzmann--Shannon entropy associated with $\sine_\beta$ in terms of the extended distance~$\mssW_\E$.

\subsection*{The unlabelled solutions in the range $0<\beta<1$}
We constructed a Dirichlet form whose symmetrising (thus also invariant) measure is $\sine_\beta$ for every $\beta>0$ (cf., for the case of $\beta=1, 2, 4$ in~\cite{Osa96, Osa13}). 
Due to Cor.~\ref{c:QRB}, there exists a diffusion process associated with $(\E^{\U, \QP}, \mathcal F)$ for every $\beta>0$. The range $0<\beta<1$ was  not covered by the construction of the SDE \eqref{d:DBMS} in~\cite{Tsa16}, where only the range $\beta \ge 1$ was discussed.  It would therefore be a natural question whether our diffusion process gives the unlabelled solution to \eqref{d:DBMS} in the range $0<\beta<1$. This question is, however,  more delicate than the case $\beta \ge 1$ because the system \eqref{d:DBMS} is expected to have collisions among particles with positive probability. In this case, the diagonal set in the product space~$\R^{\infty}$ as well as in the configuration space~$\U$ plays a role as a boundary, and one needs to understand the boundary behaviour of the Dyson SDE at the diagonal set.

\subsection*{Discussion for the best possible $K$} 
The curvature lower bound $K=0$~in Thm.~\ref{t:intromain} does not depend on the inverse temperature~$\beta$. One might wonder if there is a positive constant~$K_\beta>0$ depending on~$\beta$ so that the sharper curvature bound $\BE(K_\beta, \infty)$ holds. {However, this is not the case when $\beta=2$ due to the absence of the spectral gap proven in \cite{Suz24}. We believe that $K=0$ is the best constant also for other $\beta>0$} because  the logarithmic potential $-\beta\log|x-y|$ cannot be $K$-convex with positive $K>0$ for any $\beta>0$, which indicates that the choice of $\beta$ could not improve the curvature bound in the infinite-volume case.

\subsection*{Outlook for further study}
 In Thm.~\ref{t:GT}, we provide a sufficient condition for the Bakry--\'Emery lower Ricci curvature bound $\BE(K,\infty)$ for a general probability measure in $\U=\U(\R)$. In particular, the laws of the sine $\beta$ ensemble, the $\beta$-Riesz ensemble, and the Poisson ensemble are included due to the recent developments regarding the Dobrushin--Lanford--Ruelle (DLR) equations. The case of the Airy ensemble remains open. The corresponding interacting diffusions are closely related to what is called the {\it Airy line ensemble}, which has a thriving development in the context of the KPZ universality. 
 From the metric geometric viewpoint, a significant open question  is whether $(\U, \tau_\mrmv, \mssd_\U, \QP)$ with $\QP=\sine_\beta$ is an $\RCD$ space, which is stronger than the $\BE$ property. This is also related to an unsolved question, whether the dual flow $\{\mcT_t^{\U, \QP}\}$ of the Dyson Brownian motion~\eqref{d:DBMS} can be realised as the $\EVI$-gradient flow in terms of the $L^2$-transportation extended distance~$\mssW_{2, \mssd_\U}$ with cost $\mssd_\U^2(\gamma, \eta)$, rather than the Benamou--Brenier-type variational extended distance~$\mssW_\E$ discussed in this paper. We do not know whether $\mssW_{2, \mssd_\U}$ coincides with $\mssW_\E$, which we expect to be true in the spirit of the Benamou--Brenier formula known for the Euclidean space. Another fundamental question is whether there exists a probability measure $\QP$ in~$\U$ supporting~$\BE(K,\infty)$ with $K>0$ and~$\QP(\U^\infty)=1$, where $\U^\infty=\{\gamma \in \U: \gamma(\R)=+\infty\}$ is  the infinite-particle configuration~space.

\subsection*{Outline of the article}
In Section~\ref{sec:pre}, the notation and the preliminary materials are presented. In Section~\ref{sec: Pre}, we discuss the lower Ricci curvature bound of finite particle systems in closed intervals, where infinitely many particles are conditioned outside the intervals. To do so,  we construct the Dirichlet forms
\begin{align}\label{e:TDFI}(\E^{\U(B_r), \QP_r^\eta}, \dom{\E^{\U(B_r), \QP_r^\eta}})
\end{align}
 on the configuration space~$\U(B_r)$ over the closed metric ball $B_r$ with radius $r>0$ centred at~$0$, whose symmetrising measure is the projected regular conditional probability $\QP_r^\eta$ on~$\U(B_r)$ conditioned at $\eta$ on the complement~$B^c_r \subset \R$. The key point for the lower Ricci curvature bound of~\eqref{e:TDFI} is {\it the geodesic convexity} of the corresponding Hamiltonian on $(\U(B_r), \bar{\mssd}_\U)$, i.e., the logarithm of the Radon--Nikod\'ym density $\Psi_{r}^{\eta}:=-\log(\diff \mu_r^\eta/\diff \pi_{\mssm_r})$, where  $\pi_{\mssm_r}$ denotes the Poisson measure  on $\U(B_r)$ having the intensity measure~$\mssm_r$, which is the Lebesgue measure restricted on~$B_r$.  This convexity is due to the following DLR equation proven in~\cite[Thm.1.1]{DerHarLebMai20}: for $\mu$-a.e.~$\eta$, there exists a unique $k=k(\eta) \in \N_0$ so that  
\begin{align*}
\diff\mu_r^\eta&
=\frac{1}{Z_{r}^{\eta}} e^{-\Psi_{r}^{k, \eta}} \diff \mssm_r^{\odot k} \comma \quad \gamma=\sum_{i=1}^k \delta_{x_i} \in \U(B_r) 
\\ 
 \Psi_{r}^{k, \eta}(\gamma) &:= - \log \Biggl(  \prod_{i<j}^k |x_i-x_j|^\beta \prod_{i=1}^k\lim_{R \to \infty}\prod_{y \in \eta_{B_r^c}, |y| \le R} \Bigl|1-\frac{x_i}{y}\Bigr|^\beta \Biggr) \comma \notag
\end{align*}
where $\mssm_r^{\odot k}$ is the $k$-symmetric product measure of $\mssm_r$ 
and $Z_{r}^{\eta}$ is the normalising constant (note that the roles of the notation~$\gamma$ and $\eta$ in \cite{DerHarLebMai20} are opposite to this article). 
In Section~\ref{sec:CI}, we prove $\BE(0,\infty)$ of $(\E^{\U, \QP}, \dom{\E^{\U, \QP}})$ in the following steps: we first construct {\it the truncated form}~$(\E_r^{\U, \QP}, \dom{\E_r^{\U, \QP}})$ on $\U$ whose gradient operator is truncated inside configurations on $B_r$ (Prop.~\ref{t:ClosabilitySecond}). We then identify it with the {\it superposition} Dirichlet form $(\bar{\E}_r^{\U, \QP}, \dom{\bar{\E}_r^{\U, \QP}})$ lifted from $(\E^{\U(B_r), \QP_r^\eta}, \dom{\E^{\U(B_r), \QP_r^\eta}})$ (Thm.~\ref{t:S=M}). By this identification, we can lift $\BE(0,\infty)$ from $(\E^{\U(B_r), \QP_r^\eta}, \dom{\E^{\U(B_r), \QP_r^\eta}})$ onto the truncated form~$(\E_r^{\U, \QP}, \dom{\E_r^{\U, \QP}})$. By the monotonicity of the form $(\E_r^{\U, \QP}, \dom{\E_r^{\U, \QP}})$ with respect to $r$ and passing to the limit $r \to \infty$, we prove $\BE(0,\infty)$ for the limit form~$(\E^{\U, \QP}, \dom{\E^{\U, \QP}})$ (Thm.~\ref{t: main}). 
\begin{figure}[htb!]
\begin{equation*} 
{ \xymatrix @C=1pc { 
&  (\bar{\E}_r^{\U, \QP}, \dom{\bar{\E}_r^{\U, \QP}}) \ \  \ar@{=}[r]^{identification} \ar@{<-}[d]^{superposition} & \ \ (\E_r^{\U, \QP}, \dom{\E_r^{\U, \QP}}) \ar@{->}[r]^{monotone}  & (\E^{\U, \QP}, \dom{\E^{\U, \QP}}) &\hspace{-12mm}\BE(0,\infty)&
\\
&(\E^{\U(B_r), \QP_r^\eta}, \dom{\E^{\U(B_r), \QP_r^\eta}}) & \hspace{-20mm}\BE(0,\infty)&   \hspace{45mm} &
}
}
\end{equation*}
\caption{{\rm The idea of the proof of $\BE(0,\infty)$}: 
{\rm $\BE(0,\infty)$ is transferred to $\E^{\U, \QP}$ via the vertical arrow by the superposition, the equality by the identification,  and the right arrow by the monotone convergence.}}
\end{figure}
As a consequence of $\BE(0,\infty)$, we obtain the integral Bochner inequality, the local Poincar\'e inequality (Cor.~\ref{t:LPS}),  the exponential decay of the transition semigroup, the $p$-Bakry--\'Emery estimate, the local log-Sobolev inequality and the local hyper-contractivity (Cor.~\ref{c:TES}, \ref{p:PBE}, \ref{c:LLSI}, \ref{c:LHC}).  
In Section~\ref{sec:LH}, we prove the dimension-free Harnack inequality, the log-Harnack inequality, the Lipschitz contraction and $L^\infty(\U, \QP)$-to-$\Lip(\U, \bar{\mssd}_\U, \QP)$ regularisation properties (Thm.~\ref{t:DFH}). 
In Section \ref{sec:GF}, we discuss the Benamou--Brenier-type variational extended distance~$\mssW_{\E}$, the evolutional variation inequality and the gradient flow property of the dual flow. 
In Section~\ref{sec:GL}, we extend these results to the case of general~$\QP$ beyond $\sine_\beta$ (Thm.~\ref{t:GT}) and discuss the  $(\beta,s)$-circular Riesz ensembles.

\subsection*{Acknowledgement}
The author expresses his great appreciation to Professor Thomas Lebl\'e for making him aware of the $\beta$-Riesz ensemble as an example of Assumption~\ref{a:GT}.
He also appreciates the anonymous referees for reading the manuscript very carefully and giving many valuable suggestions for the improvement. 
A large part of the current work has been completed while he was at Bielefeld University.  He gratefully acknowledges funding by the Alexander von Humboldt Stiftung to support his stay in Bielefeld. 
This research was partly supported by the Theoretical Sciences Visiting Program (TSVP) at the Okinawa Institute of Science and
Technology Graduate University in Japan. 

\subsection*{Data Availability Statement}
No datasets were generated or analysed during the current study.

\section{Notation and Preliminaries} \label{sec:pre}

\subsection{Numbers, Tensors, Function Spaces}
We write $\N:=\{1, 2, 3, \ldots\}$, $\N_0:=\{0, 1, 2, \ldots \}$, $\EN:=\N \cup \{+\infty\}$ and $\EN_0:=\N_0 \cup\{+\infty\}$. 
The uppercase letter $N$ is used for  $N \in \EN_0$, while the lowercase letter $n$ is used for $n \in \N_0$. 
We shall adhere to the following conventions:
\begin{itemize}
\item the superscript~${\square}^\tym{N}$  denotes $N$-fold \emph{product objects};

\item the superscript~${\square}^\otym{N}$ denotes $N$-fold \emph{tensor objects};

\item the superscript~${\square}^\osym{N}$ denotes $N$-fold \emph{symmetric tensor objects}.
\end{itemize}
Let~$(X, \tau)$ be a topological space with $\sigma$-finite Borel measure~$\nu$. If not otherwise stated, a {\it function} always means a function taking values in the real number field~$\R$ or~the extended real number field~$\R\cup\{\pm \infty\}$. We say that a function $v$ is a  {\it $\nu$-modification} of a function~$u$ if $u=v$ $\nu$-almost everywhere. For a subset~$A \subset X$, we write $\nu\mrestr{A}$ for the restriction of the measure~$\nu$ to $A$, and $u|_A$ for the restriction of the function~$u$ to $A$. 
We use the following symbols:
\begin{enumerate}[$(a)$]
\item  {$L^0(X, \nu)$ for the space of $\nu$-equivalence classes of functions $X \to \R$;} for $1 \le p \le \infty$, $L^p(X, \nu):=\{u \in L^0(X, \nu): \|u\|_{L^p(\nu)}<\infty\}$, where $\|u\|_p=\|u\|_{L^p(\nu)}:=\int_X |u|^p \diff \nu$ for $1 \le p<\infty$ and $\|u\|_{\infty}=\|u\|_{L^\infty(\nu)}:=\esssup_X u$ for $p=\infty$.  In the case $p=2$, the inner-product is denoted by $(u, v)_{L^2(\nu)}:=\int_X uv \diff \nu$. If no confusion could occur, we simply write $L^p(\nu)=L^p(X, \nu)$;

\item  $L^p_s(X^{\times k}, \nu^{\otimes k}):=\{u \in L^p(X^{\times k}, \nu^{\otimes k}): u\ \text{is symmetric}\}$, where $u$ is said to be {\it symmetric} if and only if $u(x_1, \ldots, x_k)=u(x_{\sigma(1)}, \ldots, x_{\sigma(k)})$ for every $\sigma \in \mathfrak S(k)$ in the $k$-symmetric group;

\item $C_b(X, \tau)$ for the space of bounded $\T$-continuous functions on~$X$; if $X$ is locally compact, $C_0(X, \tau)$ denotes the space of $\tau$-continuous and compactly supported  functions on $X$; $C_0^\infty(\R)$ for the space of compactly supported smooth functions on~$\R$. If no confusion could occur, we simply write $C_b(X)$ and $C_0(X)$ respectively. 

\item $\mathscr B(X)$ for the Borel $\sigma$-algebra with respect to~$\tau$; $\mathscr B(X)^{\nu}$ for the completion of $\mathscr B(X)$ with respect to~$\nu$; $\mathscr B(X)^*$ for the universal $\sigma$-algebra, i.e., the intersection of $\mathscr B(X)^{\rho}$ among all Borel probability measures $\rho$ on $X$;
A  measurable function $u: X \to \R$ with respect to $\mathscr B(X)$, $\mathscr B(X)^{\nu}$, $\mathscr B(X)^*$ is called {\it Borel measurable, $\nu$-measurable, universally measurable} respectively and denoted by $u\in \mathcal B(X), \mathcal B(X)^{\nu}, \mathcal B(X)^*$ respectively;

\item $F_\# \nu$ for the push-forward measure, i.e., $F_\# \nu(\cdot)=\nu(F^{-1}(\cdot))$ given a measurable space $(Y, \Sigma)$ and a measurable map $F: (X, \mathcal B(X)^{\nu}) \to (Y, \Sigma)$;
\item $\1_{A}$ for the indicator function on $A$, i.e., $\1_{A}(x)=1$ if and only if $x \in A$, and $\1_A(x)=0$ otherwise; $\delta_x$ for the Dirac measure at $x$, i.e., $\delta_x(A)=1$ if and only if $x \in A$, and $\delta_{x}(A)=0$ otherwise;
\item $\square_+$ for a subspace of  nonnegative functions from $X$ to $\R$. For instance, $C_{b,+}(X):=\{u \in C_b(X): u \ge 0\}$. 
\end{enumerate}

\subsection{Dirichlet form} \label{ss:DF}
We refer the reader to \cite{MaRoe90, BouHir91, FukOshTak11, CheFuk11} for comprehensive references regarding materials presented in this subsection. Throughout this article, a Hilbert space always means a separable Hilbert space with inner product~$(\cdot, \cdot)_H$ taking value in $\R$. 

\paragraph{Closed form}
Let $H$ be a Hilbert space and $\dom{Q}$ be a dense linear subspace in~$H$. We call a pair $(Q, \dom{Q})$  {\it symmetric form} or simply {\it form} if it is  a non-negative definite symmetric bilinear form~$Q: \dom{Q}\times \dom{Q} \to \R$, i.e., $Q(u, v)=Q(v, u)$, $Q(u+v, w)=Q(u, w)+Q(v, w)$, $Q(\alpha u, v)=\alpha Q(u, v)$ and $Q(u, u) \ge 0$ for $u, v, w \in \dom{Q}$ and $\alpha \in \R$.  Set
\begin{align*}
Q(u)\eqdef Q(u,u) \comm \qquad Q_\alpha(u,v)\eqdef Q(u,v)+\alpha(u, v)_H\comm \alpha>0\fstop
\end{align*}
The form~$(Q, \dom{Q})$ is  \emph{closed} if the space $\dom{Q}$ endowed with the norm
\begin{align}\label{d:FN}
\norm{\emparg}_{\dom{Q}}:= Q_1(\emparg)^{1/2}=\sqrt{Q(\emparg)+\norm{\emparg}_{L^2(\nu)}^2} 
\end{align}
 is a Hilbert space. We say that $(Q, \dom{Q})$ is {\it closable} if for $u_n \in \dom{Q}$, 
\begin{align*}
Q(u_n-u_m) \xrightarrow{n, m \to \infty} 0 \cquad \|u_n\|_{H} \xrightarrow{n \to \infty}0  \quad \implies \quad Q(u_n) \xrightarrow{n \to \infty}0 \fstop
\end{align*}
If $(Q, \dom{Q})$ is closable, there exists the smallest closed extension (also called {\it closure}) $(Q', \dom{Q'})$ of  $(Q, \dom{Q})$, i.e., {$(Q', \dom{Q'})$ is the smallest form satisfying that $\dom{Q} \subset \dom{Q'}$, $Q'=Q$ on $\dom{Q}^{\times 2}$ and $(Q', \dom{Q'})$ is closed.}

\paragraph{Generator, semigroup, and resolvent} 
Let $(Q, \dom{Q})$ be a closed symmetric form on a Hilbert space~$H$.  {\it The infinitesimal generator $(A, \dom{A})$}  is the unique densely defined self-adjoint operator on $H$ satisfying the following integration-by-parts formula:
$$-(u, Av)_{H}=Q(u, v) \cquad u \in \dom{Q},\ v\in \dom{A} \fstop$$
{\it The resolvent operator} $\{{G}_\alpha\}_{\alpha \ge 0}$ is the unique bounded linear operator on $H$ satisfying 
$$Q_\alpha({G}_\alpha u, v) = (u, v)_{H} \cquad  u \in H \quad v \in \dom{Q} \fstop$$
{\it The semigroup} $\{T_t\}_{t \ge 0}$ is the unique bounded linear operator on $H$ satisfying 
$$G_\alpha u = \int_{0}^\infty e^{-\alpha t} T_{t}u \diff t \cquad u \in H\fstop$$
The semigroup $\{T_t\}_{t \ge 0}$ has the following contraction properties (see, e.g., \cite[\S1.3, p.16 \& Lem.~1.3.3]{FukOshTak11}): for every $t>0$
\begin{align}\label{e:con}
\|T_t u\|_{H} \le \|u\|_H \quad u \in H  \comma \quad Q(T_tu) \le Q(u) \quad u \in \dom{Q}\fstop
\end{align}

\paragraph{Dirichlet form}
Let $(X, \Sigma, \nu)$ be a $\sigma$-finite measure space. 
A closed symmetric form~$(Q,\dom{Q})$ on~$L^2(\nu)$ is a symmetric {\it Dirichlet form} if it  satisfies the Markovian property (also called {\it sub-Markovian property})
\begin{align*}
u_0\eqdef 0\vee u \wedge 1\in \dom{Q} \quad \text{and} \quad Q(u_0)\leq Q(u)\comm \quad u\in\dom{Q}\fstop
\end{align*}
It is equivalent to the following property: $0 \le T_t u \le 1$ whenever $0 \le u \le 1$, see, e.g., \cite[Thm.~1.4.1]{FukOshTak11}. In this case,  the contraction \eqref{e:con} extends to the $L^p$ space for all $1 \le p \le \infty$:
\begin{align}\label{e:con1}
\|T_t u\|_{L^p} \le \|u\|_{L^p} \quad u \in L^p(\nu)  \fstop
\end{align}
See, e.g., \cite[Thm.~1.3.3]{Dav89}. 
To distinguish Dirichlet forms defined in different base spaces with different reference measures, we often use the notation $Q^{X, \nu}$ to specify the base space $X$ and the reference measure $\nu$. 

\paragraph{Square field} A symmetric Dirichlet form $(Q, \dom{Q})$ {\it admits a square field $\cdc$} if there exists a dense subspace $L \subset \dom{Q} \cap L^\infty(\nu)$ with respect to the norm~$\|\cdot\|_{\dom{Q}}$ such that the following property holds: for every $u \in L$, there exists $v \in L^1(\nu)$ so that 
$$2Q(uh, u) -Q(h, u^2) = \int_X h v \diff \nu \cquad  h \in \dom{Q} \cap L^\infty(\nu) \fstop$$
Such $v$ is denoted by $\Gamma(u)$. The square field $\Gamma$ can be uniquely extended to an operator on $\dom{Q} \times \dom{Q} \to L^1(\nu)$ (\cite[Thm.\ I.4.1.3]{BouHir91}).

\paragraph{Quasi-notion} Let $(X, \tau)$ be a Polish space (i.e., separable and metrisable by some complete distance), $\nu$ be a $\sigma$-finite Borel measure on $X$ and $(Q, \dom{Q})$ be a symmetric Dirichlet form on $L^2(\nu)$.
For~$A\in\mathscr B(X)$, we define
$$\dom{Q}_A\eqdef \set{u\in \dom{Q}: u= 0 \ \text{$\nu$-a.e. on}\ X\setminus A}.$$
A sequence $\seq{A_n}_{n \in \N}\subset \mathscr B(X)$ is a \emph{Borel nest} if $\cup_{n \in \N} \dom{Q}_{A_n}$ is dense in~$\dom{Q}$.
A Borel nest~$\seq{A_n}_{n \in \N}$ is \emph{closed (resp.~compact)} if $A_n$ is closed (resp.~compact) for every $n \in \N$. 
A set~$N\subset X$ is \emph{exceptional} if there exists a closed nest~$\seq{F_n}_{n  \in \N}$ such that~$N\subset X\setminus \cup_n F_n$. 
A property~$(p_x)$ for $x\in X$ holds \emph{quasi-every~$x$} (in short: q.e.~$x$) if there exists an exceptional set~$N$ so that~$(p_x)$ holds for every~$x\in X\setminus N$. For a closed nest~$\seq{F_n}_{n \in \N}$, we define
$$C\bigl(\seq{F_n}_{n \in \N}\bigr):=\{u: A \to \R: \cup_{n \ge 1} F_n  \subset A \subset X,\ u|_{F_n} \ \text{is continuous for every $n \in \N$}\} \fstop$$
A function $u$ defined quasi-everywhere on $X$ is {\it quasi-continuous} if there exists a closed nest $\seq{F_n}_{n \in \N}$ so that $u \in C\bigl(\seq{F_n}_{n \in \N}\bigr)$.  

\paragraph{Quasi-regularity} Now we recall a key property, which connects Dirichlet forms and Markov processes. 
The form~$(Q, \dom{Q})$ is {\it quasi-regular} if the following conditions hold:
\begin{enumerate}[{\rm (QR1)}]
\item there exists a compact nest $(A_n)_{n \in \N}$; 
\item there exists a dense subspace $\mathcal D \subset \dom{Q}$ so that every $u \in \mathcal D$ has a quasi-continuous~$\nu$-modification~$\tilde{u}$;
\item there exists $\{u_n: n \in \N\} \subset \dom{Q}$ and an exceptional set $N \subset X$ so that every $u_n$ has a quasi-continuous~$\nu$-modification~$\tilde{u}_n$ and $\{\tilde{u}_n: n \in \N\}$ separates points in~$X \setminus N$.
\end{enumerate}
If $(Q, \dom{Q})$ is quasi-regular, there exists a continuous-time strong Markov process  on~$X$ 
whose transition semigroup coincides with the $L^2$-semigroup of $(Q, \dom{Q})$ quasi-everywhere (see, \cite[Thm.~3.5 p.103]{MaRoe90}). 
The form~$(Q, \dom{Q})$ is called {\it regular} if $X$ is 
a locally compact separable metric space, and $C_0(X) \cap \dom{Q}$ is dense in $(\dom{Q}, Q_1)$ as well as in $C_0(X)$ with respect to the uniform topology.

\paragraph{Locality} Let $(X, \Sigma, \nu)$ be a $\sigma$-finite measure space and let $(Q, \dom{Q})$ be a symmetric Dirihclet form on $L^2(\nu)$. 
The form~$(Q, \dom{Q})$ is called {\it strongly local} if 
$$Q(u, v)=0 \quad \text{whenever} \quad u, v \in \dom{Q} \comma u(v-c)=0 \ \text{on} \ X \ \text{for some constant $c \in \R$} \fstop$$
Note that this definition does not require any topology in $X$. If $(Q, \dom{Q})$ is regular under some topology and distance in~$X$, this definition is equivalent to the standard definition: $Q(u, v)=0$ whenever $u, v \in \dom{Q}$ have compact support and $v$ is a constant on a neighbourhood of the support of $u$, see~\cite[Thm.~2.4.3]{CheFuk11}. If $(Q, \dom{Q})$ is quasi-regular and strongly local, the corresponding continuous-time strong Markov process has continuous trajectories (i.e., a diffusion process) and has no killing inside~$X$, see~\cite[Rmk.~2.4.4 and Thm.~4.3.4]{CheFuk11}.

\subsection{Extended metric space}  \label{s:MS} 
Let~$X$ be a non-empty set. A function $\mssd\colon X^\tym{2}\rar [0,+\infty]$ is an \emph{extended distance} if it is symmetric, satisfying the triangle inequality, and does not vanish outside the diagonal in~$X^{\tym{2}}$, i.e.~$\mssd(x,y)=0$ iff~$x=y$; a \emph{distance} if it maps~$X^{\times 2} \to [0,+\infty)$. 
Let~$x_0\in X$ and~$r\in [0,+\infty)$. We write~$B_r(x_0)\eqdef \set{\mssd_{x_0} \le r}$, where $\mssd_{x_0}:=\mssd(x_0, \cdot)$. 
A space~$X$ equipped with an extended distance (resp.~ a distance) is called an {\it extended metric space} (resp. a {\it metric space})~$(X, \mssd)$. The topology induced by $\mssd$ is denoted by $\tau_\mssd$. We write $x \sim y$ if $\mssd(x,y)<+\infty$, which provides an equivalence relation in $X$. We denote by $\tilde{X}$ the quotient of $X$ with respect to the equivalence relation $\sim$. The space $X$ is decomposed as the disjoint union $X:=\sqcup_{[x] \in \tilde{X}}X_{[x]}$, where $X_{[x]}$ is the space represented by the element $[x] \in \tilde{X}$. Namely, $X_{[x]}:=\{y \in X: \mssd(y, x)<+\infty\}$, which is a metric space. We say that $(X, \mssd)$ is {\it complete} if each component $X_{[x]}$ is complete  as a metric space for every $[x] \in \tilde{X}$.

\paragraph{Lipschitz algebra} A function~$f\colon X\rar \R$ is {\it $\mssd$-Lipschitz} if there exists a constant~$L \ge 0$ so that
\begin{align}\label{eq:Lipschitz}
\tabs{ u(x)-u(y)}\leq L\, \mssd(x,y) \comm \qquad x,y\in X \fstop
\end{align}
The smallest constant~$L$ satisfying~\eqref{eq:Lipschitz} is the (global) \emph{Lipschitz constant of $u$}, denoted by~$\Lip_\mssd{(u)}$.
We write~$\Lip(X,\mssd)$ (resp.~$\bLip(X,\mssd)$) for the space of $\mssd$-Lipschitz functions (resp.\ bounded $\mssd$-Lipschitz functions) on~$X$. 
If no confusion could occur, we simply write 
\begin{align}\label{d:LS}
\Lip(\mssd)=\Lip(X,\mssd)\comma\quad \ \bLip(\mssd)= \bLip(X,\mssd)\fstop
\end{align}
For a given measure $\nu$ on a $\sigma$-algebra $\Sigma$ in $X$ (not necessarily the Borel $\sigma$-algebra $\mathscr B(X, \tau_{\mssd})$), we set 
\begin{align}\label{d:LS2}
\Lip(X, \mssd, \nu)&:=\{u \in \Lip(\mssd): \text{$u$ is $\Sigma^\nu$-measurable}\} \comma
\\
{\Lip_b(X, \mssd, \nu)}&:=\{u \in \Lip(X, \mssd, \nu): \text{$\|u\|_{L^\infty(\nu)}$}<+\infty\}  \comma
\end{align}
where $\Sigma^\nu$ is the completion of the $\sigma$-algebra~$\Sigma$ with respect to~$\nu$.
Similarly, we simply write $\Lip(\mssd, \nu)=\Lip(X, \mssd, \nu)$ and $\Lip_b(\mssd, \nu)=\Lip_b(X, \mssd, \nu)$ if no confusion could occur.

\paragraph{Absolutely continuous curve}
Let $(X, \mssd)$ be an extended metric space, $\tau_\mssd$ be the topology induced by $\mssd$,  and $J \subset \R$ be an open interval. A continuous map $\rho : J \to (X, \tau_\mssd)$ is {\it $p$-absolutely continuous} and denoted by $\rho=(\rho_t)_{t \in J} \in AC^p(J, (X, \mssd))$ if there exists $g \in L^p(J, \diff x)$ so that 
\begin{align}\label{d:ABC}
\mssd(\rho_s, \rho_t) \le \int_s^t g(r) \diff r \cquad s, t \in J \quad s<t \fstop
\end{align}
If $p=1$, we simply say that $\rho$ is {\it absolutely continuous} and denoted by $\rho \in AC(J, (X, \mssd))$. The minimal $g$ among those satisfying~\eqref{d:ABC} exists and this is identical to 
\begin{align}\label{d:MD}
|\dot\rho_t| :=\lim_{s \to t}\frac{\mssd(\rho_s, \rho_t)}{|s-t|} \cquad \text{the limit exists in a.e.~$t \in J$}\comma
\end{align}
which is called {\it metric speed, or metric derivative of $\rho$}. Namely, $|\dot\rho_t|$ satisfies \eqref{d:ABC} and $|\dot\rho_t| \le g(t)$ for a.e.~$t \in J$ for every $g$ satisfying~\eqref{d:ABC}, see \cite[Thm.~1.1.2]{AmbGigSav08}. We say that an absolutely continuous curve $\rho$ is of {\it constant speed} if $|\dot\rho_t|$ is a constant for a.e. $t \in J$.

\paragraph{Geodesic space} Let $(X, \mssd)$ be an extended metric space. We say that $\rho=(\rho_t)_{t \in [0,1]}$ is {\it a constant speed geodesic} connecting $x_0$ and $x_1$ with $\mssd(x_0, x_1)<+\infty$ if $\rho_0=x_0$, $\rho_1=x_1$ and 
\begin{align}\label{d:Geo1}
\mssd(\rho_t, \rho_s) = |t-s|\mssd(x_0, x_1) \cquad  s, t \in [0,1] \fstop
\end{align}
We denote by $\mathsf{Geod}(X, \mssd)$ the space of constant speed geodesics on $(X, \mssd)$. We say that $(X, \mssd)$ is a {\it geodesic} extended metric space if for every~$x_0, x_1 \in X$ with $\mssd(x_0, x_1)<+\infty$, there exists at least one $\rho=(\rho_t)_{t \in [0, 1]} \in \mathsf{Geod}(X, \mssd)$ connecting $x_0$ and $x_1$.  

\paragraph{Geodesic convexity} Let $(X, \mssd)$ be a {geodesic} extended metric space. We say that the function~$U: X \to \R\cup\{+\infty\}$ is {\it $K$-geodesically convex} with $K \in \R$ if for every $x_0, x_1 \in \dom{U}:=\{x \in X: U(x) \in \R\}$ and $\mssd(x_0, x_1)<+\infty$, there exists a constant speed geodesic $\rho=(\rho_t)_{t \in [0,1]}$ with $\rho_0=x_0$, $\rho_1=x_1$ and 
$$U(\rho_t)\le (1-t)U(\rho_0)+tU(\rho_1)-\frac{K}{2}t(1-t)\mssd^2(x_0, x_1) \quad  t \in [0,1] \fstop$$
When $K=0$, we say that $U$ is {\it geodesically convex}.


\paragraph{Slope}
Let~$(X, \mssd)$ be an extended metric space and $u: X \to \R\cup\{\pm \infty\}$ be a function. For $x \in \dom{u}=\{x \in X: u(x) \in \R\}$,  {\it the slope of $u$ at $x$} is defined as 
\begin{align} \label{d:SLP}
|\mathsf D_{\mssd}u|(x):=
\begin{cases} \displaystyle
\limsup_{y \to x}\frac{|u(y)-u(x)|}{\mssd(y, x)} &\text{if $x$ is not isolated}; 
\\
0 &\text{otherwise} \fstop
\end{cases}
\end{align}
It is straightforward to see 
\begin{align} \label{e:SLC}
|\mathsf D_{\mssd}u| \le \Lip_{\mssd}(u) \cquad u \in \Lip(X, \mssd) \fstop
\end{align}
{\it The ascending slope} and {\it the descending slope} are defined correspondingly as 
\begin{align} \label{d:USL}
|\mathsf D_{\mssd}^+u|(x):=
\begin{cases} \displaystyle
\limsup_{y \to x}\frac{(u(y)-u(x))^+}{\mssd(y, x)} &\text{if $x$ is not isolated}; 
\\
0 &\text{otherwise} \fstop
\end{cases}
\end{align}
\begin{align} \label{d:DLP}
|\mathsf D_{\mssd}^-u|(x):=
\begin{cases} \displaystyle
\limsup_{y \to x}\frac{(u(y)-u(x))^-}{\mssd(y, x)} &\text{if $x$ is not isolated}; 
\\
0 &\text{otherwise} \fstop
\end{cases}
\end{align}

\subsection{Cheeger energy} \label{sec:Ch}
We say that $(X, \mssd, \nu)$ is a {\it metric measure space} if $(X,\mssd)$ is a complete metric space with $\tau_\mssd$ separable, $\nu$ is a Radon measure on $\mathscr B(\tau_\mssd)$ such that the topological support of $\nu$ is the whole space~$X$ and $\nu(X)<+\infty$. 
Let $(X, \mssd, \nu)$ be a metric measure space.  According to~\cite[Dfn.~4.2]{AmbGigSav14}, we say that $G \in L^2(\nu)$ is a {\it relaxed gradient of $u \in L^2(\nu)$} if there exists $u_n \in \Lip(\mssd) \cap L^2(\nu)$ and $\tilde G \in L^2(\nu)$ so that 
\begin{enumerate}[(a)]
\item $u_n \to u$ strongly in $L^2(\nu)$ and $|\mathsf D_{\mssd}u_n| \to \tilde G$ weakly in $L^2(\nu)$;
\item $\tilde G \le G$ $\nu$-a.e..
\end{enumerate} 
We say that $G$ is the {\it minimal relaxed gradient of $u$} if its $L^2(\nu)$-norm is minimal among all relaxed gradients. We denote by $|\nabla_{\mssd, \nu} u|_*$ the {minimal relaxed gradient of $u$}. By~\cite[(4.9)]{AmbGigSav14} and \eqref{e:SLC},  we have 
\begin{align}\label{i:MSL}
|\nabla_{\mssd, \nu} u|_* \le |\mathsf D_\mssd u| \le \Lip_{\mssd}(u) \quad \text{$\QP$-a.e.}\cquad u\in \Lip_b(\mssd)\fstop
\end{align}
The Cheeger energy $\Ch^{\mssd, \nu}: L^2(\nu) \to \R \cup\{+\infty\}$
 is defined as 
  \begin{align*}
 \Ch^{\mssd, \nu}(u)&:=\frac{1}{2}\int_X |\nabla_{\mssd, \nu} u|_*^2 \diff \nu\comma
\end{align*}
and set $\Ch^{\mssd, \nu}(u)=+\infty$ if $u$ has no relaxed slope. The domain is denoted by $W^{1,2}(X, \mssd, \nu):=\{u \in L^2(\nu): \Ch^{\mssd, \nu}(u)<+\infty\}$. The functional $\Ch^{\mssd, \nu}$ is convex and lower semi-continuous in~$L^2(\nu)$ (\cite[Thm.~4.5]{AmbGigSav14}). According to the definition of $|\nabla_{\mssd, \nu} u|_*$ and the $L^2$-strong approximation of the minimal relaxed slope~\cite[(c) Lem.~4.3]{AmbGigSav14}, we have
 \begin{align*}
 \Ch^{\mssd, \nu}(u)&=\frac{1}{2}\inf \biggl\{ \liminf_{n \to \infty} \int_{X} |\mathsf D_{\mssd} u_n|^2 \diff \nu:\ u_n \in \Lip_b(\mssd) \xrightarrow{L^2} u \biggr\} \fstop
\end{align*}
Therefore, by, e.g., \cite[Prop.~3.12]{FonLeo07}, $\Ch^{\mssd, \nu}$ is the lower semi-continuous envelope of the functional~
\begin{align}\label{d:ISL}\displaystyle
\mathsf E^{\mssd, \nu}(u):= 
\begin{cases}
\frac{1}{2}\int_{X} |\mathsf D_\mssd u|^2 \diff \nu \quad &u \in \Lip_b(\mssd);
\\
+\infty \quad &\text{otherwise} \fstop
\end{cases}
\end{align}
Namely, for every $u \in L^2(\nu)$, 
\begin{align}\label{d:lse}
\Ch^{\mssd, \nu}(u)=\sup\biggl\{H(u): \ &H: L^2(\nu) \to [-\infty, +\infty] \  \text{is lower semi-continuous}
\\
& \text{and} \  H\le\mathsf E^{\mssd, \nu} \biggr\} \fstop \notag
\end{align}

\subsection{Riemannian Curvature-dimension condition} \label{ss:RCD}
Let $(X, \mssd, \nu)$ be a metric measure space. The following definition is one of the equivalent characterisations of {\it an $\RCD(K,\infty)$ space}, see~\cite[Cor.\ 4.18]{AmbGigSav15}.
We say that $(X, \mssd, \nu)$ satisfies the {\it Riemannian Curvature-Dimension Condition} $\RCD(K,\infty)$ with $K \in \R$ if 
\begin{enumerate}[{\rm (i)}]
\item $\Ch^{\mssd, \nu}$ is quadratic, i.e., 
$$\Ch^{\mssd, \nu}(u+v)+\Ch^{\mssd, \nu}(u-v)= 2\Ch^{\mssd, \nu}(u)+2\Ch^{\mssd, \nu}(v) \cquad u, v \in W^{1,2}(X, \mssd, \nu) \fstop$$
\item Sobolev-to-Lipschitz property holds, i.e., every $u \in W^{1,2}(X, \mssd, \nu)$ with $|\nabla_{\mssd, \nu} u|_* \le 1$ has a $\mssd$-Lipschitz $\nu$-modification~$\tilde{u}$ with $\Lip_\mssd(\tilde{u}) \le 1$;
\item $\Ch^{\mssd, \nu}$ satisfies $\BE(K,\infty)$, i.e., 
$$|\nabla_{\mssd, \nu} T_t u|^2_* \le e^{-2Kt} T_t|\nabla_{\mssd, \nu} u|^2_* \cquad u \in W^{1,2}(X, \mssd, \nu) \comma t>0 \comma$$
 {where $(T_t)_{t>0}$ is  the $L^2$-gradient flow  associated with~$(\Ch^{\mssd, \nu}, W^{1,2}(X, \mssd, \nu))$, see \cite[(4.24)]{AmbGigSav14}.}
\end{enumerate}
In this case, the Cheeger energy $(\Ch^{\mssd, \nu}, W^{1,2}(X, \mssd, \nu))$ is a strongly local symmetric Dirichlet form (\cite[\S 4.3]{AmbGigSav14b}), and the $L^2$-gradient flow $\{T_t\}_{t>0}$ coincides with the $L^2$-semigroup associated with the Dirichlet form~$(\Ch^{\mssd, \nu}, W^{1,2}(X, \mssd, \nu))$. We note that, while \cite[Cor.\ 4.18]{AmbGigSav15} is stated in terms of {\it the minimal weak upper gradient} denoted by $|\nabla\cdot|_w$, it is identical to the minimal relaxed slope $|\nabla_{\mssd, \nu} \cdot|_*$ due to \cite[Thm.~6.2]{AmbGigSav14}.

\subsection{Configuration space}
A \emph{configuration} on a locally compact Polish space~$X$ is any $\overline\N_0$-valued Radon measure~$\gamma$ on~$X$, which can be expressed by $\gamma = \sum_{i=1}^N \delta_{x_i}$ for $N \in \overline{\N}_0$, where $x_i \in X$ for every $i$ and  $\gamma \equiv 0$ if $N=0$.
The \emph{configuration space}~$\U=\dUpsilon(X)$ is the space of all configurations over~$X$. The space~$\dUpsilon$ is endowed with the vague topology~$\tau_\mrmv$, i.e., the topology  by the duality of the space $C_0(X)$ of compactly supported continuous functions. Throughout this paper, the Borel $\sigma$-algebra~$\mathscr B(\U, \tau_\mrmv)$ is written simply as~$\mathscr B(\U)$. We write the {\it restriction}~$\gamma_A\eqdef \gamma\mrestr{A}$ as a measure on a Polish subspace~$A \subset X$ and the corresponding restriction map is denoted by 
\begin{align}\label{eq:ProjUpsilon}
\pr_A\colon \dUpsilon\longrar \dUpsilon(A)\colon \gamma\longmapsto \gamma_{A}\fstop
\end{align}
The $N$-particle configuration space over~$A$ is denoted by
\begin{equation*}
\begin{aligned}
\dUpsilon^N(A)\ \eqdef&\ \set{\gamma\in \dUpsilon(A): \gamma(A)=N}\comma
\end{aligned}
\quad N\in\overline\N_0 \fstop
\end{equation*}
If $A=X$, we simply write~$\U^N=\U^N(X)$. 
Let $\mathfrak S_k$ be the $k$-symmetric group for $k \in \N_0$. It can be readily seen that the $k$-particle configuration space~$\U^k(A)$ is {\it isomorphic} as a topological space to the quotient space~$A^{\times k}/\mathfrak S_k$ endowed with the quotient topology:
\begin{align} \label{e:STS}
\dUpsilon^k(A)\cong A^{\odot k}:=A^{\times k}/\mathfrak S_k \comma \quad k \in \N\fstop
\end{align}
The associated projection map from the product space~$A^{\times k}$ to the quotient space~$A^{\times k}/\mathfrak S_k$ is denoted by~$\quot_k$.

\paragraph{Conditional probability}
Let $(X,  \mssd)$ be a locally compact separable complete metric space. Let $B_r=B_r(x_0)=\{y \in X: \mssd(x_0, y) \le r\}$ be the closed ball with radius~$r$ centred at some fixed point~$x_0 \in X$. For $\eta \in \U=\U(X)$ and $r>0$, we set 
\begin{align} \label{e:CES}
\U_r^\eta:=\{\gamma \in \U: \gamma_{B_r^c}=\eta_{B_r^c}\} \fstop
\end{align}
Let $\mu$ be a Borel probability measure on $\U$ and $\QP_{B_r^c}:={\pr_{B_r^c}}_\#\QP$. According to e.g.,~\cite[452E, 452O, 452G(c)]{Fre00}, there exists a family of Borel probability measures~$\{\QP_{r}^\eta: r>0, \eta \in \U(B_r^c)\}$ on $\U$ so that 
\begin{enumerate}[(a)]
\item (disintegration) for  every $\Xi \in \mathscr B(\U)^\QP$
$$\QP(\Xi)= \int_{\U(B_r^c)} \QP_{r}^\eta(\Xi) \diff \QP_{B_r^c}(\eta)\ ;$$ 
\item (strong consistency) for every  $\Xi \in \mathscr B(\U)^\QP$ and $\Lambda \in  \mathscr B(\U(B_r^c))^{\QP_{B_r^c}}$, 
$$\QP(\Xi \cap \Lambda)= \int_{\Lambda} \QP_{r}^\eta(\Xi) \diff \QP_{B_r^c}(\eta) \quad \text{and} \quad \QP_{r}^\eta(\U_r^\eta)=1 \quad \QP_{B_r^c}\text{-a.e.~$\eta$}\fstop$$ 
\end{enumerate}
We call $\{\QP_{{r}}^{\eta}: r>0, \eta \in {\U(B_r^c)}\}$ {\it strongly consistent regular conditional probability measures}. 
\begin{rem}\label{r:CP}
We may think of $\QP_r^\eta$ as a Borel probability measure on $\U(B_r)$ instead of~$\U$. Indeed, thanks to the strong consistency, the projection~$\pr_{B_r}:\U_r^\eta \to \U(B_r)$ with its inverse $\pr_{B_r}^{-1}: \U(B_r) \to \U_r^\eta$ defined as $\gamma \mapsto \gamma+\eta$ gives a bi-measure preserving bijection map between the two measure spaces 
\begin{align} \label{r:BMP} 
(\U_r^\eta, \QP_r^\eta) \cong (\U(B_r), {\pr_{B_r}}_\# \QP_r^\eta) \fstop
\end{align}
Throughout this paper, we identify  $\QP_r^\eta$ with ${\pr_{B_r}}_\# \QP_r^{{\eta_{B_r^c}}}$ and regard~$\QP_r^\eta$ as a probability measure on $\U(B_r)$ indexed by $\eta \in \U$ and $r>0$. 
\end{rem}

\paragraph{Disintegration formulas} For a function $ u\colon \dUpsilon\to \R$, $r>0$ and~$\eta \in \dUpsilon$, we set 
 \begin{align} \label{e:SEF}
u_{r}^\eta(\gamma)\eqdef  u(\gamma+\eta_{B_r^\complement})  \qquad \gamma\in \dUpsilon(B_r) \fstop
 \end{align}
{By  the property (a) of the conditional probability and the identification~\eqref{r:BMP}}, it is straightforward to see that for every~$u \in L^1(\mu)$, 
\begin{align} \label{p:ConditionalIntegration}
\int_{\dUpsilon} u \diff\QP = \int_{\dUpsilon} \quadre{\int_{\dUpsilon(B_r)} u_{r}^\eta \diff \QP^\eta_r }\diff\QP(\eta) \fstop
\end{align}
For a $\QP$-measurable set~$\Omega \in \mathscr B(X)^{\QP}$, define a {\it section}~$\Omega_r^\eta \subset \U(B_r)$ at~$\eta \in \U$ on~$B_r^c$ by
 \begin{align} \label{e:SEF2}
 \Omega_r^\eta:=\{\gamma \in \U(B_r): \gamma+\eta_{B_r^c} \in \Omega\} \fstop
 \end{align}
 By applying the disintegration formula~\eqref{p:ConditionalIntegration} to $u=\1_{\Omega}$, we obtain
\begin{align} \label{p:ConditionalIntegration2}
\QP(\Omega)=\int_{\U} \QP_r^\eta(\Omega_r^\eta) \diff \QP(\eta) \fstop
\end{align}

\paragraph{Intensity measure} For a Borel probability measure~$\QP$ on $\U=\U(X)$, the {\it intensity measure} $I_{\QP}$ is a Borel measure on~$X$ defined as 
\begin{align} \label{d:IS}
I_\QP(A):=\int_{\U} \sum_{x \in \gamma }\1_A(x)\diff \QP(\gamma) \cquad  A \in \mathscr B(X)\fstop
\end{align}
For $c \in \R^n$, we define the translation $T_c: \U(\R^n) \to \U(\R^n)$ by $\gamma=\sum_{x \in \gamma}\delta_x \mapsto \gamma+c:=\sum_{x \in \gamma}\delta_{x+c}$. A Borel probability measure~$\mu$ on $\U(\R^n)$ is called {\it translation-invariant (or stationary)} if 
\begin{align} \label{d:TI}
(T_c)_\#\mu=\mu \cquad c \in \R^n \fstop
\end{align}
If $\QP$ is translation-invariant, there exists a constant $c_\QP \in \R_+\cup\{+\infty\}$ such that the intensity measure is the Lebesgue measure in~$\R^n$ multiplied by $c_\QP$.
The law $\sine_\beta$ is translation-invariant and $c_\QP=\frac{1}{2\pi}$ for every~$\beta>0$, see \cite{ValVir09}.

\paragraph{Poisson measure} Let~$\nu$ be a Radon measure on~$X$ with~$\nu(X)<+\infty$.  {\it The Poisson measure} $\pi_{\nu}$ on~$\U=\U(X)$ with intensity~$\nu$ is defined in terms of the symmetric tensor measure $\nu^{\odot k}$ as follows:
\begin{align} \label{d:PS}
\pi_{\nu}(\cdot):=e^{-\nu(X)}\sum_{k=0}^\infty \nu^{\odot k}(\cdot)=e^{-\nu(X)}\sum_{k=0}^\infty \frac{1}{k!}(\quot_k)_\#\nu^{\otimes k}(\cdot) \comma
\end{align}
where $\nu^{\odot 0}=\delta_{0}$ is the Dirac measure on the element~$\gamma\equiv 0$.
When $\nu(X)=+\infty$ and $\nu(B)<+\infty$ for every bounded Borel set~$B$, {\it the Poisson measure $\pi_{\nu}$ on~$\U$ with intensity~$\nu$} is defined as the projective limit of $\{\pi_{\nu_{B}}: B \in \mathscr B(X) \ \text{bounded}\}$ with the projection~$\pr_{B}: \U \to \U(B)$ defined as $\gamma \mapsto \pr_B(\gamma)=\gamma_B$, i.e., $\pi_\nu$ is the unique Borel probability measure on~$\U$ such that 
$$(\pr_{B})_{\#} \pi_{\nu}:=\pi_{\nu_{B}} \cquad \text{for every bounded~$B\in\mathscr B(X)$} \comma$$
where $\nu_{B}:=\nu\mrestr{B}$ is the restriction of the measure $\nu$ on $B$.

\subsection{Extended distances in $\U$}
We introduce an extended distance $\bar\mssd_\U$ called {\it $L^2$-transportation-type distance} on the configuration space~$\U$. 

\paragraph{$L^2$-transportation-type distance} Let $(X, \mssd)$ be a locally compact complete separable metric space. For~$i=1,2$ let~$\proj_i\colon X^{\times 2}\rar X$ denote the projection to the~$i^\text{th}$ coordinate for $i=1,2$. 
For~$\gamma,\eta\in \dUpsilon$, let~$\Cpl(\gamma,\eta)$ be the set of all couplings of~$\gamma$ and~$\eta$, i.e., 
\begin{align*}
\Cpl(\gamma,\eta)\eqdef \set{\cpl\in \Meas(X^{\tym{2}}) \colon (\proj_1)_\pfwd \cpl =\gamma \comma (\proj_2)_\pfwd \cpl=\eta} \fstop
\end{align*}
Here $\Meas(X^{\tym{2}})$ denotes the space of all Radon measures on $X^{\tym{2}}$. 
The \emph{$L^2$-transportation}  \emph{extended distance} on~$\dUpsilon(X)$ is
\begin{align}\label{eq:d:W2Upsilon}
\mssd_{\dUpsilon}(\gamma,\eta)\eqdef \inf_{\cpl\in\Cpl(\gamma,\eta)} \paren{\int_{X^{\times 2}} \mssd^2(x,y) \diff\cpl(x,y)}^{1/2}\comma \qquad \inf{\emp}=+\infty \fstop
\end{align}
We introduce a variant of the $L^2$-transportation extended distance, called \emph{$L^2$-transportation-type}  \emph{extended distance}~$\bar{\mssd}_\U$ defined as 
\begin{align} \label{eq:dW2L}
\bar{\mssd}_\U(\gamma, \eta):=
\begin{cases}
\mssd_\U(\gamma, \eta) \quad &\text{if $\gamma_{B_r^c}=\eta_{B_r^c}$ for some $r>0$\ ,}
\\
+\infty \quad  &\text{otherwise} \fstop
\end{cases}
\end{align}
By definition, $\mssd_\U \le \bar{\mssd}_\U$ on $\U$,  and  $\mssd_\U = \bar{\mssd}_\U$ on $\U(B_r)$ for every~$r>0$. In particular, we have the following relation regarding the space of Lipschitz functions:
\begin{align} \label{e:LLR}
\Lip(\U, \mssd_\U) \subset \Lip(\U, \bar{\mssd}_\U) \comma \quad \Lip_{\bar{\mssd}_\U}(u) \le  \Lip_{\mssd_\U}(u)\comma  \quad u \in \Lip(\U, \mssd_\U)\comma
\end{align}
where $\Lip(\U, \mssd_\U)$ (resp.~$\Lip(\U, \bar{\mssd}_\U)$) denotes the space of Lipschitz functions with respect to~$\mssd_\U$ (resp.~$\bar\mssd_\U$) and $\Lip_{{\mssd}_\U}(u)$ (resp.~$\Lip_{\bar{\mssd}_\U}(u)$) is the Lipschitz constant with respect to $\mssd_\U$ (resp.~$\bar\mssd_\U$), see~\eqref{d:LS}.
It can be readily seen that 
\begin{align} \label{e:LLR2}
\bar{\mssd}_\U(\gamma, \eta) <+\infty \quad \iff \quad \gamma_{B_r^c}=\eta_{B_r^c} \comma \gamma(B_r)=\eta(B_r) \quad \text{for some $r>0$}\fstop
\end{align}
When we work with the configuration space over~the $n$-dimensional Euclidean space~$\R^n$ or over any Polish subset in $\R^n$, we always choose the Euclidean distance~$\mssd(\mathbf x, \mathbf y)=\bigl(\sum_{i=1}^n|x_i-y_i|^2)^{1/2}$ for $\mathbf x=(x_i)_{i=1}^n$ and $\mathbf y=(y_i)_{i=1}^n$, and the notation~$\mssd_\U$  (resp.~$\bar{\mssd}_\U$) always means  the $L^2$-transportation distance~(resp.~the $L^2$-transportation-type distance) associated with~the cost~$\mssd^2$.

\paragraph{Properties of $\mssd_\U$ and $\bar\mssd_\U$} In the following, we summarise relevant properties of the extended distances~$\mssd_\U$ and $\bar\mssd_\U$. 
\begin{rem}(properties of $\mssd_\U$ and $\bar\mssd_\U$)\label{r:ped} Let $X=\R^n$, $\mssd$ be the standard Euclidean distance in $\R^n$, and $\QP$ be a Borel probability measure on $\U$. 
\begin{enumerate}[(a)]
\item $\mssd_\U$ is a complete extended distance, $\tau_\mrmv^{\times 2}$-lower semicontinuous, and generates a stronger topology~$\tau_{\mssd_\U}$ than~the vague topology~$\tau_\mrmv$. See~\cite[Lem.~4.1]{RoeSch99}; 

\item $\bar\mssd_\U$ is complete, and the function $\bar\mssd_\U$ is $\mathscr B(\tau_\mrmv^{\times 2})$-measurable.  However, $\bar\mssd_\U$ is not $\tau_\mrmv^{\times 2}$-lower semicontinuous;

\item Both $\Lip_b(\U, \mssd_\U, \QP)$ and $\Lip_b(\U, \bar\mssd_\U, \QP)$ are dense in $L^2(\U, \QP)$. The density of $\Lip_b(\U, \mssd_\U, \QP)$ follows from e.g.,\cite[Prop.\ 4.1]{AmbGigSav14}. The density of $\Lip_b(\U, \bar\mssd_\U, \QP)$ follows by the inclusion~$\Lip_b(\U, \mssd_\U, \QP) \subset \Lip_b(\U, \bar\mssd_\U, \QP)$;

\item \label{r:ped4}Let $X=\R$. If $\QP$ is translation-invariant, then $\mssd_\U(\cdot, \gamma)=+\infty$ $\QP$-a.e.~for $\gamma=\sum_{x \in \Z}\delta_{x}$. In particular, this holds for $\sine_\beta$ for every $\beta>0$. The same holds for $\bar\mssd_\U$ as $\mssd_\U \le \bar\mssd_\U$.
 \end{enumerate}
\end{rem}
\begin{proof} We only prove (b) and (d). 
\\
(b): The completeness immediately follows by $\mssd_\U \le \bar\mssd_\U$ and the completeness of $\mssd_\U$. The measurability is due to~\cite[Prop.~2.2]{Suz24}. We show the non-lower semi-continuity. Take any pair $(\gamma, \eta) \in \U^{\times 2}$ with $\mssd_\U(\gamma, \eta)<+\infty$ such that $\gamma_{B_r^c} \neq \eta_{B_r^c}$, $\gamma(B_r)=\eta(B_r)$ for every $r>0$, and $\lim_{r\to \infty}\mssd_{\U}(\gamma_{B_r}, \eta_{B_r}) = \mssd_\U(\gamma, \eta)$. It holds that $\gamma_{B_r} \xrightarrow{\tau_\mrmv} \gamma$ and $\eta_{B_r} \xrightarrow{\tau_\mrmv} \eta$ as $r \to \infty$.
But, since $\bar\mssd_{\U}(\gamma, \eta)=+\infty$ due to \eqref{e:LLR2} and $\gamma_{B_r^c} \neq \eta_{B_r^c}$ for every~$r>0$, we have
$$\lim_{r\to \infty}\bar\mssd_{\U}(\gamma_{B_r}, \eta_{B_r})= \lim_{r\to \infty}\mssd_{\U}(\gamma_{B_r}, \eta_{B_r}) = \mssd_\U(\gamma, \eta)<\bar\mssd_{\U}(\gamma, \eta)=+\infty \fstop$$
(d): Let $B_\gamma=\{\eta \in \U: \mssd_{\U}(\gamma, \eta)<+\infty\}$ be the $\mssd_{\U}$-accessible component of $\gamma$. For $c \neq c' \in [0, 1)$, the sets $B_{\gamma+c}$ and $B_{\gamma+{c'}}$ are disjoint  because $\mssd_{\U}({\gamma+c, \gamma+{c'}})=+\infty$. By the translation-invariance, $\QP(B_\gamma)= \QP(B_{\gamma +c})$ for every $c \in \R$. If $\QP(B_\gamma)=m>0$, then $\{B_{\gamma+c}\}_{c \in [0,1)}$ is a family of uncountably many disjoint sets, each of which has a positive measure $\QP(B_{\gamma+c})=m$. This contradicts the hypothesis that $\QP$ is a probability measure.
\end{proof}

The statement~\ref{r:ped4} in Rem.~\ref{r:ped} shows that the extended distance~$\bar\mssd_{\U}(\cdot, \gamma)$  from a point can take $+\infty$ very often from the measure-theoretic viewpoint. However, if we see the distance $\gamma \mapsto \bar\mssd_\U(\gamma, \Lambda)=\inf_{\eta \in \Lambda}\bar\mssd_\U(\gamma, \eta)$  from a set $\Lambda \subset \U$, it recovers the finiteness and provides a non-trivial Lipschitz function.
\begin{ese}[Example and counterexample of Lipschitz functions] \label{r:LT} We assume the same setting as Rem.~\ref{r:ped}, and give examples and counterexamples of non-trivial functions in~$\Lip(\U, \bar\mssd_\U, \QP)$.
\begin{enumerate}[(a)]
\item Let $\U^{\infty}:=\{\gamma \in \U: \gamma(\R^n)=+\infty\}$. Fix an arbitrary open metric ball $U=B_r \subset \R^n$ with $r>0$, and take $\eta \in \U^\infty$. Define $\Lambda_{\eta, U}:=\{\gamma \in \U: \gamma_U=\eta_U\}$. The map
 $$ \U \ni \gamma \mapsto \mssd_{\U}(\gamma, \Lambda_{\eta, U}) :=\inf_{\zeta \in \Lambda_{\eta, U}}\mssd_\U(\gamma, \zeta) \in [0,+\infty]$$
 is $\tau_\mrmv$-continuous and $\mssd_\U(\cdot, \Lambda_{\eta, U}) \in \Lip(\U, \mssd_\U, \QP)$, therefore, also $\mssd_\U(\cdot, \Lambda_{\eta, U}) \in \Lip(\U, \bar\mssd_\U, \QP)$.  
Furthermore,
 $$\mssd_{\U}(\gamma, \Lambda_{\eta, U})<+\infty \cquad \gamma \in \U^\infty \fstop$$
 See \cite[Lem.~4.2]{RoeSch99} when $X$ is a complete Riemannian manifold, and \cite[Prop.~4.29]{LzDSSuz21} when $X$ is a metric space with local structure.

\item Let $\Lambda \in \mathscr B^*(\U)$ be a universally measurable set in $\U$. The map 
$$\U \ni \gamma \mapsto \bar\mssd_\U(\gamma, \Lambda):=\inf_{\eta \in \Lambda}\bar\mssd_\U(\gamma, \eta)$$ 
is universally measurable (\cite[Prop.~2.3]{Suz23}) and $\bar\mssd_\U(\cdot, \Lambda) \in \Lip(\U, \bar\mssd_\U, \QP)$. Furthermore, for $\Lambda \in \mathscr B^*(\U)$ with $\QP(\Lambda)>0$, 
 \begin{align}\label{e:FD}
 \bar\mssd_\U(\gamma, \Lambda)<+\infty \quad \text{$\QP$-a.e.~$\gamma$ }
 \end{align}
 under the assumption that $\QP$ is number rigid and  tail-trivial, see~\cite[Thm.~I]{Suz23}. In particular, \eqref{e:FD} holds for, e.g., $\sine_2$. 

\item (Cylinder functions are not Lipschitz).
Let $X=\R$. For $u \in C_0(\R)$, we define $u^*(\gamma)=\sum_{x \in \gamma}u(x)$.
We denote by ${\sf Cyl}(\U)$ the space of {\it cylinder functions} $U: \U \to \R$:
\begin{align}\label{d:CLF}
U=\Phi(u_1^*, \ldots, u_k^*) \comma \  \{u_1,\ldots, u_k\} \subset C_0^\infty(\R) \comma \  \Phi \in C_b^\infty(\R^k) \comma \  k \in \N \fstop
\end{align}
For each fixed $k \in \N_0$ and $r>0$, it is easy to see $U|_{\U^k(B_r)} \in \Lip(\U^k(B_r), \mssd_{\U})$. 
In~\cite[Example 4.35]{LzDSSuz21}, however, it was shown that ${\sf Cyl}(\U) \not \subset \Lip(\U, \mssd_{\U})$: the idea is to construct a function $U \in {\sf Cyl}(\U)$ whose square field $\cdc^{\U}(U)$ (see Dfn.~\ref{d:DFF}) does not belong to $L^\infty(\QP)$, where $\QP=\pi$ is the Poisson measure whose intensity measure is the Lebesgue measure in~$\R$. This shows that the function~$U$ cannot be $\mssd_\U$-Lipschitz due to the Rademacher-type property $\cdc^{\U}(u) \le \Lip_{\mssd_\U}(u)^2$ for every $u \in \Lip_b(\U, \mssd_\U)$, see~\cite[Thm.~1.3]{RoeSch99}.
The same argument also applies to $\bar\mssd_\U$ by the Rademacher-type property with respect to $\bar\mssd_\U$, see~Prop.~\ref{p:DF}.
 \end{enumerate}
\end{ese}

Here, we prove a Lipschitz contraction property of the operator $(\cdot)_r^\eta$ defined in~\eqref{e:SEF}. 
\begin{lem}\label{l:SEF3}
Let $u \in \Lip(\U, \bar{\mssd}_\U)$. Then, $u_r^\eta \in \Lip(\U(B_r), \mssd_\U)$ and 
\begin{align} \label{e:SEF3}
\Lip_{\mssd_\U}(u_r^\eta) \le \Lip_{\bar{\mssd}_\U}(u) \comma \quad \eta \in \U\comma \quad r>0\fstop
\end{align}
\end{lem}
\begin{proof}
Let $\gamma, \zeta \in \U(B_r)$ and $\eta \in \U$. By the definition of $\bar\mssd_\U$, 
\begin{align*}
|u_{r}^\eta(\gamma)-u_{r}^\eta(\zeta)|=|u(\gamma+\eta_{B_r^c})-u(\zeta+\eta_{B_r^c})| &\le \Lip_{\bar{\mssd}_\U}(u) \bar{\mssd}_\U(\gamma+\eta_{B_r^c}, \zeta+\eta_{B_r^c}) 
\\
&= \Lip_{\bar{\mssd}_\U}(u) \mssd_\U(\gamma, \zeta) \fstop \qedhere
\end{align*}
\end{proof}

\section{Curvature bound for conditioned particle systems} \label{sec: Pre}
In this section, we work on~$\U=\U(X)$ with $X=\R$. Let $\QP=\sine_\beta$ for $\beta>0$, which is a Borel probability measure~on $\U$, see the second paragraph in~\S\ref{s:Int} for the definition.  Recall that $\QP_r^\eta$ is the projected conditional probability~\eqref{r:BMP}. Define the measure~$\QP_r^{k, \eta}:=\QP_r^\eta\mrestr{\U^k(B_r)}$ restricted in the $k$-particle configuration space~$\U^k(B_r)$ over the closed interval~$B_r=[-r, r]$. In this section, we construct Dirichlet forms with the symmetrising measure~$\QP_r^{k, \eta}$. We denote by $\mssm$  and $\mssm_r=\mssm\mrestr{B_r}$ the Lebesgue measure on~$\R$ and its restriction on~$B_r$ respectively, and by $\mssd(x, y):=|x-y|$ the standard Euclidean distance for $x, y \in \R$. For $\mathbf x=(x_i)_{i=1}^k$ and $\mathbf y=(y_i)_{i=1}^k$, the standard Euclidean distance (i.e., the $\ell_2$-product distance) in $\R^k$ is denoted by $\mssd^{\times k}(\mathbf x, \mathbf y)^2:=\sum_{i=1}^k|x_i-y_i|^2$.

\subsection{Construction of Dirichlet forms on $\U^k(B_r)$} For $k \in \N$, 
let $W^{1,2}_s(B_r^{\times k},  \mssm_r^{\otimes k})$ be the space of~$(1,2)$-Sobolev and {symmetric} functions on the product space $B_r^{\times k}$, i.e., 
$$W^{1,2}_s(B_r^{\times k}, \mssm_r^{\otimes k}):=\biggl\{u \in L^2_s(B_r^{\times k}, \mssm_r^{\otimes k}): \int_{B_r^{\times k}} |\nabla^{\otimes k} u|^2 \diff \mssm_r^{\otimes k} <+\infty \biggr\} \comma$$
where $\nabla^{\otimes k}$ denotes the weak derivative on $\R^{\times k}$: $\nabla^{\otimes k}u:=(\partial_1 u, \ldots, \partial_ku)$.   
As the space $W^{1,2}_s(B_r^{\times k}, \mssm_r^{\otimes k})$ consists of symmetric functions, the projection $\quot_k: B_r^{\times k} \to \U^k(B_r) \cong B_r^{\times k} /\mathfrak S_k$ acts on $W^{1,2}_s(B_r^{\times k}, \mssm_r^{\otimes k})$. The associated quotient $(1,2)$-Sobolev space is denoted by $W^{1,2}(B_r^{\times k}, \mssm_r^{\odot k})$, which is the $(1,2)$-Sobolev space on $\U^k(B_r)$:
$$W^{1,2}(\U^k(B_r), \mssm_r^{\odot k}):=\biggl\{u \in L^2(\U^k(B_r), \mssm_r^{\odot k}): \int_{\U^k(B_r)} |\nabla^{\odot k} u|^2 \diff \mssm_r^{\odot k} <+\infty \biggr\} \comma$$
where $\nabla^{\odot k}$ is the quotient operator of the weak gradient operator $\nabla^{\otimes k}$ through the projection $\quot_k$, and $\mssm_r^{\odot k}$ is the symmetric product measure defined as 
$$\mssm_r^{\odot k}:=\frac{1}{k!} (\quot_k)_\#\mssm_r^{\otimes k} \fstop$$
When $k=0$, $\U^0(B_r)$ is a one-point set consisting of $\gamma\equiv 0$ and $\mssm_r^{\odot 0}=\delta_0$ is the Dirac measure on~$\gamma \equiv 0$. We set $\nabla^{\odot 0}u \equiv 0$, so $W^{1,2}(\U^0(B_r), \mssm_r^{\odot 0}) = L^{2}(\U^0(B_r), \mssm_r^{\odot 0}) \cong~\R$.

\paragraph{Weighted Sobolev spaces} We construct a weighted Sobolev space on $B_r^{\times k}$ whose reference measure is the projected conditional probability~$\QP_r^{\eta}$ on~$\U(B_r)$. 
Thanks to~the DLR equation~\cite[Thm. 1.1]{DerHarLebMai20}, the measure $\QP_r^{\eta}$ has the density with respect to $\mssm_r^{\odot k}$:
for $\mu$-a.e.~$\eta$,  there exists $k=k(\eta) \in \N_0$  so that 
\begin{itemize}
\item the {\it number rigidity} holds:
\begin{align} \label{e:R1}
\text{$\mu_{r}^{\eta}(\U^l(B_r))>0$ if and only if $l=k(\eta)$} \ ;
\end{align}
\item the {\it Dobrushin--Lanford--Ruelle (DLR)} equation holds:  
\begin{align} \label{d:cp}
 \mu_{r}^{\eta}(A) &= \mu_{r}^{k, \eta}\bigl(A \cap \U^k(B_r)\bigr)=\int_{A \cap \U^k(B_r)} \frac{e^{-\Psi^{k, \eta}_{r}}}{Z_{r}^{k, \eta}} \diff \mssm_r^{\odot k} \cquad A \in \mathscr B\bigl(\U(B_r)\bigr) \comma
\end{align}
and the Hamiltonian~$\Psi^{k, \eta}_{r}$ has the following expression for $\gamma=\sum_{i=1}^k \delta_{x_i} \in \U(B_r)$:
\begin{align}
 \Psi_{r}^{k, \eta}(\gamma) &:=  - \lim_{R \to \infty}\Psi_{r, R}^{k, \eta}(\gamma) \notag
 \\
 &:= - \lim_{R \to \infty} \log \Biggl(  \prod_{i<j}^k |x_i-x_j|^\beta \prod_{i=1}^k\prod_{y \in \eta_{B_r^c}, |y| \le R} \Bigl|1-\frac{x_i}{y}\Bigr|^\beta\Biggr) \comma \notag
\end{align}
where $Z_r^{k, \eta}:=\int_{\U^k(B_r)}e^{-\Psi_r^{k, \eta}} \diff \mssm_r^{\odot k}$ is the normalising constant. The limit~$\Psi_{r, R}^{k, \eta}(\gamma)\xrightarrow{R \to \infty} \Psi_{r}^{k, \eta}(\gamma)$ exists for~$\QP$-a.e.~$\eta$, every $r>0$ and every~$\gamma \in \U(B_r)$. 
\end{itemize}
We note that the DLR equation~\eqref{d:cp} holds true also when $k(\eta)=0$. In this case, both the LHS and the RHS in~\eqref{d:cp}  are equal to~the Dirac measure~$\delta_0$ on $\gamma \equiv 0$ in $\U(B_r)$. 
We define the following weighted energy: for $k \in \N_0$ and for  $u, v \in \Lip_b(\U^k(B_r), \mssd_{\U})$,
\begin{align} \label{eq:form2}
\E^{\U^k(B_r), \mu_{r}^{k, \eta}}(u)&:= \frac{1}{2}\int_{\U^k(B_r)} |\nabla^{\odot k} u|^2 \diff \mu_{r}^{k, \eta}  \comma
\\
\E^{\U^k(B_r), \mu_{r}^{k, \eta}}(u, v)&:=\frac{1}{4} \biggl(\E^{\U^k(B_r), \mu_{r}^{k, \eta}}(u+v) - \E^{\U^k(B_r), \mu_{r}^{k, \eta}}(u-v)\biggr) \fstop \notag
\end{align}
Using the DLR equation~\eqref{d:cp}, we prove that the form~\eqref{eq:form2} is closable, and the closure (i.e., the smallest closed extension) is a strongly local regular Dirichlet form.
\begin{prop}\label{p:form2}
Let $\QP=\sine_\beta$ for $\beta>0$. For every $k \in \N_0$, the form~\eqref{eq:form2} is well-defined and closable for $\mu$-a.e.~$\eta$. The closure is a strongly local symmetric regular Dirichlet form on~$L^2(\U^k(B_r), \mu_{r}^{k, \eta})$ and its domain is denoted by $\dom{\E^{\U^k(B_r), \mu_{r}^{k, \eta}}}$. 
\end{prop}
\begin{proof}
The case~$k=0$ is trivial. Suppose $k \ge 1$. 
As $e^{-\Psi_{r, R}^{k, \eta}} \xrightarrow{R \to \infty} e^{-\Psi_{r}^{k, \eta}}$ uniformly on $\U^k(B_r)$ for $\mu$-a.e.~$\eta$ due to \cite[Lem.~2.3 and Proof of Thm.\ 2.1 in~p.~183]{DerHarLebMai20}, the density $e^{-\Psi_{r}^{k, \eta}}$ is continuous and bounded on $\U^k(B_r)$, hence 
the well-definedness follows by the following inequality:
\begin{align*} 
 \int_{\U^k(B_r)} |\nabla^{\odot k} u|^2 \diff \mu_{r}^{k, \eta} \le\Bigl\|e^{-\Psi_{r}^{k, \eta}}\Bigr\|_{L^\infty(\U^k(B_r), \mu_{r}^{k, \eta})}\int_{\U^k(B_r)} |\nabla^{\odot k} u|^2 \diff \mssm_r^{\odot k}<+\infty \fstop
\end{align*}
The closability follows by the continuity of the density~$e^{-\Psi_{r}^{k, \eta}}$ on $B_r^{\times k}$ and the standard Hamza-type argument, see e.g., \cite[Lem.~3.2]{Osa96}. 
Via the quotient map~$\mathsf P_k: B_r^{\times k} \to \U^k(B_r) \cong B_r^{\times k}/\mathfrak S(k)$,  the symmetry, the strong locality and the Markovian property of $\E^{\U(B_r), \mu_{r}^{k, \eta}}$ descend  from the corresponding properties of the following bilinear form on the product space~$B_r^{\times k}$:  
\begin{align}\label{c:CB}\E^{B_r^{\times k}, \mu_{r}^{k, \eta}}:=\frac{1}{2}\int_{B_r^{\times k}} |\nabla^{\otimes k} u|^2 e^{-\Psi_{r}^{k, \eta}} \diff \mssm_r^{\otimes k} \cquad u \in \Lip_{b, s}(B_r^{\times k}, \mssd^{\times k}) \comma
\end{align}
where $\Lip_{b, s}(B_r^{\times k}, \mssd^{\times k})$ is the space of symmetric bounded $\mssd^{\times k}$-Lipschitz functions.
These properties extend to the closure (e.g., \cite[Thm.~3.1.1,~3.1.2]{FukOshTak11}). The regularity is straightforward as $\Lip_b(\U^k(B_r), \mssd_{\U})$ is dense in $C(\U^k(B_r), \tau_\mrmv)=C_0(\U^k(B_r), \tau_\mrmv)$ as well as in $\dom{\E^{\U(B_r), \mu_{r}^{k, \eta}}}$. 
\end{proof}

\subsection{Curvature bound for conditioned particle systems} \label{sec:CBFPS}
We show that the interaction potential~$\Psi_{r}^{k, \eta}$ defined in~\eqref{d:cp} is geodesically convex in $(\U^k(B_r), \mssd_{\U})$. 
\begin{prop} \label{p:conv}
$\Psi_{r}^{k, \eta}$ is geodesically convex in $(\U^{k}(B_r), \mssd_{\U})$ for every $0<r<+\infty$, $k \in \N_0$ and $\eta \in \U(B_r^c)$.
\end{prop}
\begin{proof}
The case~$k=0$ is trivial as $\U^k(B_r)$ is a one-point set~$\{\gamma \equiv 0\}$. Suppose $k \ge 1$. 
Recall that, for $\gamma=\sum_{i=1}^k \delta_{x_i}$,
\begin{align} \label{eq: conv}
\Psi_{r, R}^{k, \eta}(\gamma) = -\beta\sum_{i<j}^k \log (|x_i-x_j|) - \beta\sum_{i=1}^k \sum_{y \in \eta_{B_r^c}, |y| \le R} \log\Bigl|1-\frac{x_i}{y}\Bigr| \fstop
\end{align} 
Let $H_{ij}, H_i^y$ be the Hessian matrices of the functions $(x_1, \ldots, x_k) \mapsto -\log |x_i-x_j|$ and~$(x_1, \ldots, x_k) \mapsto -\log|1-\frac{x_i}{y}|$ respectively. For every vector~$\mathbf v=(v_1, \ldots, v_k) \in \R^{k}$, 
\begin{align} \label{e:HCP}
\mathbf v H_{ij} \mathbf v^t = \frac{(v_i-v_j)^2}{|x_i-x_j|^2}, \quad  \mathbf v H^y_{i} \mathbf v^t = \frac{v_i^2}{|y-x_i|^2} \fstop
\end{align}
Both $H_{ij}$ and $H_i^y$ are, therefore,  positive semi-definite. 
Thus,  for every $0<r<R$, $y \in [-R, -r] \cup [r, R]$ and $i, j \in  \{1, 2, \ldots, k\}$ with $i<j$, the functions~$(x_1, \ldots, x_k) \mapsto -\log |x_i-x_j|$ and~$(x_1, \ldots, x_k) \mapsto -\log|1-\frac{x_i}{y}|$ are geodesically convex  in the following closed convex space:
$$W_{r}^k:=\bigl\{(x_1, \ldots, x_k) \in B_r^{\times k}: x_{1} \ge x_{2} \ge \cdots \ge x_{k}\bigr\}  \fstop$$ 
Note that, if $u_1, \ldots, u_k$ are geodesically convex and $\alpha_1,\ldots, \alpha_k \ge 0$, the sum~$\sum_{i=1}^k \alpha_iu_i$ is geodesically convex as well. Thus,  we obtain that that $\Psi_{r, R}^{k, \eta}$ is also geodesically convex on $W_{r}^k$. 
By the isometry $(W_{r}^k, \mssd^{\times k}) \cong (\U^k(B_r), \mssd_\U)$, the function~$\Psi_{r, R}^{k, \eta}$ is geodesically convex in~$(\U^k(B_r), \mssd_\U)$ as well. 
Since $\Psi_{r, R}^{k, \eta} \xrightarrow{R \to \infty}\Psi_{r}^{k, \eta}$ pointwise in $\U^k(B_r)$ for $\mu$-a.e.~$\eta$ by \cite[Thm.~1.1]{DerHarLebMai20},  the limit~$\Psi_{r}^{k, \eta}$ is geodesically convex in~$(\U^k(B_r), \mssd_\U)$ as well. The proof is complete.
\end{proof}

Thanks to Prop.~\ref{p:conv}, the Dirichlet form $\bigl(\E^{\U(B_r), \mu_{r}^{k, \eta}}, \dom{\E^{\U(B_r), \mu_{r}^{k, \eta}}}\bigr)$ satisfies the Riemannian Curvature Dimension condition~$\RCD(0,\infty)$. 
\begin{prop} \label{p:BE2} 
Let $\QP=\sine_\beta$ for $\beta>0$. For every $0<r<+\infty$ and $\mu$-a.e.~$\eta \in \U$, 
the metric measure space $(\U^k(B_r), \mssd_\U, \mu_{r}^{k, \eta})$ satisfies $\RCD(0, \infty)$ with $k=k(\eta) \in \N_0$ as in \eqref{e:R1}. Furthermore,  
$$\bigl(\E^{\U^k(B_r), \mu_{r}^{k, \eta}}, \dom{\E^{\U^k(B_r), \mu_{r}^{k, \eta}}} \bigr)=\bigl(\Ch^{\mssd_\U, \mu_{r}^{k, \eta}}, W^{1,2}(\U^k(B_r), \mssd_\U, \mu_{r}^{k, \eta})\bigr) \fstop$$
\end{prop}
\begin{proof}
We only discuss the case $k \ge 1$ as we have nothing to discuss for $k=0$.
As $B_r^{\times k}$ is a convex subset in $\R^{k}$, the space $(B_r^{\times k}, \mssd^{\times k}, \mssm_r^{\otimes k})$ is a geodesic subspace of~$\R^{k}$. Therefore, it satisfies $\RCD(0, \infty)$ by the Global-to-Local property of $\RCD(0,\infty)$, see~\cite[Thm.~6.20]{AmbGigSav14b}. 
The $k$-particle configuration space $(\U^k(B_r), \mssd_\U, \mssm_r^{\odot k})$ is the quotient space of~$(B_r^{\times k}, \mssd^{\times k}, \mssm_r^{\otimes k})$  with respect to the symmetric group~$\mathfrak S_k$. Thanks to~\cite[{Thm.~1.1}]{GalKelMonSos18}, the property~$\RCD(0,\infty)$ is preserved under the quotient by a compact Lie group.  {The symmetric group~$\mathfrak S_k$ is a compact Lie group as it is a finite discrete group}. Thus, we obtain that $(\U^k(B_r), \mssd_\U, \mssm_r^{\odot k})$ satisfies $\RCD(0, \infty)$ as well. 
We now take the convex closed geodesic subspace~$\U^k_\e(B_r) \subset \U^k(B_r)$ defined as
\begin{align*}
\U^k_\e(B_r)&:=\biggl\{\gamma =\sum_{i=1}^k\delta_{x_i} \in \U^k(B_r): |x_i-x_j| \ge \e, \ i, j \in \{1, \ldots, k\}\biggr\} \comma
\\
\mssd_{\U, \e}&:=\mssd_{\U}|_{\U^k_\e(B_r)\times {\U^k_\e(B_r)}}\cquad \mssm^{\odot k}_{r, \e}:=\mssm_r^{\odot k}\mrestr{\U^k_\e(B_r)} \fstop
\end{align*}
By using the Global-to-Local property of $\RCD(0,\infty)$ again, the space~$(\U^k_\e(B_r), \mssd_{\U, \e}, \mssm^{\odot k}_{r, \e})$ is $\RCD(0,\infty)$ for every~$\e>0$. 
As $\Psi_{r}^{k, \eta}: \U^k(B_r) \to \R \cup\{+\infty\}$ is geodesically convex in~$\U^k(B_r)$ by~Prop.~\ref{p:conv}, it is also geodesically convex in the convex subspace~$\U^k_\e(B_r)$. Thus, $\Psi_{r}^{k, \eta}$ is a bounded continuous and geodesically convex function on $\U^k_\e(B_r)$ taking value in~$\R$ (not taking $+\infty$).  Noting the fact that the constant multiplication (by the normalisation constant) does not change the $\RCD$ property, the weighted space $(\U_\e^k(B_r), \mssd_{\U, \e}, \mu_{r,\e}^{k, \eta})$, therefore,  satisfies $\RCD(0,\infty)$ for every~$\e>0$ by \cite[Prop.~6.21]{AmbGigSav14b}, where 
$$\diff \mu_{r,\e}^{k, \eta}:= \frac{1}{Z_{r, \e}^{k, \eta}}e^{-\Psi_{r}^{k, \eta}}\diff \mssm_r^{\odot k} \cquad Z_{r, \e}^{k, \eta}=\int_{\U^k_\e(B_r)}e^{-\Psi_{r}^{k, \eta} }\diff \mssm_r^{\odot k} \fstop$$
Noting that~$e^{-\Psi_{r}^{k, \eta}}$ is bounded and continuous in $\U(B_r)$,  we can easily show that $\mu_{r,\e}^{k, \eta} \xrightarrow{\e \to 0}\mu_{r}^{k, \eta}$ weakly as probability measures in $\U^k(B_r)$.
 Thus, by the stability of the $\RCD(0,\infty)$~condition due to~\cite[Thm.~IV]{GigMonSav15}, the limit space~$(\U^k(B_r), \mssd_{\U}, \mu_{r}^{k, \eta})$ satisfies~$\RCD(0,\infty)$ as well.

We prove the second assertion:
$$\Bigl(\E^{\U^k(B_r), \mu_{r}^{k, \eta}}, \dom{\E^{\U^k(B_r), \mu_{r}^{k, \eta}}}\Bigr)=\Bigl(\Ch^{\mssd_\U, \mu_{r}^{k, \eta}}, W^{1,2}(\U^k(B_r), \mssd_\U, \mu_{r}^{k, \eta})\Bigr) \fstop$$ 
Recall the Rademacher theorem, i.e., $\mssd^{\times k}$-Lipschitz functions are differentiable almost everywhere with respect to $\mssm_r^{\otimes k}$ on $B_r^{\times k}$. As it  descends to~ $\U^k(B_r)$ via the projection $\quot_k: B_r^{\times k} \to \U^k(B_r)$, we have 
\begin{align}\label{e:S=D}
|\mathsf D_{\mssd_\U} u|=|\nabla^{\odot k} u| \qquad \mssm_r^{\odot k}\text{-a.e.} \cquad u \in \Lip(\U^k(B_r), \mssd_\U) \fstop
\end{align} 
Recalling~\eqref{d:ISL},  the functional $\mathsf E^{\mssd_\U, \QP_r^{k, \eta}}: L^2(\U^k(B_r), \QP_r^{k, \eta}) \to \R \cup\{+\infty\}$ is defined as
\begin{align} \label{d:ISLK}
\mathsf E^{\mssd_\U, \QP_r^{k, \eta}}(u):=
\begin{cases}
\frac{1}{2}  \int_{\U^k(B_r)} |\mathsf D_{\mssd_\U} u|^2 \diff \QP_{r}^{k ,\eta}  \quad & u \in \Lip_b(\U^k(B_r), \mssd_\U) \ ;
\\
+\infty \quad& \text{otherwise} \fstop
\end{cases}
\end{align}
By \eqref{e:S=D}, we have 
\begin{align} \label{eq:BE1}
\E^{\U^k(B_r), \mu_{r}^{k, \eta}} =  \mathsf E^{\mssd_\U, \QP_r^{k, \eta}} \quad  \text{on} \quad \Lip_b(\U^k(B_r), \mssd_\U)\comma
\end{align}
and $\E^{\U^k(B_r), \mu_{r}^{k, \eta}} \le \mathsf E^{\mssd_\U, \QP_r^{k, \eta}}$ on $L^2(\U^k(B_r), \QP_r^{k, \eta})$.
By~Prop.~\ref{p:form2}, $(\E^{\U^k(B_r), \mu_{r}^{k, \eta}}, \dom{\E^{\U(B_r), \mu_{r}^{k, \eta}}})$ is the closure (i.e., the smallest closed extension) of \eqref{d:ISLK}. Thus, it coincides with the $L^2$-lower semi-continuous envelope $\Ch^{\mssd_\U, \mu_{r}^{k, \eta}}$ of \eqref{d:ISLK}, see, e.g., \cite[(e) Relaxation, p.373]{Mos94}. \qedhere
\end{proof}

\section{Curvature bound for infinite-particle systems} \label{sec:CI}
In this section, we construct a strongly local Dirichlet form on $\U=\U(\R)$ whose symmetrising measure is $\sine_\beta$, and we prove that it satisfies ~the~$\BE(0,\infty)$ gradient estimate. The structure of the proof is as follows: we first construct {\it truncated Dirichlet forms} $(\E_r^{\U, \QP}, \dom{\E_r^{\U, \QP}})$ on $\U$ whose gradient operators are truncated by configurations inside~$B_r$. We then identify them with the superposition (also called the direct integral) Dirichlet forms $(\bar\E_r^{\U, \QP}, \dom{\bar\E_r^{\U, \QP}})$ lifted from the form~\eqref{eq:form2} on $\U^k(B_r)$. The truncated Dirichlet form is used to construct the limit Dirichlet form as the monotone limit $r \to \infty$, while the superposition Dirichlet form is used to show $\BE(0,\infty)$.
 At the end of this section, we discuss several applications of the~$\BE(0,\infty)$ gradient estimate.

\subsection{Truncated Dirichlet forms}
In this subsection, we construct the truncated Dirichlet forms on~$\U$. We start with the construction of the Dirichlet forms on~$\U(B_r)$ as the countable sum over $k \in \N_0$ of the  forms \eqref{eq:form2} on $\U^k(B_r)$. 
\begin{defs}[Square field on $\U(B_r)$]\label{d:DT2} 
Fix $r>0$ and $\eta \in \U$. For a {$\mu_r^\eta$}-measurable function $u: \U(B_r) \to \R$ satisfying $u|_{\U^k(B_r)} \in \mathcal D(\E^{\U(B_r), \mu_r^{k, \eta}})$ for every $k \in \N_0$, the square field $\cdc^{\U(B_r)}(u)$ is defined as 
\begin{align} \label{d:GSF}
\Gamma^{\U(B_r)}(u) := \sum_{k=0}^\infty \Bigl|\nabla^{\odot k} \bigl(u |_{\U^k(B_r)}\bigr)\Bigr|^2  \quad (\le +\infty) \comma
\end{align}
and define the following form:
\begin{align} \label{e:FFV}
\E^{\U(B_r), \mu_r^\eta}(u)&:=\frac{1}{2}\int_{\U(B_r)} \Gamma^{\U(B_r)}(u) \diff\mu_{r}^{\eta} \comma
\\
\dom{\E^{\U(B_r), \mu_r^\eta}}&:=\Bigl\{u:  \U(B_r) \to \R: {u|_{\U^k(B_r)} \in \mathcal D(\E^{\U(B_r), \mu_r^{k, \eta}}) \quad k \in \N_0} \comma \E^{\U(B_r), \mu_r^\eta}(u)<+\infty \Bigr\} \fstop \notag
\end{align}
The form~\eqref{e:FFV} is a strongly local symmetric Dirichlet form as it is a countable sum of strongly local symmetric Dirichlet forms see e.g., \cite[Exercise~3.9 in p.31]{MaRoe90}. 
Due to the number-rigidity~\eqref{e:R1}, the Dirichlet form~$\E^{\U(B_r), \mu_r^\eta}$ is equal to $\E^{\U(B_r), \mu_r^{k, \eta}}$ for $k=k(\eta)$. 
The corresponding $L^2(\U(B_r), \QP_r^\eta)$-semigroup  is denoted by $\sem{T_t^{\U(B_r), \QP_r^\eta}}$.
\end{defs} 
\begin{rem}[{Comparison with~\cite{KawOsaTan21}}] \label{r:CK}
The form~\eqref{e:FFV} coincides with the form given in \cite[(2.45)]{KawOsaTan21}, where the domain there is the smallest closed extension of smooth functions in $\U(B_r)$ while \eqref{e:FFV}  is the smallest closed extension of Lipschitz functions $\Lip_b(\U(B_r), \mssd_\U)$ due to Prop.~\ref{p:BE2}. This identification~follows by  the standard fact that every bounded Lipschitz function can be approximated by smooth functions in the $(1,2)$-Sobolev space on $W^{1,2}(\U^k(B_r), \mssm_r^{\odot k})$  and this approximation inherits to the weighted Sobolev space 
$$(\E^{\U(B_r), \mu_r^{k,\eta}}, \dom{\E^{\U(B_r), \mu_r^{k, \eta}}})\comma$$
 which can be readily seen by the boundedness of the density $\frac{\diff\mu_r^{k,\eta}}{\diff \mssm_r^{\odot k}} \in L^\infty(\U^k(B_r), \mssm_r^{\odot k})$ for every $k \in \N_0$.   Therefore, due to the argument \cite[line 11--18, p.654]{KawOsaTan21}, the semigroup $\sem{T_t^{\U(B_r), \QP_r^\eta}}$ gives the transition probability of the unlabelled solution to the finite-particle Dyson SDE~\cite[(2.40)--(2.43)]{KawOsaTan21} with the configuration outside $B_r$ conditioned to be~$\eta_{B_r^c}$ and with the reflecting boundary condition at $\partial B_r$. 
\end{rem}

Recall that $u_r^\eta(\gamma):=u(\gamma+\eta_{B_r^c})$ for $\gamma \in \U(B_r)$ and $\eta \in \U$ was defined in~\eqref{e:SEF}.
\begin{defs}[Core]\label{d:core}For $r>0$, $\mathcal C_r$ is defined as the space of $\mu$-classes of measurable functions $u$ so that
\begin{enumerate}[$(a)$]
\item\label{i:d:core1} $u\in L^\infty(\U, \mu)$; 
\item\label{i:d:core2} $u_r^\eta \in \Lip_b(\U(B_r), \mssd_\U)$ for $\QP$-a.e.~$\eta$;

\item\label{i:d:core3}  The following integral is finite:
\begin{align} \label{eq:VariousFormsA} 
\E^{\U, \QP}_r(u):=\int_{\U} {\E^{\U(B_r), \QP_r^\eta}(u_r^\eta)} \diff\QP(\eta) <+\infty \fstop
\end{align} 
\end{enumerate}
It will be proven in Prop.~\ref{t:ClosabilitySecond} that
$\mathcal C_r$ is non-trivial in the sense that $\Lip_b(\U, \mssd_\U, \QP)$ is contained in $\mathcal C_r$, in particular, $\mathcal C_r$ is dense in $L^2(\U, \QP)$ due to (c) Rem.~\ref{r:ped}.
\end{defs}

\paragraph{Square fields of the truncated forms}
For~$u: \U \to \R$, define~$\mathcal U_{\gamma, x}(u): \R \to \R$ by 
\begin{align} \label{d:UO}
\mathcal U_{\gamma, x}(u)(y):=u\ttonde{\car_{\R\setminus\set{x}}\cdot\gamma + \delta_y}-u\ttonde{\car_{\R\setminus\set{x}}\cdot\gamma} \comma \quad \gamma \in \U,\quad x \in \R \fstop
\end{align}
The operation $\mathcal U_{\gamma, x}$ was introduced in~\cite[Lem.~1.2]{MaRoe00} to define {\it a partial derivative} in the configuration space, see also \cite[Lem.~2.16]{LzDSSuz21}.  We introduce a localised version~$\mathcal U_{\gamma, x}^r$ below. 
\begin{lem} \label{l:LSO} 
For $u: \U(B_r) \to \R$, define  $\mathcal U_{\gamma, x}^r(u): B_r \to \R$ by 
$$\mathcal U_{\gamma, x}^r(u)(y):=u(\1_{B_r \setminus \{x\}} \cdot \gamma + \delta_y) - u(\1_{B_r \setminus \{x\}} \cdot \gamma) \quad \gamma \in \U(B_r), \ x \in B_r \fstop$$
The operation $\mathcal U^r_{\gamma, x}$ maps $\Lip(\U(B_r), \mssd_\U)$ to $\Lip(B_r, \mssd)$ and Lipschitz constants are contracted by $\mathcal U^r_{\gamma, x}$:
 $$\Lip_{\mssd}(\mathcal U^r_{\gamma, x}(u)) \le \Lip_{\mssd_\U}(u) \cquad \gamma \in \U(B_r) \cquad x \in B_r \fstop$$
Furthermore, for every function~$u:  \U \to \R$, 
$$\mathcal U_{\gamma_{B_r}, x}^r(u_{r}^\gamma)(y) = \mathcal U_{\gamma, x}(u)(y) \cquad \gamma \in \U \cquad x \in B_r \cquad y \in B_r \fstop$$
\end{lem}
\begin{proof}
 Let $u \in \Lip(\U(B_r), \mssd_\U)$. Then
\begin{align*} 
|\mathcal U^r_{\gamma, x}(u)(y)- \mathcal U^r_{\gamma, x}(u)(z)| &= |u(\car_{B_r\setminus\set{x}}\cdot\gamma + \delta_y)-u(\car_{B_r\setminus\set{x}}\cdot\gamma + \delta_z)|
\\
& \le \Lip_{\mssd_{\U}}(u)\mssd_\U(\car_{B_r\setminus\set{x}}\cdot\gamma + \delta_y, \car_{B_r\setminus\set{x}}\cdot\gamma + \delta_z) \notag
\\
&= \Lip_{\mssd_{\U}}(u) \mssd(y,z) \comma \notag
\end{align*}
which concludes the first assertion. 
For every $x \in {B_r}$ and $y \in B_r$, 
\begin{align*}
 \mathcal U_{\gamma, x}(u)(y)& = u(\1_{\R \setminus \{x\}} \cdot \gamma + \delta_y)- u(\1_{\R \setminus \{x\}} \cdot \gamma)
 \\
 &  = u(\1_{B_r \setminus \{x\}} \cdot \gamma_{B_r} + \gamma_{B_r^c}+ \delta_y)- u(\1_{B_r \setminus \{x\}} \cdot \gamma_{B_r} + \gamma_{B_r^c})
 \\
 &= u_{r}^{\gamma}(\1_{B_r \setminus \{x\}} \cdot \gamma_{B_r}+ \delta_y)- u_{r}^{\gamma}(\1_{B_r \setminus \{x\}} \cdot \gamma_{B_r})
 \\
 &=\mathcal U^r_{\gamma_{B_r}, x}(u_{r}^\gamma)(y) \fstop
\end{align*}
The proof is complete.
\end{proof}

We now define the square field operator on $\U$ truncated by particles inside $B_r$. {To do so, we make use of a strong Borel lifting operator~$\ell: L^\infty(B_r, \mssm_r) \to \mathcal B_b(B_r)$. The lifting chooses Borel representatives of elements in~$L^\infty(B_r, \mssm_r)$ in such a way that  the algebraic and order structures of $ L^\infty(B_r, \mssm_r)$ are preserved, and continuous functions are fixed, see Dfn.~\ref{d:Liftings} in Appendix. }
\begin{defs}[Truncated square field on $\U$]\label{d:DT}
Let $\ell:L^\infty(B_r, \mssm_r) \to \mathcal B_b(B_r)$ be a strong Borel lifting.
The following operator is called {\it the truncated square field}:
\begin{equation}\label{eq:d:LiftCdCRep}
\begin{gathered}
\cdc^{\dUpsilon}_{r, \ell}(u)(\gamma):= \sum_{x\in \gamma_{B_r}} {\ell}\Bigl(|\nabla \mathcal U_{\gamma, x}(u)|^2\Bigr)(x) \cquad u \in \mathcal C_r\fstop
\end{gathered}
\end{equation}
{Due to Lem.~\ref{l:LSO} and the strong lifting $\ell$, the formula~\eqref{eq:d:LiftCdCRep} is well-defined for $u \in \mathcal C_r$. Indeed, by (b) Dfn.~\ref{d:core},  we have $u_r^\gamma \in \Lip_b(\U(B_r), \mssd_\U)$ for $\QP$-a.e.~$\gamma$. By Lem.~\ref{l:LSO}, 
$$\mathcal U_{\gamma, x}(u)|_{B_r}=\mathcal U_{\gamma_{B_r}, x}^r(u_r^\gamma) \in \Lip_b(B_r, \mssd) \cquad \QP\text{-a.e.~$\gamma$,\ $x \in B_r$} \fstop$$
 Since Lipschitz functions on $B_r$ are $\mssm_r$-almost everywhere differentiable, the expression $|\nabla \mathcal U_{\gamma, x}(u)|$ is well-defined $\mssm_r$-a.e.~on $B_r$ and $|\nabla \mathcal U_{\gamma, x}(u)| \in L^\infty(B_r, \mssm_r)$.
Thanks to the strong lifting $\ell$,  the function~${\ell}\Bigl(|\nabla \mathcal U_{\gamma, x}(u)|^2\Bigr)$ is defined everywhere (as opposed to $\mssm_r$-a.e.), so that the summation in the RHS of~\eqref{eq:d:LiftCdCRep} is well-defined.  The function~$\cdc^{\dUpsilon}_{r, \ell}(u)$ {\it does} depend on the lifting $\ell$, but its $\QP$-equivalence class  {\it does not} depend on $\ell$, which will be discussed in (a) of Prop.~\ref{t:ClosabilitySecond}.}
\end{defs} 
\begin{rem}
If $|\nabla \mathcal U_{\gamma, x}(u)|$ is defined everywhere (e.g., $u$ is a cylinder function as in~\eqref{d:CLF}, or a local smooth function in the sense of \cite[(1.2)]{Osa96}),  we do not need the strong lifting $\ell$ in~\eqref{eq:d:LiftCdCRep} as there is no ambiguity of sets of measure zero. However,  for later arguments (e.g. in the proof of Thm.~\ref{t:S=M}), we need to take a sufficiently large core such that it is fixed by the action of the $L^2$-semigroup. For such a core, we need the strong lifting to obtain the concrete expression~\eqref{eq:d:LiftCdCRep}, which shall play a key role to show the monotonicity of the truncated forms in Prop.~\ref{p:mono}.
\end{rem}

The following proposition relates the two square fields $\cdc^{\U}_r$ and $\cdc^{\U(B_r)}$. 
\begin{prop}[Truncated form]\ 
\label{t:ClosabilitySecond} 
\begin{enumerate}[(a)]
{\item Let $\ell_1$ and $\ell_2$ be any two strong Borel liftings. Then, 
$$\cdc^{\dUpsilon}_{r, \ell_1}(u) = \cdc^{\dUpsilon}_{r, \ell_2}(u) \quad  \text{$\QP$-a.e.} \quad  u \in \mathcal C_r \fstop$$
We denote by $\cdc^{\dUpsilon}_{r}(u) \in L^0(\QP)$ the unique $\QP$-equivalence class represented by~$\cdc^{\dUpsilon}_{r, \ell_1}(u)$.
}
\item The following identities hold:
\begin{align}\label{eq:p:MarginalWP:0}
\cdc^{\dUpsilon}_r(u)(\gamma+\eta_{B_r^c}) &= \cdc^{\dUpsilon(B_r)}(u_{r}^\eta)(\gamma) \comma \quad \text{$\QP$-a.e.~$\eta$, \ $\mu_{r}^\eta$-a.e.~$\gamma \in \U(B_r)$}  \comma
\\
\E_{r}^{\U, \mu}(u) &=\frac{1}{2}\int_\dUpsilon \cdc^{\U}_r(u) \diff\QP\comma \quad u \in \mathcal C_r\fstop \notag
\end{align}
\item The Rademacher-type property holds: $\Lip_b(\U, \bar{\mssd}_\U, \QP)\subset \mathcal C_r$ and 
\begin{align}\label{p:Rad}
 \cdc^\U_r(u) \le \Lip_{\bar{\mssd}_\U}(u)^2 \qquad u \in \Lip_b(\U, \bar{\mssd}_\U, \QP) \fstop
\end{align}
As a consequence, the form $(\E^{\U, \QP}_{r}, \mathcal C_r)$ in~\eqref{eq:p:MarginalWP:0} is a densely defined closable  form and 
the closure~$(\E_r^{\U, \QP}, \dom{\E_r^{\U, \QP}})$ is a strongly local symmetric Dirichlet form on~$L^2(\U, \mu)$. The $L^2$-semigroups corresponding to $(\E_{r}^{\U, \mu}, \dom{\E_{r}^{\U, \mu}})$ is denoted by $\sem{T_{r, t}^{\U, \QP}}$. 
\end{enumerate}
\end{prop}

\begin{proof}
{(a) and (b): Take a strong Borel lifting $\ell$. We denote~$|\partial_{i,\ell} \cdot|^2:=\ell(|\partial_{i} \cdot|^2)$ for $i \in \N$, and define the corresponding squared field in~$B_r^{\times k}$ as 
$$|\nabla^{\otimes k}_\ell \cdot|^2:=\sum_{i=1}^k|\partial_{i,\ell} \cdot|^2 \fstop$$
We denote by $|\nabla^{\odot k}_\ell \cdot|$ the corresponding quotient square field on $\U^k(B_r)$. When $k=0$, $|\nabla^{\odot k}_\ell \cdot| \equiv 0$.
By the property~\ref{i:d:Liftings:1} in Dfn.~\ref{d:Liftings} of~the strong Borel lifting~$\ell$, we have 
$|\nabla^{\odot k}_\ell(u)|^2 = |\nabla^{\odot k}(u)|^2$ $\mssm_r^{\odot k}$-a.e.~for $k \in \N_0$. 
As $\QP_r^{k, \eta} \ll \mssm_r^{\odot k}$, we have
\begin{align}\label{e:IUL}
\Bigl|\nabla^{\odot k}_\ell \bigl(u \bigr)\Bigr|^2 = \Bigl|\nabla^{\odot k} \bigl(u\bigr)\Bigr|^2\quad \text{$\QP_r^{k, \eta}$-a.e.} \cquad k \in \N_0 \fstop
\end{align}
}We note that $\QP_r^\eta$ is concentrated on the set of no multiple points, i.e., $\QP_r^\eta(\U_{\le 1}(B_r))=1$, where $\U_{\le 1}(B_r):=\{\gamma \in \U(B_r): \gamma(\{x\}) \in \{0,1\},\  x \in B_r\}$ because $\mu_r^\eta$ is absolutely continuous with respect to the Poisson measure~$\pi_{\mssm_r}$ on $\U(B_r)$ and the Poisson measure does not have multiple points almost surely. Thus, by the definition of the symmetric gradient operator $\nabla^{\odot k}$, it can be readily checked that 
\begin{align}\label{e:ESU}
\sum_{k =0}^\infty \Bigl|\nabla^{\odot k}_\ell \bigl(u \bigr)\Bigr|^2(\gamma)  = \sum_{x\in\gamma_{B_r}} \Bigl|\nabla_\ell u \tparen{\car_{B_r \setminus \set{x}} \cdot\gamma_{B_r}+\delta_\bullet}\Bigr|^2(x) \cquad \text{$\QP_r^\eta$-a.e.~$\gamma$} \fstop
\end{align}
Thus, for every $\phi \in L^2(\U, \QP)$,  
\begin{align} \label{e:ILT}
&\frac{1}{2} \int_{\U} \cdc^\dUpsilon_{r, \ell}(u) (\gamma) \phi(\gamma) \diff \QP(\gamma)
\\
&= \frac{1}{2}\int_{\U} \biggl( \sum_{x\in\gamma_{{B_r}}} \Bigl|\nabla_\ell \Bigl( u\tparen{\car_{\R \setminus \set{x}} \cdot\gamma+\delta_\bullet}-u\tparen{\car_{\R\setminus\set{x}}\cdot\gamma} \Bigr)\Bigl|^2(x) \biggr) \phi(\gamma) \diff \QP(\gamma) \notag
\\
&= \frac{1}{2}\int_{\U} \biggl(  \sum_{x\in\gamma_{{B_r}}}  \Bigl|\nabla_\ell \Bigl( u_{r}^{\gamma}\tparen{\car_{B_r \setminus \set{x}} \cdot\gamma_{B_r}+\delta_\bullet}-u^\gamma_{r}\tparen{\car_{B_r \setminus\set{x}}\cdot\gamma_{B_r}} \Bigr)\Bigr|^2(x) \biggr) \phi(\gamma) \diff \QP(\gamma) \notag
\\
&=  \frac{1}{2} \int_{\U} \biggl(\sum_{x\in\gamma_{{B_r}}} \bigl|\nabla_\ell u_{r}^{\gamma}\tparen{\car_{B_r \setminus \set{x}} \cdot\gamma_{B_r}+\delta_\bullet}\bigr|^2(x)\biggr)\phi(\gamma) \diff \QP(\gamma) \notag
\\
&=  \frac{1}{2} \int_{\U} \int_{\U(B_r)}\biggl(\sum_{x\in\gamma_{{B_r}}} \bigl|\nabla_\ell u_{r}^{\eta}\tparen{\car_{B_r \setminus \set{x}} \cdot\gamma_{B_r}+\delta_\bullet}\bigr|^2(x)\biggr)\phi_r^\eta(\gamma) \diff \QP_r^\eta(\gamma) \diff \QP(\eta) \notag
\\
&=  \frac{1}{2} \int_{\U} \int_{\U(B_r)} \sum_{k =0}^\infty \Bigl|\nabla^{\odot k}_\ell \bigl(u_r^\eta \bigr)\Bigr|^2(\gamma) \phi_r^\eta(\gamma) \diff \QP_r^{\eta}(\gamma)  \diff \QP(\eta) \notag
\\
&=  \frac{1}{2} \int_{\U} \int_{\U(B_r)} \sum_{k =0}^\infty \Bigl|\nabla^{\odot k} \bigl(u_r^\eta \bigr)\Bigr|^2(\gamma) \phi_r^\eta(\gamma) \diff \QP_r^{\eta}(\gamma) \diff \QP(\eta)  \notag
\\
&=  \frac{1}{2} \int_{\U} \int_{\U(B_r)} \cdc^{\U(B_r)}(u_r^\eta)(\gamma)\phi_r^\eta(\gamma) \diff \QP_r^{\eta}(\gamma) \diff \QP(\eta) \comma \notag
\end{align}
where the first equality is the definition of the square field~$\cdc^\dUpsilon_r$; the second equality is the definition of $u_r^\gamma$; the third equality holds as $u^\eta_{r}\tparen{\car_{B_r\setminus\set{x}}\cdot\gamma}$ does not depend on the variable denoted as~$\bullet$ on which the weak gradient~$\nabla$ operates; the fourth and the fifth equalities follow by the conditional probability formula~\eqref{p:ConditionalIntegration} and~\eqref{e:ESU} respectively. The sixth equality follows by~\eqref{e:IUL}. The last equality is the definition~\eqref{d:GSF} of $\cdc^{\U(B_r)}$. This proves that the LHS of \eqref{e:ILT} does not depend on the choice of $\ell$. As $\phi$ runs over every element in $L^2(\U, \QP)$,  the statements (a) and (b) are concluded.

(c): The strong locality and the Markov property of~$\E_{r}^{\U, \mu}$ readily follow from~\eqref{eq:VariousFormsA} and the fact that $\E^{\U(B_r), \mu_r^\eta}$ possesses the corresponding properties for $\QP$-a.e.~$\eta$. 
 We now show the Rademacher-type property: $\Lip_b(\U, \bar{\mssd}_\U, \QP) \subset \mathcal C_r$ and
\begin{align} \label{e:RD}
\cdc^\U_r(u) \le \Lip_{\bar{\mssd}_\U}(u)^2  \cquad u \in \Lip_b(\U, \bar{\mssd}_\U, \QP) \quad r>0\fstop
\end{align}
We first show $\Lip_b(\U, \bar{\mssd}_\U, \QP) \subset \mathcal C_r$. The verification of~\ref{i:d:core1} in Dfn.~\ref{d:core} is obvious. The verification of~\ref{i:d:core2} in Dfn.~\ref{d:core} follows from the Lipschitz contraction~\eqref{e:SEF3} of the operator~$(\cdot)_r^\eta$. To verify~\ref{i:d:core3} in Dfn.~\ref{d:core}, it suffices to show~\eqref{e:RD} as~$\QP$ is a probability measure.   
As the Cheeger energy~$\Ch^{\mssd_\U, \QP_r^{k, \eta}}$ coincides with the form~$\E^{\U(B_r), \QP_r^{k, \eta}}$ by Prop.~\ref{p:BE2}, in particular, the square field $\cdc^{\U(B_r)}$ coincides with the minimal relaxed slope: 
$$\cdc^{\U(B_r)}(u) =|\nabla_{\mssd_\U, \QP_r^{k, \eta}}u|^2_* \qquad \text{$\QP_r^{k, \eta}$-a.e.~} \quad u \in \dom{\E^{\U(B_r), \QP_r^{k, \eta}}} \fstop$$
Thus, by \eqref{i:MSL} and the fact that $\mssd_\U$ and $\bar\mssd_\U$ coincide when restricted in $\U(B_r)$, 
the following Rademacher-type property follows:
\begin{align}\label{e:RD2}
 \cdc^{\U(B_r)}(u) \le \Lip_{\bar{\mssd}_\U}(u)^2 \cquad  u \in \Lip(\U(B_r), \bar\mssd_\U) \quad r>0  \fstop
\end{align}
In view of the relation between~$\cdc^\U_r$ and~$\cdc^{\U(B_r)}$ in~\eqref{eq:p:MarginalWP:0} and the Lipschitz contraction~\eqref{e:SEF3} of the operator~$(\cdot)_r^\eta$, we conclude~\eqref{e:RD}.

By Rem.~\ref{r:ped}, $\Lip_b(\U, \bar\mssd_\U, \QP) \subset L^2(\U, \mu)$ is dense. By $\Lip_b(\U, \bar\mssd_\U, \QP) \subset \mathcal C_r$, the space $\mathcal C_r$ is dense in $L^2(\U, \QP)$ as well, so the form~$(\E_{r}^{\U, \mu}, \mathcal C_r)$ is densely defined. 

We now show the closability. Noting that $\E^{\U(B_r), \QP_r^\eta}$ is closable for $\mu$-a.e.~$\eta$ by Prop.~\ref{p:BE2},  the superposition form~$(\bar{\E}^{\U, \QP}_r,\dom{\bar{\E}^{\U, \QP}_r})$ (defined below in Dfn.~\ref{d:SPF}) is closable (indeed it is closed) by \cite[Prop.~V.3.1.1]{BouHir91}. As the two forms~$(\E_{r}^{\U, \mu}, \mathcal C_r)$ and~$(\bar{\E}^{\U, \QP}_r,\dom{\bar{\E}^{\U, \QP}_r})$ coincide on $\mathcal C_r$ by definition and $\mathcal C_r\subset \dom{\bar{\E}^{\U, \QP}_r}$ by construction, the closability of~$(\E_{r}^{\U, \mu}, \mathcal C_r)$ is inherited from the closedness of the superposition form~$(\bar{\E}^{\U, \QP}_r,\dom{\bar{\E}^{\U, \QP}_r})$. 
 The proof is complete.
\end{proof}
\subsection{Superposition form}
In this section, we study another Dirichlet form on $\U$ lifted from $\E^{\U(B_r), \QP_r^\eta}$, called {\it superposition Dirichlet form} (also called {\it direct integral}). 
By definition, the domain of the superposition form is larger than $\dom{\E^{\U, \QP}_r}$. We identify these two Dirichlet forms by using the stability of~$\mathcal C_r$ under the action of the semigroup $\bar{T}_{t, r}^{\U, \QP}$ associated with the superposition form. Due to this identification, we can express the $L^2$-semigroup~$T_{t, r}^{\U, \QP}$ by the superposition semigroup $\bar{T}_{t, r}^{\U, \QP}$ in Cor.~\ref{prop: 1}, which plays a key role to lift the $\BE(0,\infty)$ gradient estimate from $\E^{\U(B_r), \QP_r^\eta}$. 
\begin{defs}[Superposition Dirichlet form,~e.g., {\cite[Prop.\ V.3.1.1]{BouHir91}}] \label{d:SPF}
\begin{align} \label{eq:SP} 
\mathcal D(\bar{\E}^{\U, \mu}_r)&:= \biggl\{u \in L^2(\U, \QP): \ u_r^\eta \in \dom{\E^{\U(B_r), \QP_r^\eta}}\ \text{$\QP$-a.e.~$\eta$} \comma
\\
&\hspace{32mm} \int_\dUpsilon  \E^{\U(B_r), \QP_r^\eta}(u_{r}^\eta) \diff\QP(\eta)<+\infty  \biggr\} \comma \notag
\\
\bar{\E}^{\U, \QP}_r(u)&:=\int_\dUpsilon  \E^{\U(B_r), \QP_r^\eta}(u_{r}^\eta) \diff\QP(\eta) \fstop \notag
\end{align}
It is known that $(\bar{\E}^{\U, \QP}_r,\dom{\bar{\E}^{\U, \QP}_r})$ is a symmetric Dirichlet form on~$L^2(\U, \mu)$ \cite[Prop.\ V.3.1.1]{BouHir91}.  The $L^2$-semigroup and the infinitesimal generator corresponding to $(\bar{\E}^{\U, \QP}_r,\dom{\bar{\E}^{\U, \QP}_r})$  are denoted by $\sem{\bar{T}_{r, t}^{\U,\QP}}$ and $(\bar{A}_r^{\U, \QP}, \dom{\bar{A}_r^{\U, \QP}})$ respectively. 
\end{defs}
The semigroup $\sem{\bar{T}_{r, t}^{\U, \QP}}$ corresponding to the superposition form $\bar{\E}_r^{\U, \QP}$ can be obtained as the superposition of  the semigroup $\sem{T^{\U(B_r), \QP_r^\eta}_{t}}$ associated with the form~$\E^{\U(B_r), \QP_r^\eta}$. For the following proposition, we refer the reader to~\cite[(iii) Prop.~2.13]{LzDS20}.
\begin{prop}[{\cite[(iii) Prop.~2.13]{LzDS20}}]\label{prop: 1-1}
The following holds: 
	\begin{align} \label{eq: R-1}
		\bar{T}_{r,t}^{\U, \QP}u(\gamma) & = T^{\U(B_r), \QP_r^\gamma}_{t} u_{r}^{\gamma}(\gamma_{B_r}) \, ,
	\end{align}
	for $\QP$-a.e.\ $\gamma\in \dUpsilon$ and every $t \ge 0$.
\end{prop}

We now discuss the relation between $\E_r^{\U, \QP}$ and $\bar{\E}_r^{\U, \QP}$. As the former form is constructed as the smallest closed extension of $(\E_r^{\U, \QP}, \mathcal C_r)$, it is clear by definition that 
\begin{align} \label{i:OI}
\E_r^{\U, \QP}=\bar{\E}_r^{\U, \QP} \quad \text{on} \quad \mathcal C_r \comma \quad \dom{\E_r^{\U, \QP}} \subset \dom{\bar{\E}_r^{\U, \QP}} \fstop
\end{align}
The following theorem proves that the opposite inclusion holds as well. 
\begin{thm} \label{t:S=M}
$(\E_r^{\U, \QP}, \dom{\E_r^{\U, \QP}}) = (\bar{\E}_r^{\U, \QP}, \dom{\bar{\E}_r^{\U, \QP}})$. 
\end{thm}
\begin{proof}
In view of \eqref{i:OI} and the fact that $\dom{\E^{\U, \QP}_r}$ is the closure of~$\mathcal C_r$, 
it suffices to show that $\mathcal C_r\subset \dom{\bar{\E}^{\U, \QP}_r}$ is dense. 
Thanks to Lem.~\ref{l:MU}, it is sufficient to show $\bar{T}_{r, t}^{\U, \QP}\mathcal C_r \subset \mathcal C_r$ for every $t>0$. So, we now verify (a)--(c) in Dfn.~\ref{d:core} for~$\bar{T}_{r, t}^{\U, \QP}\mathcal C_r$.

\paragraph{Verification of (a) in Dfn.~\ref{d:core}}
Since $\bar{T}_{r, t}^{\U, \QP}$ contracts the $L^\infty$-norm by~\eqref{e:con1},  
we obtain $\bar{T}_{r, t}^{\U, \QP} \mathcal C_r \subset L^\infty(\U, \mu)$, which verifies (a) in Dfn.~\ref{d:core}.

\paragraph{Verification of (b) in Dfn.~\ref{d:core}} Let $u \in \mathcal C_r$.
By Prop.~\ref{prop: 1-1}, we can identify the following two operators:
\begin{align*} 
\bar{T}_{r, t}^{\U, \QP} u(\cdot)=T_{t}^{\U(B_r), \QP_r^\cdot}u_{r}^{\cdot}(\cdot_{B_r})  \fstop
\end{align*}
This implies that 
\begin{align*} 
\Bigl( \bar{T}_{r, t}^{\U, \QP} u \Bigr)_r^\eta(\cdot)=\bar{T}_{r, t}^{\U, \QP} u (\cdot + \eta_{B_r^c})=T_{t}^{\U(B_r), \QP_r^\eta}u_{r}^{\eta}(\cdot)   \fstop
\end{align*}
Take $k=k(\eta)$ as in \eqref{e:R1}. As $\QP_{r}^{\eta}$ is supported on $\U^k(B_r)$, we only need to show 
\begin{align}\label{e:St1}
T_{t}^{\U^k(B_r), \QP_r^{k, \eta}}u_{r}^{\eta} \in \Lip_b(\U^k(B_r), \mssd_\U) \fstop
\end{align}
As $(\U^k(B_r), \mssd_\U, \mu_r^{k, \gamma})$ is $\RCD(0,\infty)$ for $k=k(\eta)$ for $\QP$-a.e.~$\eta$ due to~Prop.~\ref{p:BE2}, the corresponding semigroup satisfies 
the~$\Lip(\U^k(B_r), \mssd_\U)$-contraction property (\cite[(iv) Thm.\ 6.1]{AmbGigSav14}), which shows that for $\QP$-a.e.~$\eta$
\begin{align*} 
T_{t}^{\U^k(B_r), \QP_r^{k, \eta}}u_r^\eta \in \Lip_b(\U^k(B_r), \mssd_\U) \comma
\end{align*}
and its Lipchitz constant is bounded as 
$$\Lip_{\mssd_\U}(T_{t}^{\U^k(B_r), \QP_r^{k, \eta}}u_r^\eta) \le \Lip_{\mssd_\U}(u_r^\eta)  \comma$$
which completes the verification of~\ref{i:d:core2}.

\paragraph{Verification of (c) in Dfn.~\ref{d:core}} Let $u \in \mathcal C_r$. Thanks to the verification of (b), the square field~$\cdc^{\dUpsilon}_r(\bar{T}_{r, t}^{\U, \QP} u)$ is well-defined. In view of~\eqref{eq:p:MarginalWP:0}, it holds that for $\QP$-a.e.~$\eta$
\begin{align} \label{e:RLSP}
\cdc^{\dUpsilon}_r(\bar{T}_{r, t}^{\U, \QP} u)(\gamma+\eta_{B_r^c}) &= \cdc^{\dUpsilon(B_r)}\bigl((\bar{T}_{r, t}^{\U, \QP} u)_{r}^\eta\bigr)(\gamma) \quad \text{$\mu_{r}^\eta$-a.e.~$\gamma \in \U(B_r)$} \fstop 
\end{align}
By the contraction property \eqref{e:con}, we have 
$$\E^{\U(B_r), \QP_r^{\eta}}(T_{t}^{\U(B_r), \QP_r^{\eta}} u_{r}^\eta) \le \E^{\U(B_r), \QP_r^{\eta}}(u_{r}^\eta) \fstop$$
By Prop.~\ref{prop: 1-1} and \eqref{e:RLSP}, we obtain
\begin{align*}
\frac{1}{2}\int_{\U} \cdc^{\dUpsilon}_r(\bar{T}_{r, t}^{\U, \QP} u) \diff \QP &= \int_{\U} \E^{\U(B_r), \QP_r^{\eta}}\bigl((\bar{T}_{r, t}^{\U, \QP} u)_{r}^\eta \bigr)  \diff \mu(\eta)
\\
&= \int_{\U} \E^{\U(B_r), \QP_r^{\eta}}(T_{t}^{\U(B_r), \QP_r^{\eta}} u_{r}^\eta)  \diff \mu(\eta)
\\
&\le   \int_{\U} \E^{\U(B_r), \QP_r^{\eta}}(u_{r}^\eta) \diff\mu(\eta) 
\\
&=   \E_r^{\U, \QP}(u) <+\infty \comma
\end{align*}
which concludes the statement. 
\end{proof}

As a consequence of~Prop.~\ref{prop: 1-1} and Thm.~\ref{t:S=M} , we obtain the superposition formula for  the semigroup $\sem{T_{r, t}^{\U, \QP}}$ in terms of the semigroup $\sem{T_t^{\U(B_r), \QP_r^\eta}}$. 
\begin{cor}[Coincidence of semigroups]\label{prop: 1}
The following three operators coincide: 
	\begin{align} \label{eq: R-1}
		T_{r, t}^{\U, \QP}u(\gamma) & = \bar{T}_{r, t}^{\U, \QP}u(\gamma) =T^{\U(B_r), \QP_r^\gamma}_{t} u_{r}^\gamma(\gamma_{B_r}) \, ,
	\end{align}
	for $\QP$-a.e.\ $\gamma\in \dUpsilon$ and every $t>0$.
\end{cor}

\begin{rem}[{Comparison with~\cite{KawOsaTan21}}] \label{r:IDI}
As a result of Cor.~\ref{prop: 1}, the form $(\E^{\U, \QP}_r, \dom{\E^{\U, \QP}_r})$ coincides with the form $(\E_r^{\mathsf{lwr}}, \mathcal D_r^{\mathsf{lws}})$ defined in \cite[line 8, p.644]{KawOsaTan21} for $\beta=1,2,4$: as noted in Rem.~\ref{r:CK}, the semigroup~$T^{\U(B_r), \QP_r^\eta}_{t}$ corresponds the finite-particle Dyson SDE~\cite[(2.40)--(2.43)]{KawOsaTan21} with the configurations outside $B_r$ conditioned to be $\eta_{B_r^c}$ and with the reflecting boundary condition at $\partial B_r$. Hence, by~the expression~\eqref{eq: R-1}, the semigroup $\{T_{r, t}^{\U, \QP}\}_{t \ge 0}$ gives the transition probability of the solution to the same SDE~\cite[(2.40)--(2.43)]{KawOsaTan21} with the same boundary condition, but without conditioning the configuration outside $B_r$ to be a particular $\eta$.  In this case, the configuration~$\eta$ outside $B_r$ is randomly chosen according to the law~$\QP$ at $t=0$, and will stay at the same configuration for $t>0$. By the argument~on~\cite[p.~653]{KawOsaTan21}, this transition probability corresponds to~the semigroup associated with $(\E_r^{\mathsf{lwr}}, \mathcal D_r^{\mathsf{lws}})$. These two Dirichlet forms $(\E^{\U, \QP}_r, \dom{\E^{\U, \QP}_r})$ and  $(\E_r^{\mathsf{lwr}}, \mathcal D_r^{\mathsf{lws}})$, therefore,  have the same $L^2(\U, \QP)$-semigroup, which concludes 
$$(\E^{\U, \QP}_r, \dom{\E^{\U, \QP}_r})=(\E_r^{\mathsf{lwr}}, \mathcal D_r^{\mathsf{lws}}) \cquad r>0\fstop$$
\end{rem}

\subsection{Infinite-volume limit of Dirichlet forms}
We now construct the infinite-volume limit of $(\E^{\U, \QP}_r, \dom{\E^{\U, \QP}_r})$ as $r \to \infty$, which is a strongly local symmetric Dirichlet form  whose symmetrising measure is  $\sine_\beta$.
A key property is the following monotonicity, which can be immediately seen by the definition~\eqref{eq:d:LiftCdCRep} of the square field $\cdc^{\U}_{r, \ell}$, while it would not be immediate if we only used the superposition form~$(\bar{\E}^{\U, \QP}_r, \dom{\bar{\E}^{\U, \QP}_r})$. 
\begin{prop}[Monotonicity]\label{p:mono}
The form $(\E^{\U, \QP}_r, \dom{\E^{\U, \QP}_r})$ and the square field $\Gamma^\U_r$ are monotone increasing as $r \uparrow \infty$, viz., 
$$\cdc^\U_r(u) \le \cdc^\U_s(u)\comma \quad \E^{\U, \QP}_r(u) \le \E^{\U, \QP}_s(u) \comma \quad \dom{\E^{\U, \QP}_s} \subset \dom{\E^{\U, \QP}_r} \quad r\le s \fstop$$
\end{prop}
\begin{proof}
As $\mathcal C_r$ is dense in $\dom{\E^{\U, \QP}_r}$, it suffices to check~$\mathcal C_s \subset \mathcal C_r$ and~$\cdc^\U_r(u) \le \cdc^\U_s(u)$ on~$\mathcal C_s$. Let $u \in \mathcal C_s$ and we show $u \in \mathcal C_r$.  
 By a similar argument to~Lem.~\ref{l:SEF3}, 
\begin{align*}
 u_r^\eta \in \Lip_b(\U(B_r), \mssd_\U) \quad \text{$\QP$-a.e.~$\eta$ \ \ if} \ \    u_s^\eta \in \Lip_b(\U(B_s), \mssd_\U) \quad \text{$\QP$-a.e.~$\eta$} \fstop
\end{align*} 
By Dfn.~\ref{d:DT}, it is straightforward to see $\cdc^\U_r(u) \le \cdc^\U_s(u)$. Thus, 
$$\E^{\U, \QP}_{r}(u) =\frac{1}{2}\int_{\U}\cdc^\U_r(u) \diff \mu \le\frac{1}{2} \int_{\U}\cdc^\U_s(u) \diff \mu=\frac{1}{2}\int_{\U}\cdc^\U_s(u) \diff \mu =\E^{\U, \QP}_s(u) <+\infty \fstop$$
Therefore, we conclude $u \in \mathcal C_r$.
\end{proof}

We now define a Dirichlet form on $\U$ whose symmetrising measure is $\sine_\beta$ by the monotone limit of~$(\E^{\U, \QP}_r, \dom{\E^{\U, \QP}_r})$. 
\begin{defs}[Monotone limit form] \label{d:DFF}
Let $\QP$ be $\sine_\beta$ for $\beta>0$.
The form $(\E^{\U, \QP}, \dom{\E^{\U, \QP}})$ is defined as the monotone limit:
\begin{align} \label{eq:Temptation} 
\dom{\E^{\U, \QP}}&:=\{u \in \cap_{r>0} \dom{\E^{\U, \QP}_r}: \E^{\U, \QP}(u) = \lim_{r \to \infty}\E^{\U, \QP}_r(u) <+\infty\} \comma
\\
\E^{\U, \QP}(u)&:=\lim_{r \to \infty}\E^{\U, \QP}_r(u) \cquad 
\E^{\U, \QP}(u,v):=\frac{1}{4}\Bigl(\E^{\U, \QP}(u+v)- \E^{\U, \QP}(u-v)\Bigr) \notag \fstop
\end{align} 
The form $(\E^{\U, \QP}, \dom{\E^{\U, \QP}})$ is a symmetric Dirichlet form on~$L^2(\U, \mu)$ as it is the monotone limit of symmetric Dirichlet forms (e.g., by \cite[Exercise 3.9]{MaRoe90}).
The square field~$\cdc^\U$ is defined as the monotone limit of~$\cdc^\U_r$ as well:
\begin{align} \label{d:SF}
\cdc^\U(u):=\lim_{r \to \infty}\cdc^\U_r(u) \cquad \cdc^\U(u, v)= \frac{1}{4}\Bigl( \cdc^\U(u+v)- \cdc^\U(u-v)\Bigr)\cquad u, v \in \dom{\E^{\U, \QP}}\fstop
\end{align}
The corresponding $L^2(\U, \mu)$-semigroup is denoted by $\sem{T^{\U, \QP}_t}$.
\end{defs}
Recall that ${\sf Cyl}(\U)$ denotes the space of cylinder functions defined in~\eqref{d:CLF}.
\begin{rem}[Non-triviality]\label{r:CLF2}
Note that ${\sf Cyl}(\U) \subset \mathcal D({\E}^{\U, \mu})$. For $U=\Phi(u_1^*,\ldots, u_k^*) \in {\sf Cyl}(\U)$, we have the following expression:
\begin{align} \label{e:SCF2}
\cdc^{\U}(U)=\sum_{i, j=1}^k\partial_i \Phi(u_1^*,\ldots, u_k^*)\partial_j \Phi(u_1^*,\ldots, u_k^*)\langle \nabla u_i, \nabla u_j \rangle^*  \comma
\end{align}
where $\langle \nabla u_i, \nabla u_j \rangle=\bigl(\frac{d}{dx}u_i\bigr)\bigl(\frac{d}{dx}u_j\bigr)$.
In particular,  the form $(\E^{\U, \QP}, \dom{\E^{\U, \QP}})$ is non-trivial in the sense that $\E^{\U, \QP} \not \equiv 0$.
\end{rem}
\begin{proof}
 Take a cylinder function~$U=\Phi(u_1^*,\ldots, u_k^*)$ and define a compact set~$K=\cup_{i=1}^k {\rm supp}[u_i] \subset \R$, where ${\rm supp}[u]$ is the support of $u$. Take a sufficiently large $r_0>0$ such that $K \subset B_{r_0}$. Then, the values of $U$ do not depend on configurations outside~$B_r$, i.e., 
\begin{align}\label{e:SCF}
U(\gamma)=U(\gamma_{B_r})  \cquad \gamma \in \U \fstop 
\end{align}
Take any $r \ge r_0$. 
For each fixed $k \in \N_0$, it is easy to see that $U|_{\U^k(B_r)} \in \Lip_b(\U^k(B_r), \mssd_\U)$. Thus, $U|_{\U^k(B_r)} \in \dom{\E^{\U^k(B_r), \QP_r^{k, \eta}}}$. Furthermore, we have the following expression of~$\cdc^{\U}_r(U)$, see, e.g., \cite[Lem.~1.2]{MaRoe00}:
$$\cdc^{\U}_r(U)=\sum_{i, j=1}^k\partial_i \Phi(u_1^*,\ldots, u_k^*)\partial_j \Phi(u_1^*,\ldots, u_k^*)\langle \nabla u_i, \nabla u_j \rangle^* \fstop$$
Noting that the intensity measure $I_\QP$ of $\QP=\sine_\beta$ is the Lebesgue measure~$\mssm$ multiplied by $\frac{1}{2\pi}$ for every $\beta>0$, we have $\int_{\U}u^* \diff \QP=\frac{1}{2\pi}\int_{\R}u\diff \mssm$ for $u \in C_0(\R)$. Thus, $\cdc^{\U}_r(U) \in L^1(\U, \QP)$ and $U \in \dom{\E^{\U, \QP}_r}$ for every $r \ge r_0$, actually for every $r>0$ by the monotonicity in Prop.~\ref{p:mono}.
Thanks to \eqref{e:SCF}, $\cdc^{\U}_r(U)=\cdc^{\U}_s(U)$ for every $r, s \ge r_0$, which implies~$\cdc^{\U}(U)=\cdc^{\U}_r(U) \in L^1(\U, \QP)$ for every $r \ge r_0$ and concludes~\eqref{e:SCF2}.  Due to the expression~\eqref{e:SCF2},  it is easy to find $U\in {\sf Cyl}(\U)$ such that $\E^{\U, \QP}(U) \neq 0$, which concludes the non-triviality. 
\end{proof}
\begin{rem}[Invariance with respect to $\QP$]
The semigroup $\sem{T^{\U, \QP}_t}$ is invariant with respect to~$\QP$ because $(\E^{\U, \QP}, \dom{\E^{\U, \QP}})$ is symmetric. Indeed, by the symmetry $\E^{\U, \QP}(u,v)=\E^{\U, \QP}(v, u)$ for $u, v \in \dom{\E^{\U, \QP}}$, the semigroup $\sem{T^{\U, \QP}_t}$ has the $L^2$-symmetry 
$$\bigl(v, T_t^{\U, \QP}u\bigr)_{L^2(\QP)} = \bigl(T_t^{\U, \QP}v, u\bigr)_{L^2(\QP)} \cquad u, v \in L^2(\U, \QP) \comma$$ see \cite[Lem.~1.3.2]{FukOshTak11}.  As the indicator function $\1$ belongs to $\dom{\E^{\U, \QP}}$ with $\E^{\U, \QP}(\1)=0$, the mass-preservation (also called {\it conservativeness}) $T_t^{\U, \QP}\1=\1$ holds, see, e.g., \cite[Thm.~1.6.6]{FukOshTak11}. Thus,  we have
\begin{align}\label{d:SF1} 
\int_{\U} T_t^{\U, \QP}u \diff \QP = \int_{\U} uT_t^{\U, \QP}\1 \diff \QP = \int_{\U} u \diff \QP \quad u \in L^2(\U, \QP) \fstop
\end{align}
The invariance \eqref{d:SF1} extends to $u \in L^1(\U, \QP)$ by the standard approximation by truncation $u_n:=(-n) \vee u \wedge n \in L^2(\U, \QP)$ with the fact that $T_t^{\U, \QP}$ extends to a bounded linear operator on $L^1(\U, \QP)$ with the contraction property~$\|T_t^{\U, \QP}u\|_{L^1(\U, \QP)} \le \|u\|_{L^1(\U, \QP)}$, see \eqref{e:con1}.
\end{rem}
\begin{rem}[{Quasi-regularity and comparison with~\cite{KawOsaTan21}}] \label{r:IDI2}
For $\beta=1,2,4$, the form $(\E^{\U, \QP}, \dom{\E^{\U, \QP}})$ coincides with the lower Dirichlet form~$(\E^{\mathsf{lwr}}, \mathcal D^{\mathsf{lwr}})$ defined in~\cite[p.~644]{KawOsaTan21} because  $(\E_r^{\U, \QP}, \dom{\E_r^{\U, \QP}}) = (\E_r^{\mathsf{lwr}}, \mathcal D_r^{\mathsf{lwr}})$ for every $r>0$ as discussed in~Rem.~\ref{r:IDI}, and $(\E^{\mathsf{lwr}}, \mathcal D^{\mathsf{lwr}})$ is defined as the monotone increasing limit of $(\E_r^{\mathsf{lwr}}, \mathcal D_r^{\mathsf{lwr}})$. 
In view of \cite[Thm.~3.2 and \S7.1]{KawOsaTan21}, $(\E^{\U, \QP}, \dom{\E^{\U, \QP}})$ also coincides with the upper Dirichlet form~$(\E^{\mathsf{upr}}, \mathcal D^{\mathsf{upr}})$ when $\QP=\sine_\beta$ with $\beta=1, 2, 4$. This implies that $(\E^{\U, \QP}, \dom{\E^{\U, \QP}})$ is quasi-regular with respect to the vague topology~$\tau_\mrmv$, see~\S\ref{ss:DF} for the definition.  Thus, there exists an associated diffusion process whose transition semigroup coincides with the $L^2$-semigroup $T_t^{\U, \QP}$ quasi-everywhere. In~\cite[Thm.~24]{Osa12} (see also \cite[\S8]{Tsa16}),  this diffusion process was identified to the solution $(\mathsf X_t, \mathbb P_\gamma)$ to the unlabelled infinite Dyson Brownian SDE~\eqref{d:DBMS} with $\beta=1, 2, 4$ in the sense that the semigroup~$T_t^{\U, \QP}$ gives the transition probability of the unlabelled solution to~\eqref{d:DBMS} for quasi-every stating point:
$$T_t^{\U, \QP}u(\gamma)=\mathbb E_\gamma[u(\mathsf X_t)] \quad \text{q.e.~$\gamma$} \cquad t>0 \cquad u \in \mathcal B_b(\U) \fstop$$ 
\end{rem}

We now show that, for every~$\beta>0$, the form~$(\E^{\U, \QP}, \dom{\E^{\U, \QP}})$ is a strongly local symmetric Dirichlet form on~$L^2(\U, \QP)$ and satisfies the Rademacher-type property  with respect to the $L^2$-transportation-type distance~$\bar{\mssd}_\U$. Recall that $\Lip_b(\U, \bar\mssd_\U, \QP)$ denotes the space of $\QP$-measurable bounded $\bar{\mssd}_{\U}$-Lipschitz functions on $\U$.
\begin{prop}\label{p:DF}
The form $(\E^{\U, \QP}, \dom{\E^{\U, \QP}})$ is a strongly local symmetric Dirichlet form on~$L^2(\U, \QP)$.
Furthermore,  $(\E^{\U, \QP}, \dom{\E^{\U, \QP}})$ satisfies the Rademacher-type property:
\begin{align}\label{p:Rad}
\Lip(\U, \bar{\mssd}_\U, \QP) \subset \dom{\E^{\U, \QP}}\comma \quad \cdc^\U(u) \le \Lip_{\bar{\mssd}_\U}(u)^2 \qquad  u \in \Lip(\U, \bar{\mssd}_\U, \QP) \fstop
\end{align}
\end{prop}
\begin{proof}
The strong locality follows from~\eqref{d:SF}.
We show the Rademacher-type property. 
Since~$\cdc^\U$ is the limit square field of~$\cdc^{\U}_r$ as in~\eqref{d:SF}, it suffices to show 
\begin{align*} 
\Lip(\U, \bar{\mssd}_\U, \QP)  \subset \mathcal C_r \quad \text{and} \quad \cdc^\U_r(u) \le \Lip_{\bar{\mssd}_\U}(u)^2  \cquad u \in \Lip(\U, \bar{\mssd}_\U, \QP) \quad r>0\comma
\end{align*}
which has been already proven in Prop.~\ref{t:ClosabilitySecond}. 
The proof is complete.
\end{proof}

As an application of Prop.~\ref{p:DF}, we have the quasi-regularity of $\E^{\U, \QP}$ with a smaller domain. 
\begin{cor} \label{c:QRB}
Let $\QP$ be $\sine_\beta$ for $\beta>0$. Then, $\bigl(\E^{\U, \QP}, \Lip_b(\U, \mssd_\U) \cap \mathcal C_b(\U, \tau_\mrmv)\bigr)$ is closable and the closure $(\E^{\U, \QP}, \mathcal F)$ is quasi-regular with respect to the vague topology~$\tau_\mrmv$.
\end{cor}
\begin{proof}
We first note that the algebra~$\Lip_b(\U, \mssd_\U) \cap \mathcal C_b(\U, \tau_\mrmv)$ is not empty due to, e.g.,  (a) Example~\ref{r:LT}, and it is dense in $L^2(\U, \QP)$ by~\cite[Lem.~4.5]{AmbErbSav16} combined with~\cite[Prop.~4.30]{LzDSSuz21}.
The closability follows from the closedness of $(\E^{\U, \QP}, \dom{\E^{\U, \QP}})$ and $ \Lip_b(\U, \mssd_\U) \cap \mathcal C_b(\U, \tau_\mrmv) \subset \dom{\E^{\U, \QP}}$ by Prop.~\ref{p:DF}.  By applying~\cite[Cor.~3.22]{Suz23}, we conclude that the form~$(\E^{\U, \QP}, \mathcal F)$ is quasi-regular.
\end{proof}
\begin{rem}[Quasi-regularity and associated diffusions] \label{r:IDI3}
We do not know whether the form~$(\E^{\U, \QP}, \dom{\E^{\U, \QP}})$ is quasi-regular (w.r.t.~the vague topology $\tau_\mrmv$) except for~$\beta=1,2,4$, mainly because we do not know whether  $\dom{\E^{\U, \QP}}$ has a dense subset consisting of quasi-continuous functions, see the condition (b) of the quasi-regularity in \S\ref{ss:DF}.   However, Cor.~\ref{c:QRB} shows that we can take a smaller domain $\mathcal F \subset \dom{\E^{\U, \QP}}$ in such a way that $(\E^{\U, \QP}, \mathcal F)$ is quasi-regular.  In particular by~\cite[Thm.~3.5 p.103]{MaRoe90},  there exists an associated diffusion process on~$\U$ for {\it every} $\beta>0$. This gives the construction of diffusion processes for the range $0<\beta<1$ that was not covered in~\cite{Tsa16}.  
Whether this diffusion corresponds to \eqref{d:DBMS} is, however,  open, where one needs to care about the collision among particles, which is expected to happen when $0<\beta<1$. 
\end{rem}

\begin{prop}\label{prop: MGS}
The semigroup $\{T^{\U, \QP}_{t}\}_{t \ge 0}$ is the $L^p(\U, \QP)$-strong operator limit of the semigroups $\{T^{\U, \QP}_{r, t}\}_{t \ge 0}$ for $p=1,2$, viz., 
$$\text{{\small $L^p(\mu)$--}}\lim_{r \to \infty} T^{\U, \QP}_{r, t} u= T^{\U, \QP}_t u  \quad u \in  L^p(\U, \QP)\comma \quad t>0 \cquad p=1,2 \fstop$$
\end{prop}
\begin{proof}
The case $p=2$ follows from the monotonicity of~$(\E^{\U, \QP}_r, \dom{\E^{\U, \QP}_r})$  in~Prop.~\ref{p:mono}
and~\cite[S.14, p.373]{ReeSim80}. Although the case $p=1$ is a standard consequence of the case $p=2$ due to the $L^1$-contraction property of $T_{t, r}^{\U, \QP}$ and $T_t^{\U, \QP}$, we give a proof for the sake of clarity. We note that the $L^2$-operators $T^{\U, \QP}_{r, t}$ and $T^{\U, \QP}_{t}$ can be uniquely extended to the~$L^1$-strongly continuous Markovian contraction semigroups, see~\eqref{e:con1}. As $L^1(\U, \QP) \cap L^2(\U, \QP)$ is dense in $L^1(\U, \QP)$, for any $u \in L^1(\U, \QP)$ and $\e>0$, there exists $u_\e \in L^1(\U, \QP) \cap L^2(\U, \QP)$ so that $\|u-u_\e\|_{L^1(\QP)}<\e$ and   
\begin{align*}
&\|T^{\U, \QP}_{r, t} u - T^{\U, \QP}_{t}u\|_{L^1(\QP)} 
\\
&\le \|T^{\U, \QP}_{r, t} u - T^{\U, \QP}_{r, t}u_\e\|_{L^1(\QP)} +\|T^{\U, \QP}_{r, t} u_\e - T^{\U, \QP}_{t}u_\e\|_{L^1(\QP)} +\|T^{\U, \QP}_{t} u_\e - T^{\U, \QP}_{t}u\|_{L^1(\QP)} 
\\
&\le \|u - u_\e\|_{L^1(\QP)} +\|T^{\U, \QP}_{r, t} u_\e - T^{\U, \QP}_{t}u_\e\|_{L^2(\QP)} +\| u_\e - u\|_{L^1(\QP)} 
\\
&\xrightarrow{r\to\infty} \e + 0 +\e \fstop
\end{align*}
As $\e>0$ is arbitrarily small, the proof is completed. 
\end{proof}

\begin{cor} \label{c:LSG}For  $r>0$ and $u \in \dom{\E^{\U, \QP}}$, 
\begin{align} \label{e:LSG0}
T_{r', t}^{\U, \QP}u \xrightarrow{r' \to \infty} T_t^{\U, \QP} u \quad \text{weakly in $\dom{\E^{\U, \QP}_{r}}$} \fstop
\end{align}
In particular, 
\begin{align} \label{e:LSG}
\int_{\U} \cdc^{\U}_r(T^{\U, \QP}_tu)h \diff \QP 
 \le  \liminf_{r' \to \infty} \int_{\U} \cdc^{\U}_{r}(T_{r', t}^{\U, \QP}u)   h \diff \QP  \comma
\end{align}
for every non-negative $h \in \mathcal D(\E^{\U, \QP}_r) \cap L^\infty(\QP)$.
\end{cor}
\begin{proof}
First of all, \eqref{e:LSG0} is well-posed as $T_{r', t}^{\U, \QP}u, T_t^{\U, \QP} u \in \dom{\E^{\U, \QP}_{r}}$ thanks to the inclusion $\dom{\E^{\U, \QP}}\subset \cap_{r>0} \dom{\E^{\U, \QP}_r}$ and the monotonicity $\dom{\E^{\U, \QP}_{r'}} \subset  \dom{\E^{\U, \QP}_{r}}$ for $r \le r'$. 
By the monotonicity in~Prop.~\ref{p:mono} and the contraction property~\eqref{e:con} of the semigroup~$T_{r', t}^{\U, \QP}$  in terms of the Dirichlet form~ $\E^{\U, \QP}_{r'}$, it holds that for $r \le r'$
\begin{align*}
\E^{\U, \QP}_{r}(T_{r', t}^{\U, \QP}u) \le \E^{\U, \QP}_{r'}(T_{r', t}^{\U, \QP}u) \le \E^{\U, \QP}_{r'}(u) \le  \E^{\U, \QP}(u)<+\infty \fstop
\end{align*}
Combining with the fact that the semigroup $T_{r', t}^{\U, \QP}$ also contracts the $L^2(\U, \QP)$-norm (see~\eqref{e:con}), we conclude that $\{T_{r', t}^{\U, \QP}u\}_{r' \ge r}$ is bounded in $\dom{\E^{\U, \QP}_{r}}$. Thanks to {Prop.}~\ref{prop: MGS},  $\{T_{r', t}^{\U, \QP}u\}_{r' \ge r}$ converges to $T_{t}^{\U, \QP}u$ weakly in $\dom{\E^{\U, \QP}_{r}}$. 
The latter statement is a consequence of the first statement, see, e.g., \cite[Lem.~2.4]{HinRam03}. 
\end{proof}

\subsection{Bakry--\'Emery Curvature bound for the infinite-volume form}
In this subsection, we prove the Bakry--\'Emery curvature bound for~$(\E^{\U, \QP}, \dom{\E^{\U, \QP}})$.
\begin{thm} \label{t: main}
Let $\beta>0$ and $\mu=\sine_\beta$. The Bakry--\'Emery gradient estimate~$\BE(0,\infty)$ holds:
\begin{align} \label{m:BE}
\cdc^{\U}\bigl(T_t^{\U, \mu} u\bigr) \le T_t^{\U, \mu} \bigl(\cdc^{\U}(u)\bigr) \cquad u \in  \dom{\E^{\U, \mu}} \quad  t \ge 0\fstop \tag{$\BE(0,\infty)$}
\end{align}
Furthermore, the curvature lower bound $K=0$ is optimal when $\beta=2$.
\end{thm}

We first prove $\BE(0,\infty)$ for the truncated form $(\E^{\U, \QP}_{r}, \mathcal D(\E^{\U, \QP}_r))$.
\begin{lem}\label{l:BER}
The form~$(\E^{\U, \QP}_{r}, \mathcal D(\E^{\U, \QP}_r))$ satisfies $\BE(0,\infty)$ for every $r>0$$:$
\begin{align} \label{m:BE}
\cdc_r^{\U}\bigl(T_{t, r}^{\U, \mu} u\bigr) \le T_{t, r}^{\U, \mu} \bigl(\cdc_r^{\U}(u)\bigr) \cquad u \in  \dom{\E_r^{\U, \mu}} \quad  t \ge 0 \fstop 
\end{align}
\end{lem}
\begin{proof}
Take $u \in \mathcal D(\E^{\U, \QP}_r)$. By Prop.~\ref{p:BE2} combined with~\eqref{d:GSF} and~\eqref{e:FFV}, there exists $\Xi_r^1 \subset \U$ with $\QP(\Xi_r^1)=1$ so that, for every~$\eta \in \Xi_r^1$, there exists a measurable set $\Omega_{r}^{1, \eta} \subset \U(B_r)$ with  $\QP_{r}^{\eta}(\Omega_{r}^{1, \eta} )=1$ satisfying the following Bakry--\'Emery gradient estimate:
\begin{align} \label{e:BE2}
\Gamma^{\U(B_r)}(T_t^{\U(B_r), \QP_r^\eta}u_{r}^\eta)(\gamma) \le T_t^{\U(B_r), \QP_r^\eta}\bigl(\Gamma^{\U(B_r)}(u_{r}^\eta)\bigr)(\gamma) \cquad \gamma \in \Omega_{r}^{1, \eta}\fstop
\end{align}
Here, we used the fact that $\RCD(0,\infty)$ implies $\BE(0,\infty)$, see \S\ref{ss:RCD} for a characterisation of $\RCD(0,\infty)$. 
By~Prop.~\ref{t:ClosabilitySecond}, there exists $\Xi_r^2 \subset \U$ with $\QP(\Xi_r^2)=1$ so that, for every~$\eta \in \Xi_r^2$, there exists a measurable set $\Omega_{r}^{2, \eta} \subset \U(B_r)$ with  $\QP_{r}^{\eta}(\Omega_{r}^{2, \eta} )=1$ satisfying 
\begin{align} \label{e:BE3}
\cdc^{\U}_r(T_{r, t}^{\U, \QP} u)(\gamma+\eta_{B_r^c})&=\Gamma^{\U(B_r)}\Bigl( \bigl(T_{r, t}^{\U, \QP} u\bigr)_r^\eta\Bigr)(\gamma)  \cquad \gamma \in \Omega_{r}^{2, \eta}  ;
\\
\cdc^{\U}_r(u)(\gamma+\eta_{B_r^c})&=\Gamma^{\U(B_r)}( u_r^\eta)(\gamma) \fstop  \notag
\end{align}
By~Cor.~\ref{prop: 1}, there exists $\Lambda^3_r \subset \U$ with $\QP(\Lambda^3_r)=1$ so that
\begin{align} \label{e:BE4}
T_{r, t}^{\U, \QP} u(\gamma)=T_{t}^{\U(B_r), \QP_r^\gamma} u_r^\gamma(\gamma)  \cquad \gamma \in \Lambda^3_r\fstop
\end{align}
Recalling the notation~$\U_r^\eta:=\{\gamma \in \U: \gamma_{B_r^c}=\eta_{B_r^c}\}$, we can write $\Lambda_r^3$ as the union along the fibre $\U_r^\eta$: 
$$\Lambda^3_r={\bigcup_{\eta_{B_r^c} \in \Xi_r^3} \Lambda^3_r \cap \U_{r}^{\eta} }=\bigcup_{\eta \in \Xi_r^3}\pr_{B_r}^{-1}(\Omega_{r}^{3, \eta})\cap \U_{r}^{\eta} \comma$$ where $\Xi_r^3={\rm pr}_{B_r^c}(\Lambda^3_r)$, $\Omega_{r}^{3, \eta}=(\Lambda^3_r)_r^\eta:=\{\gamma \in \U(B_r): \gamma+\eta_{B_r^c} \in \Lambda^3_r\}$, and $\pr_{B_r}$ and $\pr_{B_r^c}$ are the projections defined in~\eqref{eq:ProjUpsilon}.
By the disintegration formula~\eqref{p:ConditionalIntegration2}, $\QP(\Xi_r^3)=1$ and $\QP_r^\eta(\Omega_{r}^{3, \eta})=1$ for every~$\eta \in \Xi_r^3$. 

Let  $\Xi_r:=\Xi_{r}^{1} \cap \Xi_{r}^{2}  \cap\Xi_{r}^{3}$ and $\Omega^\eta_r:=\Omega_{r}^{1, \eta} \cap \Omega_{r}^{2, \eta} \cap \Omega_{r}^{3, \eta}$ for $\eta \in \Xi_r$.
Set 
$$\Kappa_r:=\bigcup_{\eta \in \Xi_r} \pr_{B_r}^{-1}(\Omega_{r}^{\eta})\cap \U_{r}^{\eta} \fstop$$
By construction, $\QP(\Xi_r)=1$ and $\QP_r^\eta(\Omega^\eta_r)=1$ for every~$\eta \in \Xi_r$. By \eqref{e:BE2}, \eqref{e:BE3} and \eqref{e:BE4}, the following inequalities hold for every $\gamma \in \Kappa_r$:
\begin{align} \label{e:BER1-1}
\Gamma^\U_r(T_{r, t}^{\U, \QP} u)(\gamma) 
&= \Gamma^\U_r(T_{r, t}^{\U, \QP} u)(\gamma_{B_r}+\gamma_{B_r^c}) 
\\
&= \Gamma^{\U(B_r)}((T^{\U, \QP}_{r, t} u)_{r}^{\gamma})(\gamma_{B_r})  \notag
\\
&\le T_{t}^{\U(B_r), \QP_r^\gamma}\Gamma^{\U(B_r)}(u_{r}^{\gamma})(\gamma_{B_r})  \notag
\\
&= T_{t}^{\U(B_r), \QP_r^\gamma}\bigl(\Gamma^\U_r(u)_{r}^{\gamma}\bigr)(\gamma_{B_r})  \notag
\\
&= T_{r, t}^{\U, \QP} \Gamma^\U_r(u)(\gamma) \fstop \notag
\end{align}
Let $\Theta:=\{\gamma \in \U: \Gamma^\U_r(T_{r, t}^{\U, \QP} u)(\gamma) \le T_{r, t}^{\U, \QP} \Gamma^\U_r(u)(\gamma)\}$.  Then $\Theta$ is $\QP$-measurable since it is a sub-level set of a measurable function. Thanks to~\eqref{e:BER1-1}, $\Kappa_r  \subset \Theta$. By applying Lem.~\ref{l:sp}, we obtain $\QP(\Theta)=1$, which concludes $\BE(0,\infty)$ for the truncated form $(\E^{\U, \QP}_{r}, \mathcal D(\E^{\U, \QP}_r))$ for every $r>0$.
\end{proof}

We now prove $\BE(0,\infty)$ for $(\E^{\U, \QP}, \dom{\E^{\U, \QP}})$.
\begin{proof}[Proof of the first statement of Thm.~\ref{t: main}]
We prove $\BE(0,\infty)$ for $(\E^{\U, \QP}, \dom{\E^{\U, \QP}})$. It suffices to prove 
\begin{align} \label{e:BEG}
\int_{\U}\cdc^{\U}(T^{\U, \QP}_tu) h \diff \QP \le \int_{\U}  T^{\U, \QP}_t\cdc^{\U}(u) h \diff \QP  \comma
\end{align}
for all non-negative $h \in \mathcal D(\E^{\U, \QP}) \cap L^\infty(\QP)$.
 Indeed, thanks to the Rademacher-type property in Prop.~\ref{t:ClosabilitySecond},
we have 
$$\Lip_{b, +}(\U, \mssd_\U, \QP)\subset  \mathcal D(\E^{\U, \QP}) \cap L^\infty_+(\U, \QP).$$
As $\Lip_{b,+}(\U, \mssd_\U, \QP)\cap C(\tau_{\mrmv})$ is point separating (see \cite[(a) in Rem.~5.13]{LzDSSuz21}), it is measure-determining, see, e.g., \cite[p.113, (a) in Thm.~4.5 in Chap.~3]{EthKur86}. Thus, the inequality \eqref{e:BEG} implies $\cdc^{\U}(T^{\U, \QP}_tu) \le T^{\U, \QP}_t\cdc^{\U}(u)$ $\QP$-a.e.. 

We now prove~\eqref{e:BEG}. Let $u \in \dom{\E^{\U, \QP}}$ and recall the inclusion~$\dom{\E^{\U, \QP}} \subset \cap_{r>0} \dom{\E^{\U, \QP}_r}$. The following inequalities hold: 
\begin{align*}
\int_{\U}\cdc^{\U}(T^{\U, \QP}_tu) h \diff \QP
&= \int_{\U} \lim_{r \to \infty} \cdc^{\U}_r(T^{\U, \QP}_tu) h \diff \QP 
\\
&=  \lim_{r \to \infty} \int_{\U} \cdc^{\U}_r(T^{\U, \QP}_tu) h \diff \QP 
\\
& \le   \limsup_{r \to \infty} \liminf_{r' \to \infty} \int_{\U} \cdc^{\U}_{r}(T_{r', t}^{\U, \QP}u)   h \diff \QP 
\\
& \le  \limsup_{r' \to \infty} \int_{\U} \cdc^{\U}_{r'}(T_{r', t}^{\U, \QP}u)   h \diff \QP 
\\
& \le  \limsup_{r' \to \infty} \int_{\U} T_{r', t}^{\U, \QP} \cdc^{\U}_{r'}(u)   h \diff \QP 
\\
&= \int_{\U}  T^{\U, \QP}_t\cdc^{\U}(u) h \diff \QP \comma
\end{align*}
where in the first and the fourth lines,  we used the definition~$\cdc^{\U}(u)=\lim_{r \to \infty}\cdc^\U_r(u)$ and the monotonicity $\cdc^{\U}_{r} \le \cdc^{\U}_{r'}$ for $r \le r'$; in the third line, \eqref{e:LSG} in~Cor.~\ref{c:LSG} was used; in the fifth line,  $\BE(0,\infty)$ in Lem.~\ref{l:BER} was used; 
the last equality  followed by the $L^1$-contraction property~\eqref{e:con1} of $\|T_{r', t}^{\U, \QP}u\|_{L^1(\QP)}\le \|u\|_{L^1(\QP)}$, the monotone convergence $\cdc_r^\U(u) \nearrow \cdc^\U(u)$ as $r \to \infty$, and the $L^1$-strong operator convergence $T_{r', t}^{\U, \QP} \to T_{ t}^{\U, \QP}$ in~Prop.~\ref{prop: MGS}:
\begin{align*}
&\bigl\|T_{r', t}^{\U, \QP} \cdc^{\U}_{r'}(u) - T^{\U, \QP}_t \cdc^{\U}(u) \bigr\|_{L^1(\QP)} 
\\
&= \bigl\|T_{r', t}^{\U, \QP} \cdc^{\U}_{r'}(u) - T^{\U, \QP}_{r', t} \cdc^{\U}(u) \bigr\|_{L^1(\QP)} + \bigl\|T^{\U, \QP}_{r', t} \cdc^{\U}(u) - T^{\U, \QP}_t \cdc^{\U}(u)\bigr\|_{L^1(\QP)} 
\\
& \le \bigl\| \cdc^{\U}_{r'}(u) -  \cdc^{\U}(u)\bigr\|_{L^1(\QP)} + \bigl\|T^{\U, \QP}_{r', t} \cdc^{\U}(u) - T^{\U, \QP}_t \Gamma^{\U}(u)\bigr\|_{L^1(\QP)} 
\xrightarrow{r' \to \infty} 0 \fstop
\end{align*}
We have verified \eqref{e:BEG}, which completes the proof of the first statement in Thm.~\ref{t: main}. 
\end{proof}

\begin{proof}[Proof of the optimality $K=0$ in Thm.~\ref{t: main}]
By~\cite[Cor.~I]{Suz23} (see also \cite{OsaOsa23}), the form $(\E^{\U, \QP}, \dom{\E^{\U, \QP}})$  is irreducible, i.e.,  $\E^{\U, \QP}(u)=0$ with $u \in \dom{\E^{\U, \QP}}$ implies $u=\text{const.}$ $\QP$-a.e.. By \cite[Theorem]{Suz24}, the form $(\E^{\U, \QP},  \dom{\E^{\U, \QP}})$ does not have a spectral gap.  Recall the fact that if a Dirichlet form is irreducible and satisfies $\BE(K,\infty)$ with $K>0$, then it has a spectral gap, see e.g., \cite[Dfn.~3.1.11, Prop.~4.8.1]{BakGenLed14} (where the terminology, {\it the ergodicity}, is used for the irreducibility). This fact implies that the curvature lower bound $K$ cannot be positive, which concludes the optimality of $K=0$. 
\end{proof}

\subsection{Integral Bochner, local Poicar\'e and local log-Sobolev inequalities} \label{sec:AP1}
As an application of $\BE(0,\infty)$ in Thm.~\ref{t: main}, we show several functional inequalities.
We define {\it the integral $\mathbf \Gamma_2$-operator} as follows:
\begin{align} \label{d:IG}
&\mathbf \Gamma_2^{\U, \QP}(u, \phi):=\int_{\U}\biggl( \frac{1}{2}\cdc^{\U}(u)A^{\U, \QP} \phi - \cdc^{\U}(u, A^{\U, \QP}u) \phi\biggr) \diff \QP \comma
\\
&\dom{\mathbf \Gamma_2^{\U, \QP}}:=\bigl\{(u, \phi)\in \dom{A^{\U, \QP}}^{\times 2}: A^{\U, \QP}u \in \dom{\E^{\U, \QP}},\ \phi, A^{\U, \QP}u  \in L^\infty(\U, \QP) \bigr\} \comma \notag
\end{align}
where $A^{\U, \QP}$ denotes the $L^2(\U, \QP)$-infinitesimal generator associated with~$(\E^{\U, \QP}, \dom{\E^{\U, \QP}})$.
\begin{cor}\label{t:LPS}
Let $\mu=\sine_\beta$ with $\beta>0$.  The following  hold:
\begin{enumerate}[{\rm (a)}]
\item$(${\bf lntegral Bochner inequality}$)$ for every $(u, \phi) \in \dom{\mathbf \cdc^{\U, \QP}_2}$
\begin{align*}
\mathbf \cdc^{\U, \QP}_2(u, \phi) \ge 0 \ ;
\end{align*}
\item $(${\bf local Poincar\'e inequality}$)$  for $u \in \dom{\E^{\U, \QP}}$ and $t\ge 0$,
\begin{align*}
&T^{\U, \QP}_tu^2- (T^{\U, \QP}_tu)^2 \le 2tT^{\U, \QP}_t\cdc^{\U}(u)  \comma
\\
&T^{\U, \QP}_tu^2- (T^{\U, \QP}_tu)^2 \ge 2t\cdc^{\U} (T^{\U, \QP}_tu) \fstop
\end{align*}
\end{enumerate}
\end{cor}
\begin{proof}
The statement~(a) follows from~$\BE(0,\infty)$ proven in Thm.~\ref{t: main} and \cite[Cor.~2.3]{AmbGigSav15}. The statement (b) is a consequence of $\BE(0,\infty)$, see e.g.,~\cite[Thm.~4.7.2]{BakGenLed14}. 
\end{proof}

\begin{rem}[Local spectral gap inequality]\label{r:LSG}
Suppose that the form $(\E^{\U, \QP}, \dom{\E^{\U, \QP}})$ is quasi-regular (e.g., it is known for $\beta=1,2,4$ as discussed in Rem.~\ref{r:IDI2}). Then, there exists a diffusion process~$\{(X_t, \mathbb P_\gamma): t\ge 0, \ \gamma \in \U\}$ so that $T_t^{\U, \QP}u(\gamma) = \mathbb E_\gamma[u(X_t)]$ for quasi every~$\gamma$, where $\mathbb E_\gamma$ denotes the expectation under the probability measure~$\mathbb P_\gamma$. See \cite[Thm.~3.5 p.103]{MaRoe90}. In particular, there exists a transition probability kernel $P^{\U, \QP}_t(\gamma, \diff\eta)$ satisfying
\begin{align} \label{e:FE}
T_t^{\U, \QP}u(\gamma) = \int_{\U} u(\eta) P^{\U, \QP}_t(\gamma, \diff\eta) \qquad \text{for quasi every~$\gamma$} \fstop
\end{align}
The local Poincar\'e inequality (b) in Cor.~\ref{t:LPS} is the spectral gap inequality with the reference measure~$P^{\U, \QP}_t(\gamma, \diff\eta)$:
\begin{align}\label{e:LSGa}
\int_{\U} \Bigl|u(\eta)-{\int_{\U}u(\eta) P^{\U, \QP}_t(\gamma, \diff\eta)}\Bigr|^2P^{\U, \QP}_t(\gamma, \diff\eta) \le  2t\int_{\U} \cdc^{\U, \QP}(u)(\eta) P^{\U, \QP}_t(\gamma, \diff\eta)  \fstop 
\end{align}
The local Poincar\'e inequality is also called {\it reverse Poincar\'e inequality}. The name ``local'' comes from that the measure~$P_t^{\U, \QP}(\gamma, \diff \eta)$ is typically (e.g., heat kernel measures in complete Riemannian manifolds) {\it concentrated around $\gamma$} when $t$ is small (see \cite[\S 4.7 in p.~206]{BakGenLed14}). 
\end{rem}

The following corollary provides a  tail estimate of the measure~$P_t^{\U, \QP}(\gamma, \diff \eta)$, which decays sufficiently fast at the tail to make every (not necessarily bounded) $1$-Lipschitz function exponentially integrable. 
\begin{cor}[Exponential integrability of $1$-Lipschitz functions] \label{c:TES}
Let $\mu=\sine_\beta$ with $\beta>0$ and suppose that the form $(\E^{\U, \QP}, \dom{\E^{\U, \QP}})$ is quasi-regular.
 If $u$ is a $\QP$-measurable $\bar{\mssd}_\U$-Lipschitz function with $\Lip_{\bar{\mssd}_\U}(u) \le 1$ and $|u(\gamma)|<+\infty$ $\QP$-a.e.~$\gamma$, then for every $s<\sqrt{2/t}$
$$\int_{\U}  e^{s u(\eta)} P^{\U, \QP}_t(\gamma, \diff\eta) <+\infty \cquad \text{$\QP$-a.e.~$\gamma$}\fstop$$
\end{cor}
\begin{proof}
By the Rademacher-type property~\eqref{p:Rad} and the local Poincare inequality~\eqref{e:LSGa}, we can apply~\cite[Prop.~4.4.2]{BakGenLed14} with the reference measure~$P_t^{\U, \QP}(\gamma, \diff \eta)$.  
\end{proof}

The Bakry--\'Emery gradient estimate can be improved to the $L^1$-gradient estimate under the quasi-regularity.
\begin{cor}[$p$-Bakry-\'Emery estimate] \label{p:PBE}
Let $\mu=\sine_\beta$ with $\beta>0$ and suppose that the form $(\E^{\U, \QP}, \dom{\E^{\U, \QP}})$ is quasi-regular. 
Then,  the form $(\E^{\U, \QP}, \dom{\E^{\U, \QP}})$ satisfies $\BE_p(0,\infty)$ for every $1 \le p <\infty$:
 $$\cdc^{\U}(T_t^{\U, \QP}u)^{\frac{p}{2}} \le T_t^{\U, \QP}\bigl(\cdc^{\U}(u)^{\frac{p}{2}}\bigr) \cquad u \in \dom{\E^{\U, \QP}} \quad t \ge 0\fstop$$
 \end{cor}
 \begin{proof}
 The case of $p=2$ is proven in Thm.~\ref{t: main}. As $T_t^{\U, \QP} $can be extended to an~$L^p$-contraction semigroup by~\eqref{e:con1}, the RHS of the displayed formula in the statement is well-posed. The case of $p=1$ follows from the case of $p=2$ combined with Savar\'e's self-improvement result~\cite[Cor.~3.5]{Sav14}.
The case of $p >1$ follows by the case of $p=1$ and the Jensen inequality with the integral expression~\eqref{e:FE}
 $$\Bigl(T_t^{\U, \QP}\bigl(\cdc^{\U}(u)^{\frac{1}{2}}\bigr)\Bigr)^{p} \le T_t^{\U, \QP}\bigl(\cdc^{\U}(u)^{\frac{p}{2}} \bigr) \fstop \qedhere$$ 
 \end{proof}
 
 \begin{cor}[Local log-Sobolev inequality] \label{c:LLSI}
Let $\mu=\sine_\beta$ with $\beta>0$ and suppose that the form $(\E^{\U, \QP}, \dom{\E^{\U, \QP}})$ is quasi-regular.
Then,  for every positive~$u \in \dom{\E^{\U, \QP}}$ and $t \ge 0$,
\begin{align*}
&T^{\U, \QP}_t(u\log u)- T^{\U, \QP}_tu\log T^{\U, \QP}_t u \le t T^{\U, \QP}_t\biggl( \frac{\cdc^{\U}(u)}{u} \biggr) \comma
\\
&T^{\U, \QP}_t(u\log u)- T^{\U, \QP}_tu\log T^{\U, \QP}_t u \ge t \frac{\cdc^{\U}(T^{\U, \QP}_t u)}{T^{\U, \QP}_t u}  \fstop
\end{align*}
 \end{cor}
 \begin{proof}
 The result follows by $\BE_1(0,\infty)$ in~Cor.~\ref{p:PBE} and \cite[Thm.~5.5.2, Prop.~5.7.1]{BakGenLed14}.
 \end{proof}

 \begin{cor}[local hyper-contractivity] \label{c:LHC}
Let $\mu=\sine_\beta$ with $\beta>0$ and suppose that the form $(\E^{\U, \QP}, \dom{\E^{\U, \QP}})$ is quasi-regular.
 Then, for every $t>0$, $0<s \le t$, and $1 < p <q<\infty$ so that 
 $$\frac{q-1}{p-1}=\frac{t}{s} \comma$$
 it holds that 
 $$\Bigl(T_s^{\U, \QP}(T_{t-s}^{\U, \QP}u)^q\Bigr)^{1/q} \le \Bigl(T_t^{\U, \QP}u^p\Bigr)^{1/p} \cquad u \ge 0 \fstop$$
 \end{cor}
 \begin{proof}
 The result follows by $\BE_1(0,\infty)$ in~Cor.~\ref{p:PBE} and \cite[Thm.~5.5.5]{BakGenLed14}.
 \end{proof}

\section{Dimension-free and log Harnack inequalities} \label{sec:LH}
In this section, we prove functional inequalities associated with the $L^2$-transportation-type extended distance~$\bar{\mssd}_\U$ given in~\eqref{eq:dW2L}.
\begin{thm}\label{t:DFH}
Let $\QP=\sine_\beta$ with $\beta>0$. Then the following hold:
\begin{enumerate}[{\rm (a)}]
\item $(${\bf  log-Harnack inequality}$)$ for every non-negative $u \in L^\infty(\U, \QP)$, $\e \in (0,1]$ and $t>0$, there exists $\Omega \subset \U$ so that $\QP(\Omega)=1$ and 
$$T^{\U, \QP}_t\log (u+\e)(\gamma) \le \log (T^{\U, \QP}_tu(\eta)+\e) + \frac{\bar{\mssd}_\U(\gamma, \eta)^2}{{4t}} \comma \quad \text{$\gamma, \eta \in \Omega$} \ ;$$
\item $(${\bf  dimension-free Harnack inequality}$)$ for every non-negative $u \in L^\infty(\U, \QP)$, $t>0$ and $\alpha>1$, there exists $\Omega \subset \U$ so that $\QP(\Omega)=1$ and 
$$(T^{\U, \QP}_tu)^\alpha(\gamma)\le T^{\U, \QP}_tu^\alpha(\eta) \exp\Bigl\{ \frac{\alpha}{4(\alpha-1){t}}\bar{\mssd}_\U(\gamma, \eta)^2\Bigr\} \comma \quad \text{$\gamma, \eta \in \Omega$} \ ;$$
\item $(${\bf  Lipschitz contraction}$)$ for $u \in \Lip_b(\U, \bar{\mssd}_\U,  \QP)$ and $t>0$, 
$T_t^{\U, \QP}u$ has a $\bar{\mssd}_\U$-Lipschitz $\QP$-modification (denoted by the same symbol~$T_t^{\U, \QP}u$) such that the following estimate holds:
$$\Lip_{\bar{\mssd}_\U}({T}_t^{\U, \QP} u) \le \Lip_{\bar{\mssd}_\U}(u) \ ;$$
\item $(${\bf $L^\infty$-to-$\Lip$ regularisation}$)$ For ~$u \in L^\infty(\QP)$ and any~$t>0$, 
$T_t^{\U, \QP}u$ has a $\bar{\mssd}_\U$-Lipschitz $\QP$-modification (denoted by the same symbol~$T_t^{\U, \QP}u$)
such that the following estimate holds:
\begin{align*}
\Lip_{\bar{\mssd}_\U}({T}_t^{\U, \QP} u) \le \frac{1}{\sqrt{2 t}} \|u\|_{L^\infty(\QP)}  \fstop
\end{align*}
\end{enumerate}
\end{thm}
\begin{rem}[Non-triviality of $\bar\mssd_\U$] \label{r:UME}
As $\bar\mssd_\U$ is an extended distance, one might wonder if the RHS of (a) and (b) could be trivial, i.e., $\bar\mssd_\U(\gamma, \eta)=+\infty$ whenever $\gamma, \eta \in \Omega$ and $\gamma \neq \eta$.
When $\QP$ is tail-trivial and number rigid (e.g., both are known for~$\beta=2$),  this is not the case: 
Let $\Omega \subset \U$ be the set of full measure taken in (a) or (b). Let $\Lambda \subsetneq \Omega$ be an arbitrary subset such that $\QP(\Lambda)>0$ and $\QP(\Omega \setminus \Lambda)>0$. We write $\Omega_1:=\Omega \setminus \Lambda$.
Due to~\eqref{e:FD}, there exists $\Omega_2 \subset \Omega_1$ with $\QP(\Omega_1\setminus \Omega_2)=0$ such that for every $\gamma \in \Omega_2$, there exists $\eta \in \Lambda$ satisfying $\bar\mssd_\U(\gamma, \eta)<+\infty$, and by construction, $\gamma \neq \eta$. 
\end{rem}

\begin{rem}
In Kopfer--Sturm~\cite{KopStu21}, they proved the equivalence between the $\RCD$ condition and the dimension-free Harnack inequality in the framework of metric measure spaces. We cannot, however,  apply a similar proof to our setting because our space is not a metric measure space in their sense due to the fact that $\bar\mssd_\U$ is an extended  distance. We prove the dimension-free Harnack inequality by a finite-dimensional approximation.
\end{rem}

\begin{proof}[Proof of Thm.~\ref{t:DFH}]
We prove (a).
By the relation between $T^{\U, \QP}_{r, t}$ and~$T^{\U(B_r), \QP_r^{\cdot}}_t(\cdot_{B_r})$ in~Cor.~\ref{prop: 1}, there exists a measurable set~$\Omega^r_{\mathsf{sem}} \subset \U$ with $\QP(\Omega^r_{\mathsf{sem}})=1$ so that for every $\eta \in \Omega^r_{\mathsf{sem}}$
\begin{align}\label{e:semSO}
T^{\U, \QP}_{r, t}(\eta)=T^{\U(B_r), \QP_r^{\eta}}_t(\eta_{B_r}) \fstop
\end{align}
Let $u \in L^\infty(\QP)$. Thanks to~Lem.~\ref{l:FL},
there exists $\Omega^r_\infty \subset \dUpsilon$ so that $\QP(\Omega^r_\infty)=1$ and 
$$u_r^\eta \in L^\infty(\QP_r^\eta), \quad \eta \in \Omega^r_\infty,\quad r \in \N \fstop$$
By~Prop.~\ref{p:BE2}, there exists a measurable set~$\Omega^r_{\mathsf{rcd}} \subset \U$ so that $\QP(\Omega^r_{\mathsf{rcd}})=1$ and $(\U^k, \mssd_\U, \mu_r^{k, \eta})$ is $\RCD(0,\infty)$ with $k=k(\eta)$ as in \eqref{e:R1} for every $\eta \in \Omega^r_{\mathsf{rcd}}$. 
Let $\Omega^r:=\Omega^r_{\mathsf{sem}}\cap\Omega_\infty^r \cap \Omega_{\mathsf{rcd}}^r$. As the log-Harnack inequality holds in RCD~spaces (see, \cite[Lem.~4.6]{AmbGigSav15}), the following holds for every $\eta \in \Omega^r$ and $k=k(\eta)$ and $\e\in (0,1]$
\begin{align} \label{e:LHR}
T^{\U^k(B_r), \QP_r^{k, \eta}}_t\log (u_r^\eta+\e)(\gamma) \le \log \Bigl(T^{\U^k(B_r), \QP_r^{k, \eta}}_tu_r^\eta(\zeta)+\e\Bigr) + \frac{1}{4t}\mssd_\U(\gamma, \zeta)^2\comma 
\end{align}
for every $\gamma, \zeta \in \U^k(B_r)$. 
Noting  the convergence of the semigroups~$\sem{T_{r, t}^{\U, \QP}}$ to~$\sem{T_t^{\U, \QP}}$ in the $L^2(\U, \QP)$-operator sense by~Prop.~\ref{prop: MGS}, there exist $\Omega_{\mathsf{con}} \subset \U$ with $\QP(\Omega_{\mathsf{con}})=1$ and a (non-relabelled) subsequence of $(r)_{r \in \N}$ so that for every $\gamma \in \Omega_{\mathsf{con}}$ 
\begin{align} \label{e:semO}
&T^{\U, \QP}_{r, t}\log (u+\e)(\gamma) \xrightarrow{r \to \infty} T^{\U, \QP}_{t}\log (u+\e)(\gamma) \comma 
\\
&\log (T^{\U, \QP}_{r, t}u(\gamma)+\e) \xrightarrow{r \to \infty}  \log (T^{\U, \QP}_{t}u(\gamma)+\e) \fstop \notag
\end{align}
Let $\Omega=\Omega_{\mathsf{con}}\cap_{r\in \N}\Omega^r$, where  $\mu(\Omega)=1$ by construction. 
We now prove 
\begin{align} \label{eq: SL: 2}
T^{\U, \QP}_{t}\log (u+\e)(\gamma) \le \log (T^{\U, \QP}_{t}u(\eta)+\e) + \frac{1}{4t}\bar{\mssd}_\U(\gamma, \eta)^2\comma \quad \text{$\gamma, \eta \in \Omega$} \fstop
\end{align}
We may assume that $\bar{\mssd}_\U(\gamma, \eta)<+\infty$, otherwise there is nothing to prove. 
Thus, by~\eqref{e:LLR2}, there exists $s>0$ so that for every $r \ge s$  
 \begin{align} \label{e:NRS}
 \gamma_{B_{r}^c}=\eta_{B_{r}^c} \comma \quad \gamma(B_r)=\eta(B_r)\fstop
 \end{align}  
By \eqref{e:semSO}, \eqref{e:LHR}, and \eqref{e:NRS}, we have 
\begin{align} \label{eq: LH:3}
 T^{\U, \QP}_{r, t}\log (u+\e)(\gamma)  
 & =  T^{\U, \QP}_{r, t}\log (u+\e)(\gamma_{B_r} + \gamma_{B_r^c} )  
 \\
 &= T^{\U(B_r), \QP_r^{\gamma}}_t\log (u_r^{\gamma}+\e)(\gamma_{B_r})  \notag
 \\
 & \le \log (T^{\U(B_r), \QP_r^{\gamma}}_tu_r^\gamma(\eta_{B_r})+\e) +\frac{1}{4t} \mssd_\U(\gamma_{B_r}, \eta_{B_r})^2 \notag
 \\
  &=\log (T^{\U, \QP}_{r, t}u(\eta) +\e)+ \frac{1}{4t}\bar{\mssd}_\U(\gamma, \eta)^2 \notag \fstop
\end{align}
Therefore, by letting $r \to \infty$ with the $L^2$-strong operator convergence~\eqref{e:semO}, we obtain~\eqref{eq: SL: 2}, which completes the proof of (a). 

The proof of (b) follows precisely in the same strategy as above by replacing $T^{\U, \QP}_{t}\log (u+\e)$,  $\log (T^{\U, \QP}_{t}u+\e)$ and $\frac{1}{4t}\bar{\mssd}_\U(\gamma, \eta)^2$ by~$(T^{\U, \QP}_tu)^\alpha$, $T^{\U, \QP}_tu^\alpha$ and $\frac{\alpha}{4(\alpha-1)t}\bar{\mssd}_\U(\gamma, \eta)^2$ respectively, and noting that the dimension-free Harnack inequality holds on~$\RCD(K,\infty)$ spaces (\cite[Thm.~3.1]{Li15}).

The proof of (c): Note that $u_r^\eta \in \Lip(\U(B_r), \mssd_\U)$ whenever $u \in \Lip(\U, \bar{\mssd}_\U)$ and $\Lip_{\mssd_\U}(u_r^\eta) \le \Lip_{\bar{\mssd}_\U}(u)$ by Lem.~\ref{l:SEF3}.
Note also that the sought conclusion of (c) can be rephrased as 
$$\Bigl|{T}^{\U, \QP}_tu(\gamma)-{T}^{\U, \QP}_tu(\eta)\Bigr| \le \Lip_{\bar{\mssd}_\U}(u) \bar{\mssd}_\U(\gamma, \eta) \cquad \gamma, \eta \in \U\fstop$$
Thus, by the same proof strategy as in (a) replacing~$T^{\U, \QP}_{t}\log (u+\e)(\gamma)$ and~$\log (T^{\U, \QP}_{t}u(\eta)+\e)$ with~$T^{\U, \QP}_{t}u(\gamma)$ and~$T^{\U, \QP}_{t}u(\eta)$, and noting that the Lipschitz contraction property holds on RCD spaces (\cite[(iv) in Thm.~6.1]{AmbGigSav14b}), we conclude that there exists $\Omega \subset \U$ with $\QP(\Omega)=1$ so that 
$$\Bigl|T^{\U, \QP}_t(\gamma)-T^{\U, \QP}_t(\eta) \Bigr|\le \Lip_{\bar{\mssd}_\U}(u) \bar{\mssd}_\U(\gamma, \eta) \cquad \gamma, \eta \in \Omega\fstop$$
The conclusion now follows from the McShane extension Theorem (for extended metric spaces, see~\cite[Lem.~2.1]{LzDSSuz20}).

The proof of (d) is the same as that of (c) but using the $L^\infty$-to-$\Lip$ property (\cite[Thm.~6.5]{AmbGigSav14b}) in $\RCD(K,\infty)$ spaces instead of \cite[(iv) in Thm.~6.1]{AmbGigSav14b}).
The proof is complete.
\end{proof}

\begin{cor} \label{cor:DLA}
Let $\QP=\sine_\beta$ with $\beta>0$. Then
$$\text{$\Lip_b(\U, \bar{\mssd}_\U, \QP)$ is dense in $\dom{\E^{\U, \QP}}$. }$$
\end{cor}
\begin{proof}
Since $\Lip_b(\U, \bar{\mssd}_\U, \QP)$ is dense in $L^2(\U, \QP)$ by (c) in Rmk.~\ref{r:ped}, the statement follows from (c) in~Thm.~\ref{t:DFH} and Lem.~\ref{l:MU}.
\end{proof}

\section{Gradient flow} \label{sec:GF}
 In this section, we study  the dual flow of $\{T_t^{\U, \QP}\}_{t \ge 0}$ in the space $\mathcal P(\U)$ of Borel probability measures on $\U$. In particular, when $\beta=1,2,4$, the dual flow of the infinite Dyson Brownian motion is identified to the unique $\mssW_\E$-gradient flow of the Boltzmann--Shannon entropy associated with $\QP=\sine_\beta$, where $\mssW_\E$ is a  Benamou--Brenier type extended distance on $\mathcal P(\U)$, whch is purely given by the Dirichlet form data $(\E^{\U, \QP}, \dom{\E^{\U, \QP}})$. 

\paragraph{Boltzmann--Shannon entropy and Fisher information} Let $(\mathcal P(\U), \tau_\mrmw)$ be the space of all Borel probability measures on~$(\U, \tau_\mrmv)$ endowed with the weak topology $\tau_\mrmw$, i.e., the topology induced by the duality of $C_b(\U, \tau_\mrmv)$. Let $\mathcal P_\QP(\U)$ be the subspace of~$\mathcal P(\U)$ consisting of measures absolutely continuous with respect to $\QP$. We write $\nu=\rho \cdot \mu$ if $\rho=\frac{\diff \nu}{\diff \QP}$. 
\begin{itemize}
\item The {\it Boltzmann--Shannon entropy} $\Ent_\QP: \mathcal P_\QP(\U) \to \R\cup\{+\infty\}$ is defined as 
 $$\Ent_\QP(\nu):=\int_{\U} \rho \log \rho \diff \QP \comma \quad \nu=\rho\cdot \mu \fstop$$
The domain of $\Ent_\QP$ is denoted by $\dom{\Ent_\QP}:=\{\nu \in \mathcal P_\QP(\U): \Ent_\QP(\nu)<+\infty\}$. 
\item The {\it Fisher information} ${\sf F}_\QP:  \mathcal P_\QP(\U) \to \R\cup\{+\infty\}$ is defined as 
 $${\sf F}_\QP(\nu):=8\E^{\U, \QP}(\sqrt{\rho}) \comma \quad \nu=\rho\cdot \mu \fstop$$
The domain of ${\sf F}_\QP$ is denoted by $\dom{{\sf F}_\QP}:=\{\nu \in \mathcal P_\QP(\U): {\sf F}_\QP(\nu)<+\infty\}$.
\end{itemize}

\paragraph{The $L^2$-Monge--Kantorovich--Rubinstein--Wasserstein distance} For $\nu, \sigma \in \mathcal P(\U)$, we define an extended distance~$\mssW_{\mssd_\U}$ as 
\begin{align} \label{d:MKRW}
\mssW_{\mssd_\U}^2(\nu, \sigma):=\inf_{{\sf c} \in {\sf Cpl}(\nu, \sigma)} \int_{\U^{\times 2}} \mssd_{\U}^2(\gamma, \eta) \diff {\sf c}(\gamma, \eta) \comma
\end{align}
where ${\sf Cpl}(\nu, \sigma)$ is the space of all Borel probability measures on~$(\U^{\times 2},\tau_\mrmv^{\times 2})$ satisfying ${\sf c}(\Xi \times \U)=\nu(\Xi)$ and ${\sf c}(\U \times \Lambda)=\sigma(\Lambda)$ for every $\Xi, \Lambda \in \mathscr B(\U, \tau_\mrmv)$. 

\paragraph{Benamou--Brenier-like distance} We define a sub-algebra $\mathcal L$ in~$\dom{\E^{\U, \QP}}$:
\begin{align} \label{d:SAL}
\mathcal L=\{u \in \dom{\E^{\U, \QP}}: u \in L^\infty(\U, \QP) \comma \cdc^{\U}(u) \in L^\infty(\U, \QP)\} \fstop
\end{align}
Let $L^2((0, 1))$ denote the space of equivalence classes of square integrable functions with respect to the Lebesgue measure on the open interval~$(0, 1) \subset \R$.
\begin{defs}[{Continuity inequality} {\cite[(10.6)]{AmbErbSav16}}] \label{d:CI}
Given a family of probability {densities}~$(\rho_t)_{t\in[0, 1]} \subset L^1_+(\U, \QP)$, we write $(\rho_t)_{t \in [0, 1]} \in \mathsf{CI}^2(\E^{\U, \QP})$ if there exists $c \in L^2\bigl((0, 1)\bigr)$  so that 
\begin{align} \label{e:CE}
\biggl| \int_{\U} u \rho_t \diff \QP - \int_{\U} u \rho_s \diff \QP \biggr| \le \int_s^t c(r) \biggl(\int_{\U}\cdc^{\U, \QP}(u) \rho_r \diff \QP \biggr)^{1/2} \diff r \comma
\end{align}
for every $u \in \mathcal L$ and $0 \le s \le t \le 1$. The least $c$ in \eqref{e:CE} is denoted by $\|\rho'_t\|$. 
\end{defs}

The following definition is motivated by the celebrated Benamou--Brenier formula that is a variational characterisation of the optimal transportation distance in terms of the continuity equation. 
\begin{defs}[Benamou--Brenier-like extended distance {\cite[Dfn.~10.4]{AmbErbSav16}}]
For~$\nu, \sigma \in \mathcal P_\QP(\U)$,
\begin{align} \label{e:CE1}
\mssW_{\E}(\nu, \sigma)^2:=\inf\biggl\{ \int_0^1 \|\rho'_t\|^2 \diff t : (\rho_t) \in \mathsf{CE}^2(\E^{\U, \QP}) \comma \nu=\rho_0\cdot \QP\comma \sigma=\rho_1\cdot \QP\biggr\}\fstop
\end{align}
If there is no $(\rho_t)_{t \in [0,1]} \in \mathsf{CI}^2(\E^{\U, \QP})$ connecting $\nu$ and $\sigma$, we define $\mssW_\E(\nu, \sigma)=+\infty$. We will see in Cor.~\ref{t:mainc} that $\mssW_\E(\mathcal T_t^{\U, \QP}\nu, \nu)<+\infty$ for every $\nu \in \dom{\Ent_\QP}$ and $t \ge 0$, so that $\mssW_\E$ is non-trivial. 
\end{defs}

\begin{rem}
The extended distance $\mssW_{\E}$ on $\mathcal P(\U)$ is {\it intrinsic} for  $(\E^{\U, \QP}, \dom{\E^{\U, \QP}})$ in the sense that it is determined only by the data of the Dirichlet form $(\E^{\U, \QP}, \dom{\E^{\U, \QP}})$. It is open whether this intrinsic distance coincides with $\mssW_{\mssd_\U}$ given by the metric data $\mssd_\U$. The one inequality $\mssW_{\mssd_\U} \le \mssW_{\E}$ is true due to the Rademacher-type property in Prop.~\ref{p:DF}, which will be seen in Prop.~\ref{p:LDC} below. 
\end{rem}

\begin{prop}[Properties of $\mssW_\E$] \label{p:LDC} The following hold:
\begin{enumerate}[(i)]
\item $\mssW_{\E}$ is a complete length extended distance on~$\mathcal P_\QP(\U)$. Furthermore,  $\mssW_{\E}^2$ is jointly convex in $\mathcal P_\QP(\U)^{\times 2}$. 
\item The following inequality holds: $$\mssW_{\mssd_\U} \le \mssW_{\E} \fstop$$
\item Let $\nu_t=\rho_t\cdot\QP$ with $\rho_t=T_t^{\U, \QP}\rho_0$ and $\rho_0 \in L^2(\U, \QP)$. Then, 
$(\nu_t)_{t\in [0,1]} \in \mathsf{CE}^2(\E^{\U, \QP})$ and 
\begin{align}\label{ineq:RF}
\|\rho_t'\|^2 \le \Fis(\nu_t) \cquad t>0 \fstop
\end{align}
\end{enumerate}
\end{prop}

\begin{proof}
(i): The statement follows from~\cite[5th paragraph on p.113]{AmbErbSav16}. We note that the completeness follows from the completeness of $\mssW_{\mssd_\U}$ and the inequality~$\mssW_{\mssd_\U} \le \mssW_{\E}$, which will be proven in (ii). 

(ii): Let $\mssd_{\E}(\gamma, \eta):=\sup\{u(\gamma)-u(\eta): \cdc^{\U}(u) \le 1 \comma u \in \dom{\E^{\U, \QP}} \cap C_b(\U, \tau_\mrmv)\}$ be the intrinsic distance associated with $(\E^{\U, \QP}, \dom{\E^{\U, \QP}})$. By the Rademacher-type property in Prop.~\ref{p:DF}, we have 
$\mssd_{\U} \le \mssd_{\E}$ (see \cite[the first half of the proof of Thm.~5.25]{LzDSSuz21}).  In particular, 
$$\mssW_{\mssd_{\U}} \le \mssW_{\mssd_{\E}} \comma$$
where $\mssW_{\mssd_{\E}}$ denotes the extended distance \eqref{d:MKRW} induced by $\mssd_{\E}$ in place of $\mssd_\U$.
By {\cite[(a) Prop.~7.4]{AmbErbSav16}}, 
$$\mssW_{\mssd_{\E}} \le \mssW_{\Ch^{\mssd_\E, \QP}} \comma$$
where $\Ch^{\mssd_{\E}, \QP}$ is the Cheeger energy associated with $(\U, \mssd_{\E}, \QP)$ (see \cite[Dfn.~6.1]{AmbErbSav16}). Furthermore, by \cite[Thm.~12.5]{AmbErbSav16}, we have $\mssW_{\Ch^{\mssd_\E, \QP}} \le \mssW_{\E}$, which completes the proof.  

(iii): This is a consequence of \cite[{(10.5) and} (10.10)]{AmbErbSav16}. 
\end{proof}

\paragraph{Evolutional Variation Inequality} 
{Recall that $\{T_t^{\U, \QP}\}_{t\ge 0}$ is the $L^2$-semigroup associated with~$(\E^{\U, \QP}, \dom{\E^{\U, \QP}})$, and due to~\eqref{e:con1}, it can be extended to the $L^1$-contraction semigroup.  For $\nu=\rho\cdot \mu \in \mathcal P_\QP(\U)$ with $\rho \in L^1_+(\U, \QP)$,} we define {\it the dual flow} $\{\mathcal T_t^{\U, \QP}\}_{t \ge 0}$ as 
 $$\mathcal T_t^{\U, \QP}\nu=(T_t^{\U, \QP}\rho)\cdot \mu \cquad t \ge 0 \fstop$$
The following inequality is called {\it Evolutional Variation Inequality (EVI)}, which is a corollary of Thm.~\ref{t: main}.
 \begin{cor}[EVI] \label{t:mainc}  Suppose that $\QP=\sine_\beta$ with $\beta>0$.
For every $\nu, \sigma\in \dom{\Ent_\QP}$ with $\mssW_\E(\nu, \sigma)<+\infty$, the curve $t \mapsto \mathcal T^{\U, \QP}_t \sigma \in (\mathcal P(\U), \mssW_{\E})$ is locally absolutely continuous, $\Ent_\QP(\mathcal T_t^{\U, \QP}\sigma)<+\infty$, $\mssW_{\E}(\mathcal T_t^{\U, \QP}\sigma, \nu)<+\infty$ for every $t>0$ and 
\begin{align} \label{EVI}
\frac{1}{2}\frac{\diff^+}{\diff t}{\mssW}_{\E}\bigl({\mathcal T^{\U, \QP}_t \sigma}, \nu \bigr)^2  \le \Ent_{\mu}({\nu}) - \Ent_{\mu}({\mathcal T^{\U, \QP}_t \sigma}) \cquad  t>0\fstop \tag{$\EVI(0,\infty)$}
\end{align}
\end{cor}
\begin{proof}
This follows from Thm.~\ref{t: main} and \cite[Cor.~11.3]{AmbErbSav16}. 
\end{proof}

As a consequence of $\EVI(0,\infty)$, we have the following corollary, see \cite[Thm.~10.14, Cor.~11.2, 11.5, Thm.~11.4]{AmbErbSav16}.
\begin{cor}\label{c:GC} Suppose that $\QP=\sine_\beta$ with $\beta>0$. The following hold:\label{c:GC}
\begin{enumerate}[{\rm (a)}]
\item \label{ccc:1} The space~$(\dom{\Ent_\QP}, {\mssW}_{\E})$ is an extended geodesic metric space: for every pair $\nu, \sigma \in \dom{\Ent_\QP}$ with ${\mssW}_{\E}(\nu, \sigma)<+\infty$, there exists a ${\mssW}_{\E}$-Lipschitz curve $\nu_\cdot: [0,1] \to (\dom{\Ent_\QP},{\mssW}_{\E})$ so that
\begin{align*} 
\nu_0=\nu \cquad \nu_1=\sigma \cquad {\mssW}_{\E}(\nu_t, \nu_s) =|t-s|{\mssW}_{\E}(\nu, \sigma) \cquad s, t \in [0, 1]\fstop
\end{align*}
\item \label{ccc:2} Geodesic convexity: The entropy~$\Ent_\QP$ is ${\mssW}_{\E}$-convex along every ${\mssW}_{\E}$-geodesic $(\nu_t)_{t \in [0,1]}$$:$
\begin{align*}
\Ent_\QP(\nu_t) \le (1-t)\Ent_{\QP}(\nu_0) + t \Ent_{\QP}(\nu_1) \cquad t \in [0,1] \fstop
\end{align*}
\item Wasserstein contraction:
\begin{align*}
\mssW_{\E}\bigl(\mathcal T_t^{\U, \QP}\nu, \mathcal T_t^{\U, \QP}\sigma\bigr) \le \mssW_{\E}(\nu, \sigma)  \cquad t>0 \cquad \nu, \sigma \in \mathcal P_\QP(\U) \fstop
\end{align*}
\item \label{ccc:3}  The descending ${\mssW}_{\E}$-slope of $\Ent_\QP$ coincides with the Fisher information:
\begin{align*}
|{\sf D}_{{\mssW}_{\E}}^{-} \Ent_\QP(\nu)|^2 = {\sf F}(\rho) \cquad \nu=\rho\cdot \QP \in \dom{\Ent_\QP}\fstop
\end{align*}
\item \label{ccc:4}  The set $\Alpha_c:=\{\nu=\rho\cdot \QP \in \mathcal P_\QP(\U): \|\rho\|_{L^\infty(\QP)} \le c\}$ is geodesically convex with respect to~${\mssW}_{\E}$ for every $c>0$. 
\item \label{ccc:5} $L\log L$-regularisation of $\mathcal T_t^{\U, \QP}$:  For every~$\nu=\rho\cdot \QP \in \mathcal P_\QP(\U)$ $($not necessarily in~$\dom{\Ent_\QP}$$)$ and $\sigma \in \dom{\Ent_{\QP}}$, 
$$\Ent_{\QP}(\mathcal T_t^{\U, \QP}\nu) \le \Ent_{\QP}(\sigma) +\frac{1}{2t}{\mssW}_{\E}(\nu, \sigma)^2 \cquad t>0\fstop$$
\end{enumerate}
\end{cor}

Finally, we show that the dual flow~$\{\mathcal T_t^{\U, \QP}\}_{t\ge 0}$ is the $\mssW_{\E}$-gradient flow of $\Ent_\QP$. 
\begin{cor}[Gradient flow] \label{c:GF}Suppose that $\QP=\sine_\beta$ with $\beta>0$.
The dual flow~$\bigl\{\mathcal T_t^{\U, \QP}\bigr\}_{t\ge 0}$ is the unique solution to the $\mssW_{\E}$-gradient flow of $\Ent_{\QP}$.  Namely, for every $\nu_0 \in \dom{\Ent_\QP}$, the curve $[0, +\infty) \ni t \mapsto \nu_t=\mathcal T_t^{\U, \QP} \nu_0\in \dom{\Ent_\QP}$ is the unique solution to the energy equality starting at $\nu_0$$:$
\begin{align}
\frac{\diff}{\diff t} \Ent_\QP({\nu_t}) = -|\dot\nu_t|^2 = -|{\sf D}^-_{\mssW_{\E}} \Ent_\QP|^2(\nu_t) \quad\text{a.e.~$t>0$} \fstop
\end{align}
\end{cor}
\begin{proof}
This follows from Cor.~\ref{t:mainc} and~\cite[Thm.~3.5]{MurSav20}. The uniqueness follows from~\cite[Thm.~4.2]{MurSav20}. Note that although \cite{MurSav20} works in metric spaces (as opposed to extended metric spaces), we can apply the result to our setting by the following argument: It is easy to see that $\mssW_{\E}(\nu_t, \nu_s)<+\infty$ for every $s, t \in [0,+\infty)$ by Cor.~\ref{t:mainc} with $\sigma=\nu=\nu_0$ and the triangle inequality. Thus, we can think of $(\nu_t)$ as a curve in a (non-extended) metric space $\{\sigma \in \mathcal P(\U): \mssW_{\E}(\nu_0, \sigma)<+\infty\}$, to which we can apply the results in \cite{MurSav20}. 
\end{proof}

\section{Generalisation} \label{sec:GL}
We have been so far working in the case of $\sine_\beta$. In this section, we seek a generalisation to a broader class of probability measures on~{$\U=\U(\R)$}. As an application, we prove $\BE(0,\infty)$ in the case of the $1$-dimensional circular $(\beta,s)$-Riesz ensemble. 
In this section, we denote by $\mssm$ and $\mssm_r$ the Lebesgue measure on~$\R$ and its restriction on~$B_r=[-r,r]$ respectively, and we take the Euclidean distance~$\mssd(x, y):=|x-y|$ for $x, y \in \R$. 
Let $\mu$ be a Borel probability on~$\U$ and recall that $\mu_r^\eta$ is the projected conditional probability measure, defined right after \eqref{e:CES}, and $\mu_r^{k, \eta}=\mu_r^{\eta}\mrestr{\U^k(B_r)}$ is the restriction to~$\U^k(B_r)$. Let $\mathcal K(\mu^\eta_r) \subset \N_0$ be defined as
$$\mathcal K(\mu^\eta_r):=\{k \in \N_0: \mu_r^{k, \eta}(\U^k(B_r))>0\} \fstop$$  
Under the number rigidity~\eqref{e:R1}, we have $\#\mathcal K(\mu^\eta_r) =1$. However, we do not assume the number rigidity in the following arguments. 
Recall that the intensity measure $I_\QP$ for $\QP$ was defined in~\eqref{d:IS} and that 
the set~$\U^k_\e(B_r) \subset \U^k(B_r)$ was defined as
\begin{align*}
\U^k_\e(B_r)&:=\biggl\{\gamma =\sum_{i=1}^k\delta_{x_i} \in \U(B_r): |x_i-x_j| \ge \e, \ i, j \in \{1, \ldots, k\}\biggr\} \fstop
\end{align*}
\begin{ass}\label{a:GT}
Let $K \in \R$ and $\mu$ be a Borel probability measure on~$\U=\U(\R)$ whose intensity measure satisfies $I_\QP(C)<+\infty$ for every compact set~$C \subset \R$. Assume the following:
\begin{enumerate}[$(a)$]
\item the measure~$\mu_r^\eta$ is absolutely continuous with respect to the Poisson measure~$\pi_{\mssm_r}$, and $\mu_r^{k, \eta}$ is equivalent to $\pi_{\mssm_r}|_{\U^k(B_r)}$ for every $k \in \mathcal K(\mu^\eta_r)$, $\QP$-a.e.~$\eta$ and every $r>0$;
\item {(Conditional geodesic $K$-convexity)} the following hold:
\begin{itemize}
\item  the density~$\frac{\diff\mu_r^{k, \eta}}{\diff \pi_{\mssm_r}|_{\U^k(B_r)}}$ is continuous on $\U^k(B_r)$;
\item the logarithmic density~$\Psi_r^{k, \eta}: \U^k(B_r) \to \R \cup\{+\infty\}$ defined as 
$$\Psi_r^{k, \eta}=-\log\Bigl(\frac{\diff\mu_r^{k, \eta}}{\diff \pi_{\mssm_r} |_{\U^k(B_r)}}\Bigr)$$
 is bounded and continuous on~$\U^k_\e(B_r)$ for every $k \in  \mathcal K(\mu^\eta_r)$,  $\QP$-a.e.~$\eta$, $\e>0$ and $r>0$;
 \item $\Psi_r^{k, \eta}$ is $K$-geodesically convex with respect to~$\mssd_{\U}$ on $\U^k(B_r)$ 
  for every $k \in \mathcal K(\mu^\eta_r)$, $\QP$-a.e.~$\eta$ and every $r>0$.
  \end{itemize}
\end{enumerate}
\end{ass}
Under Assumption~\ref{a:GT}, the strongly local symmetric Dirichlet form~$(\E^{\U, \QP}, \dom{\E^{\U, \QP}})$ can be constructed in the same proof as in the case of  $\sine_\beta$ because we have not used any particular property of $K=0$. We further show the Bakry--\'Emery curvature bound~$\BE(K,\infty)$ for the form~$(\E^{\U, \QP}, \dom{\E^{\U, \QP}})$ and related functional inequalities.
\begin{thm}\label{t:GT}
Suppose that $\QP$ satisfies Assumption~\ref{a:GT}.  Then, $(\E^{\U, \QP}, \dom{\E^{\U, \QP}})$ satisfies the following:
\begin{enumerate}[$(a)$]
\item {\bf $($Bakry--\'Emery inequality $\BE(K,\infty)$$)$}
\begin{align*}
\cdc^{\U}\bigl(T_t^{\U, \mu} u\bigr) \le e^{-2Kt}T_t^{\U, \mu} \cdc^{\U}(u) \cquad u \in  \dom{\E^{\U, \mu}} \quad t \ge 0 ; 
\end{align*}
\item$(${\bf lntegral Bochner inequality}$)$ for every $(u, \phi) \in \dom{\mathbf \cdc^{\U, \QP}_2}$
\begin{align*}
\mathbf \cdc^{\U, \QP}_2(u, \phi) \ge 2K \int_{\U} \cdc^{\U}(u) \phi \diff \QP \ ;
\end{align*}
\item $(${\bf local Poincar\'e inequality}$)$  for $u \in \dom{\E^{\U, \QP}}$ and $t \ge 0$,
\begin{align*}
&T^{\U, \QP}_tu^2- (T^{\U, \QP}_tu)^2 \le \frac{1-e^{-2Kt}}{K}T^{\U, \QP}_t\cdc^{\U}(u)  \comma
\\
&T^{\U, \QP}_tu^2- (T^{\U, \QP}_tu)^2 \ge \frac{e^{{2K}t}-1}{K}\cdc^{\U} (T^{\U, \QP}_tu) \ ;
\end{align*}
\item $(${\bf  log Harnack inequality}$)$ for every non-negative $u \in L^\infty(\U, \QP)$, $\e \in (0, 1]$, $t>0$, there exists $\Omega \subset \U$ so that $\QP(\Omega)=1$ and 
$$T^{\U, \QP}_t\log (u+\e)(\gamma) \le \log (T^{\U, \QP}_tu(\eta)+\e) + \frac{K}{2(1-e^{-2Kt})}\bar{\mssd}_\U(\gamma, \eta)^2\comma \quad \text{$\gamma, \eta \in \Omega$} \ ;$$
\item $(${\bf  dimension-free Harnack inequality}$)$ for every non-negative $u \in L^\infty(\U, \QP)$, $t>0$ and $\alpha>1$ there exists $\Omega \subset \U$ so that $\QP(\Omega)=1$ and 
$$(T^{\U, \QP}_tu)^\alpha(\gamma)\le T^{\U, \QP}_tu^\alpha(\eta) \exp\Bigl\{ \frac{\alpha K}{2(\alpha-1)(1-e^{-2Kt})}\bar{\mssd}_\U(\gamma, \eta)^2\Bigr\} \comma \quad \text{$\gamma, \eta \in \Omega$} \ ;$$
\item $(${\bf  Lipschitz contraction}$)$ For $u \in \Lip_b(\U, \bar{\mssd}_\U, \QP)$ and $t>0$, 
$T_t^{\U, \QP}u$ has a $\bar{\mssd}_\U$-Lipschitz $\QP$-modification (denoted by the same symbol~$T_t^{\U, \QP}u$)
such that the following estimate holds:
$$\Lip_{\bar{\mssd}_\U}({T}_t^{\U, \QP} u) \le e^{-Kt}\Lip_{\bar{\mssd}_\U}(u) \cquad t \ge 0 \ ;$$
\item$(${\bf $L^\infty$-to-$\Lip$ regularisation}$)$ 
For $u \in L^\infty(\U, \QP)$ and $t>0$, 
$T_t^{\U, \QP}u$ has a $\bar{\mssd}_\U$-Lipschitz $\QP$-modification~(denoted by the same symbol~$T_t^{\U, \QP}u$)
such that the following estimate holds:
\begin{align*}
\Lip_{\bar{\mssd}_\U}({T}_t^{\U, \QP} u) &\le \frac{1}{\sqrt{2I_{2K}(t)} } \|u\|_{L^\infty(\QP)} \cquad  t>0 \comma
\end{align*}
where $I_K(t):=\int_0^t e^{Kr} \diff r$;
\item$(${\bf The density of Lipschitz algebra}$)$ 
$$\text{$\Lip_b(\U, \bar{\mssd}_\U, \QP)$ is dense in $\dom{\E^{\U, \QP}}$} \scolon$$
\item $(${\bf Evolutional Variation Inequality}$)$
For every $\nu, \sigma\in \dom{\Ent_\QP}$ with $\mssW_\E(\nu, \sigma)<+\infty$, the curve $t \mapsto \mathcal T^{\U, \QP}_t \sigma \in (\mathcal P(\U), \mssW_{\E})$ is locally absolutely continuous, $\Ent_\QP(\mathcal T_t^{\U, \QP}\sigma)<+\infty$, $\mssW_{\E}(\mathcal T_t^{\U, \QP}\sigma, \nu)<+\infty$ for every $t>0$, and 
\begin{align*} 
\frac{1}{2}\frac{\diff^+}{\diff t}{\mssW}_{\E}\bigl({\mathcal T^{\U, \QP}_t \sigma}, \nu \bigr)^2 +\frac{K}{2} {\mssW}_{\E}\bigl({\mathcal T^{\U, \QP}_t \sigma}, \nu \bigr)^2\le \Ent_{\mu}({\nu}) - \Ent_{\mu}({\mathcal T^{\U, \QP}_t \sigma}) \cquad  t>0\scolon
\end{align*}

\item $(${\bf Geodesic $K$-convexity}$)$ The entropy~$\Ent_\QP$ is ${\mssW}_{\E}$-convex along every ${\mssW}_{\E}$-geodesic $(\nu_t)_{t \in [0,1]}$$:$
\begin{align*}
\Ent_\QP(\nu_t) \le (1-t)\Ent_{\QP}(\nu_0) + t \Ent_{\QP}(\nu_1) - \frac{K}{2}t(1-t)\mssW_{\E}(\nu_0, \nu_1)^2 \cquad t \in [0,1] \scolon
\end{align*}
\item $(${\bf Wasserstein contraction}$)$
\begin{align*}
\mssW_{\E}\bigl(\mathcal T_t^{\U, \QP}\nu, \mathcal T_t^{\U, \QP}\sigma\bigr) \le e^{-Kt}\mssW_{\E}(\nu, \sigma)  \cquad t \ge 0 \cquad \nu, \sigma \in \mathcal P_\QP(\U) \scolon
\end{align*}

\item $(${\bf $L\log L$-regularisation}$)$
For every~$\nu=\rho\cdot \QP \in \mathcal P_\QP(\U)$ $($not necessarily in~$\dom{\Ent_\QP}$$)$ and $\sigma \in \dom{\Ent_{\QP}}$, 
$$\Ent_{\QP}(\mathcal T_t^{\U, \QP}\nu) \le \Ent_{\QP}(\sigma) +\frac{K}{e^{2Kt}-1}{\mssW}_{\E}(\nu, \sigma)^2 \cquad t>0\scolon$$

\item $(${\bf Gradient flow}$)$
The dual flow~$\bigl\{\mathcal T_t^{\U, \QP}\bigr\}_{t >0}$ is the unique solution to the $\mssW_{\E}$-gradient flow of $\Ent_{\QP}$.  Namely, for every $\nu_0 \in \dom{\Ent_\QP}$, the curve $[0, +\infty) \ni t \mapsto \nu_t=\mathcal T_t^{\U, \QP} \nu_0\in \dom{\Ent_\QP}$ is the unique solution to the energy equality starting at $\nu_0$:
\begin{align}
\frac{\diff}{\diff t} \Ent_\QP({\nu_t}) = -|\dot\nu_t|^2 = -|{\sf D}^-_{\mssW_{\E}} \Ent_\QP|^2(\nu_t) \cquad\text{a.e.~$t>0$} \fstop
\end{align}
\end{enumerate}
Furthermore, if the form $(\E^{\U, \QP}, \dom{\E^{\U, \QP}})$ is quasi-regular, the following hold:
\begin{enumerate}[(a)]  \setcounter{enumi}{13}
\item$(${\bf Exponential integrability of $1$-Lipschitz functions}$)$
 If $u$ is a $\bar{\mssd}_\U$-Lipschitz function with $\Lip_{\bar{\mssd}_\U}(u) \le 1$ and $|u(\gamma)|<+\infty$ $\QP$-a.e.~$\gamma$, then, for every $s<\sqrt{\frac{8K}{1-e^{-2Kt}}}$
$$\int_{\U} e^{s u(\eta)} P_t^{\U, \QP}(\gamma, \diff \eta)<\infty \scolon$$
\item $(${\bf$p$-Bakry-\'Emery estimate}$)$
The form $(\E^{\U, \QP}, \dom{\E^{\U, \QP}})$ satisfies $\BE_p(K,\infty)$ for every $1 \le p <\infty$:
 $$\cdc^{\U}(T_t^{\U, \QP}u)^{\frac{p}{2}} \le e^{-pKt}T_t^{\U, \QP}\bigl(\cdc^{\U}(u)^{\frac{p}{2}}\bigr) \cquad u \in \dom{\E^{\U, \QP}} \quad t \ge 0\scolon$$
 \item $(${\bf local log-Sobolev inequality}$)$
 For every non-negative $u \in \dom{\E^{\U, \QP}}$, $t \ge 0$,
\begin{align*}
&T^{\U, \QP}_t(u\log u)- T^{\U, \QP}_tu\log T^{\U, \QP}_t u \le \frac{1-e^{-2Kt}}{2K}  T^{\U, \QP}_t\biggl( \frac{\cdc^{\U}(u)}{u} \biggr) \comma
\\
&T^{\U, \QP}_t(u\log u)- T^{\U, \QP}_tu\log T^{\U, \QP}_t u \ge \frac{e^{2Kt}-1}{2K} \frac{\cdc^{\U}(T^{\U, \QP}_t u)}{T^{\U, \QP}_t u}  \scolon
\end{align*}
 \item $(${\bf local hyper-contractivity}$)$
 For every $t>0$, $0<s \le t$, and $1 < p <q<\infty$ so that 
 $$\frac{q-1}{p-1}=\frac{e^{2Kt}-1}{e^{2Ks}-1} \comma$$
 it holds that 
 $$\Bigl(T_s^{\U, \QP}(T_{t-s}^{\U, \QP}u)^q\Bigr)^{1/q} \le \Bigl(T_t^{\U, \QP}u^p\Bigr)^{1/p} \cquad u \ge 0 \fstop$$
 When $K=0$, $\frac{e^{2Kt}-1}{e^{2Ks}-1}$ is conventionally replaced by $\frac{t}{s}$. 
 \end{enumerate}

\end{thm}

\begin{proof}
Thanks to Assumption~\ref{a:GT}, the space~$(\U^k(B_r), \mssd_\U, \QP_r^{k, \eta})$ satisfies $\RCD(K,\infty)$ for every $k \in \mathcal K(\mu_r^\eta)$ as in the same proof of Prop.~\ref{p:BE2}. 
The rest of the proofs in Sections \ref{sec:CI}, \ref{sec:LH} and \ref{sec:GF} work exactly in the same way up to the multiplicative constants (e.g., $e^{-2Kt}$ instead of $1$ for the $\BE(K,\infty)$ inequality).
\end{proof}

\begin{rem}[Finite intensity]
We impose the condition $I_\QP(C)<+\infty$ for every compact set $C \subset \R$ to have the non-triviality of $(\E^{\U, \QP}, \dom{\E^{\U, \QP}})$, see Rem.~\ref{r:CLF2}. 
\end{rem}
\begin{rem}[Number rigidity]
Under the number rigidity~\eqref{e:R1}, we have $\#\mathcal K(\mu^\eta_r) =1$. However, this has not been essentially used for the proofs in the case of~$\sine_\beta$. Indeed,  in the arguments in Section~\ref{sec:CI} and~\ref{sec:LH}  involving the semigroup $T_t^{\U(B_r), \QP_r^\eta}$, we just need to observe that the $k$-particle space~$\U^k(B_r)$ is an invariant set of the semigroup~$T_t^{\U(B_r), \QP_r^\eta}$ for every $k$, i.e., 
\begin{align}\label{d:IV}
T_t^{\U(B_r), \QP_r^\eta}u\1_{\U^k(B_r)} = \1_{\U^k(B_r)}T_t^{\U(B_r), \QP_r^\eta}u \comma \quad u \in L^2(\U(B_r), \QP_r^\eta)\fstop
\end{align}
The equality~\eqref{d:IV} easily follows from \cite[Thm.~1.6.1]{FukOshTak11} and the fact 
$$\1_{\U^k(B_r)} \in \dom{\E^{\U(B_r), \QP_r^\eta}} \cquad \E^{\U(B_r), \QP_r^\eta}(\1_{\U^k(B_r)})=0 \cquad k \in \N_0 \fstop$$
From the probabilistic viewpoint, this corresponds to the fact that the number of particles in the corresponding diffusion in $B_r$ is preserved under the time evolution due to the reflecting boundary condition at the boundary $\partial B_r$, which is derived from the choice of the domain $\dom{\E^{\U(B_r), \QP_r^\eta}}$).
Thus, we may think of $\U(B_r)$ as the disjoint union $\sqcup_{k \in \mathcal K(\mu^\eta_r)}\U^{k}(B_r)$ regarding the semigroup action. 
Hence, by applying the same proofs as in~Section~\ref{sec:CI} and \ref{sec:LH} to each~$k \in \mathcal K(\mu^\eta_r)$ (instead of using the particular $k=k(\eta)$ selected by $\eta$), Thm.~\ref{t:GT} can be proven without the number rigidity~\eqref{e:R1}. 
\end{rem}
\begin{rem}
It is open whether there exists a Borel probability measure $\QP$ on $\U$ such that Assumption~\ref{a:GT} holds with $K>0$ and $\QP(\U^\infty)=1$.
\end{rem}

\subsection{$1$-dimensional $(\beta, s)$-circular Riesz ensemble} \label{ss:RE}
In this section, we apply Thm.~\ref{t:GT} to prove $\BE(0,\infty)$ in the case of the law~$\QP=\QP_{\beta, s}$ of the $(\beta, s)$-circular Riesz ensemblefor every~$\beta>0$ and $s \in (0,1)$ on~$\U(\R)$. 
Let $g(x)=|x|^{-s}$ with $s \in (0, 1)$ for $x \in \R$. Define 
\begin{align*}
&H^k_r(\gamma):=\sum_{i<j}^kg(x_i-x_j) \comma \quad M_{r, R}^{k, \eta}(\gamma, \eta):=\sum_{i=1}^k\sum_{y \in \eta_{B_r^c},\ |y| \le R}\bigl(g(x_i-y)-g(y)\bigr) \comma
\\
&\Psi_{r, R}^{k, \eta}(\gamma):= \beta \Bigl(H^k_r(\gamma)+M_{r, R}^{k, \eta}(\gamma, \eta)\Bigr)  \quad \text{for $\gamma=\sum_{i=1}^k \delta_{x_i} \in \U^k(B_r)$ and $\eta \in \U(\R)$} \fstop
\end{align*}

\begin{prop} \label{p:conv2}
$\Psi_{r, R}^{k, \eta}$ is geodesically convex in $(\U^{k}(B_r), \mssd_{\U})$ for any $0<r<R<\infty$, $k \in \N$, $\eta \in \U(B_r^c)$ and $\beta>0$.
\end{prop}
\begin{proof}
Let $H_{ij}, H_i^y$ be the Hessian matrices of the functions $(x_1, \ldots, x_k) \mapsto g(x_i-x_j)$ and~$(x_1, \ldots, x_k) \mapsto g(x_i-y)-g(y)$ respectively. 
By observing
\begin{align} \label{e:HCP2}
&\mathbf v H_{ij} \mathbf v^t = \frac{s(s+1)(v_i-v_j)^2}{|x_i-x_j|^{s+2}}, \quad  \mathbf v H^y_{i} \mathbf v^t = \frac{ s(s+1)v_i^2}{|y-x_i|^{s+2}} \comma
\\
& \mathbf v=(v_1, \ldots, v_k) \in \R^{k} \comma \notag
\end{align}
the same proof works as in Prop.~\ref{p:conv}. 
\end{proof}
The following result is due to Dereudre--Vasseur~\cite{DerVas21}.

\begin{thm}[{\cite[Lem.~1.7, Thm.~1.8]{DerVas21}}]\label{t:BRG}
There exists a translation-invariant Borel probability measure $\QP=\QP_{\beta, s}$ whose intensity measure~$I_\QP$ is the Lebesgue measure  such  that the pointwise limit $\Phi_{r}^{k, \eta}(\gamma)=\lim_{R\to \infty}\Phi_{r, R}^{k, \eta}(\gamma)$ exists for every $\gamma \in \U(B_r)$ and $\mu$-a.e.~$\eta$. Furthermore, the following DLR equation holds:
\begin{align} \label{d:cp2}
 \diff \mu_{r}^{k, \eta}=\frac{e^{-\Psi^{k, \eta}_{r}}}{Z_{r}^\eta} \diff \mssm_r^{\odot k}  \comma \quad k \in \mathcal K(\QP_r^\eta) \comma
\end{align}
where  $Z_{r}^\eta:=\sum_{k \in  \mathcal K(\QP_r^\eta)} \int_{\U(B_r)} e^{-\Psi^{k, \eta}_{r}} \diff \mssm_r^{\odot k}$ is the normalisation constant. 
\end{thm}
\begin{rem}
The probability measure $\QP_{\beta, s}$ was constructed as a subsequencial limit of the finite-volume Gibbs measures. The uniqueness of the limit points seems still open, and any translation-invariant limit point is currently called {\it the law of the $(\beta,s)$-circular Riesz gas (or ensemble)}, see e.g., \cite[Prop.~1.5]{DerVas21} for more details. 
\end{rem}
\begin{cor} 
The probability measure~$\QP=\QP_{\beta, s}$ satisfies Assumption~\ref{a:GT} for every~$\beta>0$ and $s \in (0,1)$.
\end{cor}
\begin{proof}
As $I_\QP$ is the Lebesgue measure,  it is obvious that $I_\QP(C)<+\infty$ for every compact set~$C \subset \R$. 
The condition~(a) in~Assumption~\ref{a:GT} follows from Thm.~\ref{t:BRG}. 
The geodesic convexity in~(b) follows from Prop.~\ref{p:conv2} and the pointwise convergence $\Phi_r^{k, \eta}(\gamma) = \lim_{R \to \infty}\Phi_{r, R}^{k, \eta}(\gamma)$ for every $\gamma \in \U(B_r)$ in Thm.~\ref{t:BRG}.  
It suffices to verify that 
\begin{align}\label{g:GM}
\U^k(B_r) \ni \gamma \mapsto e^{-\Psi^{k, \eta}_{r}(\gamma)} \quad \text{is continuous}
\end{align}
for every~$k \in \mathcal K(\mu^\eta_r)$, $\QP$-a.e.~$\eta$, $r>0$; and 
\begin{align}\label{g:GM1}
\U_\e^k(B_r) \ni \gamma \mapsto \Psi^{k, \eta}_{r}(\gamma)\quad \text{is bounded and continuous}
\end{align}
for every~$k \in \mathcal K(\mu^\eta_r)$, $\QP$-a.e.~$\eta$, $r>0$ and $\e>0$.
Thanks to \cite[Lem.~1.7]{DerVas21} (note that the roles of $\gamma$ and $\eta$ there are opposite to this paper), the following pointwise limit exists for every $\gamma \in \U(B_r)$ and $\QP$-a.e.~$\eta$
\begin{align} \label{e:MF}
M_{r}^{k, \eta}(\gamma, \eta):=\lim_{R \to \infty}M_{r, R}^{k, \eta}(\gamma, \eta) <+\infty \comma \quad k \in \mathcal K(\QP_r^\eta)\comma 
\end{align}
and $\Psi^{k, \eta}_{r}(\gamma, \eta)$ can be written as
$$\Psi^{k, \eta}_{r}(\gamma, \eta)=\lim_{R \to \infty}\Psi^{k, \eta}_{r, R}(\gamma, \eta)=\beta \Bigl(H^k_r(\gamma)+M_{r}^{k, \eta}(\gamma, \eta)\Bigr) \fstop$$
Furthermore, the convergence~\eqref{e:MF} is uniform in $\gamma$ (see \cite[the proof of Lem.~1.7 on p.1047]{DerVas21}. Note that the roles of $\gamma$ and $\eta$ are opposite).
Thus, $\gamma \mapsto M_{r}^{k, \eta}(\gamma, \eta)$ is continuous in $\U(B_r)$ for $\QP$-a.e.~$\eta$, which implies \eqref{g:GM} and \eqref{g:GM1}. 
\end{proof}

\begin{cor}\label{c:BRE}
The Dirichlet form~$(\E^{\U, \QP}, \dom{\E^{\U, \QP}})$ defined in~\eqref{eq:Temptation} with $\QP=\QP_{\beta, s}$ satisfies  (a)--(m) in Thm.~\ref{t:GT} with $K=0$  for every $\beta>0$ and $s \in (0,1)$. 
\end{cor}

\begin{appendix}
\section{}
Let~$(X, \tau)$ be a locally compact Polish space and $\mssm$ be a Radon measure on $(X, \tau)$. Let $\mathcal B_b(X)^\mssm$ denote the space of real-valued bounded $\mathscr B(X)^\mssm$-measurable functions. For~$\rep f$, $\rep g$ in $\mathcal B_b(X)^\mssm$, denote by~$f=[\rep f]$,~$g=[\rep g]$, the corresponding $\mssm$-classes.
For~$\rep\ell\colon \mathcal B_b(X)^\mssm\rar \mathcal B_b(X)^\mssm$, define the following properties
\begin{enumerate*}[$(a)$]
\item\label{i:d:Liftings:1} $\tclass{\rep\ell(\rep f)}=f$; 
\item\label{i:d:Liftings:2} if $f=g$, then~$\rep\ell(\rep f)=\rep\ell(\rep g)$;
\item\label{i:d:Liftings:3} $\rep\ell(\car)=\car$;
\item\label{i:d:Liftings:4} if~$\rep f\geq 0$, then~$\rep\ell(\rep f)\geq 0$;
\item\label{i:d:Liftings:5} $\rep\ell(a\, \rep f+b\, \rep g)=a\, \rep\ell(\rep f)+b\, \rep\ell(\rep g)$ for~$a,b\in \R$;
\item\label{i:d:Liftings:6} $\rep\ell(\rep f \rep g)=\rep\ell(\rep f)\,\rep\ell(\rep g)$;
\item\label{i:d:Liftings:7} $\rep\ell(\phi)=\phi$ for~$\phi\in C_b(X)$.
\end{enumerate*}
\begin{defs}[Liftings]\label{d:Liftings}
A \emph{linear lifting} is a map~$\rep\ell\colon\mathcal B_b(X)^\mssm\rar \mathcal B_b(X)^\mssm$
satisfying~\iref{i:d:Liftings:1}--\iref{i:d:Liftings:5}.
Any such~$\rep\ell$ is a (\emph{multiplicative}) \emph{lifting} if it satisfies~\iref{i:d:Liftings:1}--\iref{i:d:Liftings:6}. 
Finally, it is a \emph{strong lifting} if it satisfies~\iref{i:d:Liftings:1}--\iref{i:d:Liftings:7}.
A \emph{Borel} (linear) \emph{lifting} is a (linear) lifting with~$\rep\ell\tparen{\mathcal B_b(X)^\mssm}\subset \mathcal B_b(X)$.
Thanks to~\iref{i:d:Liftings:2}, a linear lifting~$\rep\ell\colon \mathcal B_b(X)^\mssm \rar \mathcal B_b(X)^\mssm$ descends to a linear order-preserving inverse~$\ell\colon L^\infty(\mssm)\rar\mathcal B_b(X)^\mssm$ of the quotient map~$\class[\mssm]{\emparg}\colon\mathcal B_b(X)^\mssm\rar L^\infty(\mssm)$.
Conventionally, by a (\emph{linear}/\emph{multiplicative}/\emph{strong}/\emph{Borel}) \emph{lifting} we shall mean without distinction either~$\rep\ell$ or~$\ell$ as above.
\end{defs}

\begin{thm}[e.g.,~{\cite[Thm.~4.12]{StrMacMus02}}]
$(X, \tau, \mssm)$ admits a strong Borel lifting.
\end{thm}

 Let $\U=\U(\R)$ and recall that for $\eta \in \U$, we set $\U_r^\eta:=\{\gamma \in \U: \gamma_{B_r^c}=\eta_{B_r^c}\}$.
\begin{lem}[disintegration lemma] \label{l:sp}
Assume that there exists a measurable set $\Xi \subset \U$ with $\QP(\Xi)=1$ so that for every $\eta \in \Xi$, there exists a family of measurable sets $\Omega^\eta \subset \U(B_r)$ so that $\mu_{r}^\eta(\Omega^\eta)=1$ for every $\eta \in \Xi$. Let $\Omega \subset \U$ be the (not necessarily measurable) subset defined by
$$\Omega:=\bigcup_{\eta \in \Xi} \pr_{r}^{-1}(\Omega^\eta) \cap \U_{r}^{\eta} \fstop$$ 
Assume further that there exists a measurable set $\Theta \subset \U$ so that $\Omega \subset \Theta$. Then, $\QP(\Theta)=1$. 
\end{lem}
\paragraph{Caveat} As the set $\Omega$ is defined as {\it uncountable union} of measurable sets, the measurability of~$\Omega$ is not necessarily true in general. The disintegration formula~\eqref{p:ConditionalIntegration2} is, therefore, not necessarily applicable directly to $\Omega$,  which motivates the aforementioned lemma. 
\begin{proof}[Proof of Lem.\ \ref{l:sp}]
Let $\Theta_{r}^{\eta}=\{\gamma \in \U(B_r): \gamma+\eta_{B_r^c} \in \Theta\}$ be a section of~$\Theta$ at~$\eta_{B_r^c}$ as in \eqref{e:SEF2}. Then, $\Omega^{\eta} \subset \Theta_{r}^{\eta}$ by assumption. Thus, $\mu_r^\eta(\Theta_{r}^{\eta}) \ge \mu_r^\eta(\Omega^{\eta})\ge 1$. By the disintegration formula in~\eqref{p:ConditionalIntegration2}, we have that 
\begin{align*}
\mu(\Theta)= \int_{\U} \mu_r^\eta(\Theta_{r}^{\eta}) \diff \mu(\eta)  \ge 1\fstop
\end{align*}
The proof is completed. 
\end{proof}

  \begin{lem}\label{l:FL}
  Let~$\QP$ be a Borel probability on~$\U$ and $\Omega \subset \U$ be a $\QP$-measurable set with~$\QP(\Omega)=1$. Then, there exists $\Omega'\subset \Omega$ with $\QP(\Omega')=1$ and 
  \begin{align}\label{e:FL}
  \QP_r^\eta(\Omega_r^\eta)=1 \cquad \eta \in \Omega' \fstop
  \end{align}
  \end{lem}
  \begin{proof}
  By the disintegration formula~\eqref{p:ConditionalIntegration2}, 
  $$1=\QP(\Omega)=\int_{\U} \QP_r^\eta(\Omega_r^\eta) \diff \QP(\eta) =\int_{\Omega} \QP_r^\eta(\Omega_r^\eta) \diff \QP(\eta) \comma$$
  by which the statement is readily concluded.
  \end{proof}
  
  \begin{lem} \label{l:MU}
Let $(Q, \dom{Q})$ be a {symmetric} closed form on a separable Hilbert space~$H$ endowed with the norm~$\|\cdot\|_H$. Let $\{T_t\}_{t>0}$ and $(A, \dom{A})$ be the corresponding semigroup and infinitesimal generator respectively. Suppose that there exists an algebra $\mathcal C \subset \dom{Q}$ so that $\mathcal C \subset H$ is dense and $T_t \mathcal C \subset \mathcal C$ for every $t>0$.
Then $\mathcal C$ is dense in~$\dom{Q}$.
\end{lem}
\begin{proof}
The inclusion~$T_{t}\dom{A} \subset \dom{A}$ generally holds  for semigroups associated with {symmetric} closed forms. Thus,  combining it with the hypothesis $T_t \mathcal C \subset \mathcal C$, 
$$ T_{t}\bigl(\mathcal C \cap \dom{A}\bigr) \subset \mathcal C \cap \dom{A}\fstop$$
 Thus, by \cite[Thm.~X.49]{ReeSim75}, $\mathcal C \cap \dom{A}$ is dense with respect to the graph norm~$\|\cdot\|_{\dom{A}}$ defined as $\|\cdot\|_{\dom{A}}^2:=\|A \cdot\|_{H}+\|\cdot\|_{H}^2$ in the space $\bigl(A, \dom{A}\bigr)$. Namely, 
 \begin{align*} 
 \text{$\bigl(A, \mathcal C \cap \dom{A}\bigr)$ is essentially self-adjoint} \fstop
 \end{align*}
This implies the density $\mathcal C \subset \dom{Q}$.
Indeed, by taking $u_n \in \mathcal C\cap \dom{A}$ converging to $u \in \dom{A}$ with respect to the graph norm,  a simple integration-by-parts 
 $$Q(u,u) = (-A u, u)_{H} \le \|A u\|_{H}\|u\|_{H}$$
implies that $u_n$ converges to $u$ in the space $\dom{Q}$ endowed with the form norm~$\|\cdot\|_{\dom{Q}}$ defined in \eqref{d:FN}. 
In view of the density of $\dom{A} \subset \dom{Q}$, which is a general fact for {symmetric} closed forms, the proof is complete.
  \end{proof}

\end{appendix}
\bibliographystyle{alpha}
\bibliography{Curvature_submission.bib}

\begin{thebibliography}{{Wan}14}

\bibitem[AES16]{AmbErbSav16}
{Ambrosio, L.}, {Erbar, M.}, and {Savar\'e, G.}
\newblock {Optimal transport, Cheeger energies and contractivity of dynamic
  transport distances in extended spaces}.
\newblock {\em {Nonlinear Anal.}}, 137:77--134, 2016.

\bibitem[AGS08]{AmbGigSav08}
{Ambrosio, L.}, {Gigli, N.}, and {Savar\'e, G.}
\newblock {\em {Gradient Flows in Metric Spaces and in the Space of Probability
  Measures}}.
\newblock {Lectures in Mathematics - ETH Z\"urich}. {Birkh\"{a}user},
  {$2^{\textrm{nd}}$} edition, 2008.

\bibitem[AGS14a]{AmbGigSav14}
{Ambrosio, L.}, {Gigli, N.}, and {Savar\'e, G.}
\newblock {Calculus and heat flow in metric measure spaces and applications to
  spaces with Ricci bounds from below}.
\newblock {\em {Invent.\ Math.}}, 395:289--391, 2014.

\bibitem[AGS14b]{AmbGigSav14b}
{Ambrosio, L.}, {Gigli, N.}, and {Savar{\'{e}}, G.}
\newblock {Metric measure spaces with Riemannian Ricci curvature bounded from
  below}.
\newblock {\em {Duke Math.~J.}}, 163(7):1405--1490, 2014.

\bibitem[AGS15]{AmbGigSav15}
{Ambrosio, L.}, {Gigli, N.}, and {Savar{\'{e}}, G.}
\newblock {Bakry--{\'{E}}mery Curvature-Dimension Condition and Riemannian
  Ricci Curvature Bounds}.
\newblock {\em {Ann.~Probab.}}, 43(1):339--404, 2015.

\bibitem[B{\'E}85]{BakEme85}
{Bakry, D.} and {{\'E}mery, M.}
\newblock {\em {Diffusions hypercontractives}}.
\newblock {Séminaire de Probabilités, XIX, Lecture Notes in Math. 1123}.
  {Springer, Berlin}, 1985.

\bibitem[BGL14]{BakGenLed14}
{Bakry, D.}, {Gentil, I.}, and {Ledoux, M.}
\newblock {\em {Analysis and Geometry of Markov Diffusion Operators}}, volume
  348 of {\em {Grundlehren der mathematischen Wissenschaften}}.
\newblock {Springer}, 2014.

\bibitem[BH91]{BouHir91}
{Bouleau, N.} and {Hirsch, F.}
\newblock {\em {Dirichlet forms and analysis on Wiener space}}.
\newblock {De Gruyter}, 1991.

\bibitem[CF11]{CheFuk11}
{Chen, Z-Q.} and {Fukushima, M.}
\newblock {\em {Symmetric Markov Processes, Time Change, And Boundary Theory}}.
\newblock London Mathematical Society Monographs, Lecture Notes in Math 1123.
  Princeton University Press, 2011.

\bibitem[{Dav}89]{Dav89}
{Davies, E.B.}
\newblock {\em {Heat Kernels and Spectral Theory}}.
\newblock {Cambridge University Press}, 1989.

\bibitem[{Del}21]{LzDS20}
{Dello Schiavo, L.}
\newblock {Ergodic Decomposition of Dirichlet Forms via Direct Integrals and
  Applications}.
\newblock {\em {Potential Anal.}}, 2021.

\bibitem[DHLM20]{DerHarLebMai20}
{Dereudre, D.}, {Hardy, A.}, {Leblé, T.}, and {Maïda, M}.
\newblock {DLR Equations and Rigidity for the Sine-Beta Process}.
\newblock {\em {Commun. Pure Appl. Math.}}, pages 172--222, 2020.

\bibitem[DS21a]{LzDSSuz21}
{Dello Schiavo, L.} and {Suzuki, K.}
\newblock {Configuration spaces over singular spaces --I. Dirichlet-Form and
  Metric Measure Geometry --}.
\newblock {\em {arXiv:2109.03192v2} (version 2)}, 2021.

\bibitem[DS21b]{LzDSSuz20}
{Dello Schiavo, L.} and {Suzuki, K.}
\newblock {On the Rademacher and Sobolev-to-Lipschitz Properties for Strongly
  Local Dirichlet Spaces}.
\newblock {\em {J. Func. Anal.}}, 281(11):Online first, 2021.

\bibitem[DS22]{LzDSSuz22}
{Dello Schiavo, L.} and {Suzuki, K.}
\newblock {Configuration Spaces over Singular Spaces II -- Curvature}.
\newblock {\em {arXiv:2205.01379}}, 2022.

\bibitem[DV23]{DerVas21}
{Dereudre, D} and {Vasseur, T}.
\newblock {Number-Rigidity and $\beta$-Circular Riesz gas}.
\newblock {\em {Ann.\ Probab.}}, 51(3), 2023.

\bibitem[Dys62]{Dys62}
F.~J. Dyson.
\newblock {A Brownian-motion model for the eigenvalues of a random matrix.}
\newblock {\em {J. Math. Phys}}, 3:1191--1198, 1962.

\bibitem[EH15]{ErbHue15}
{Erbar, M.} and {Huesmann, M.}
\newblock {Curvature bounds for configuration spaces}.
\newblock {\em {Calc.\ Var.}}, 54:307--430, 2015.

\bibitem[EHJM25]{ErbHueJalMul23}
{Erbar, M}, {Huesmann, M.}, {Jalowy, J.}, and {M\"uller, B.}
\newblock {Optimal Transport Of Stationary Point Processes: Metric Structure,
  Gradient Flow And Convexity Of The Specific Entropy}.
\newblock {\em {J.~Funct.~Anal.~}}, {289, Issue 4}, 2025.

\bibitem[EK86]{EthKur86}
{Ethier, S.N.} and {Kurtz, T.G.}
\newblock {\em {Markov Processes ---Characterization and Convergence---}}.
\newblock {Wiley Inter-science, A JOHN WILEY \& SONS, INC., PUBLICATION}, 1986.

\bibitem[FL07]{FonLeo07}
{Fonseca, I} and {Leoni, G.}
\newblock {\em {Modern Methods in the Calculus of Variations: $L^p$ Spaces}}.
\newblock Springer Monographs in Mathematics. {Springer}, 2007.

\bibitem[FOT11]{FukOshTak11}
{Fukushima, M.}, {Oshima, Y.}, and {Takeda, M.}
\newblock {\em {Dirichlet forms and symmetric Markov processes}}, volume~19 of
  {\em {De Gruyter Studies in Mathematics}}.
\newblock {de Gruyter}, extended edition, 2011.

\bibitem[{Fre}08]{Fre00}
{Fremlin, D.~H.}
\newblock {\em {Measure Theory -- Volume I - IV, V Part I \& II}}.
\newblock {Torres Fremlin (ed.)}, 2000-2008.

\bibitem[GKMS18]{GalKelMonSos18}
{Galaz-Garc{\'{i}}a, F.}, {Kell, M.}, {Mondino, A.}, and {Sosa, G.}
\newblock {On quotients of spaces with Ricci curvature bounded below}.
\newblock {\em {J.~Funct.~Anal.}}, 275:1368--1446, 2018.

\bibitem[GMS15]{GigMonSav15}
{Gigli, N.}, {Mondino, A.}, and {Savar\'e, G.}
\newblock {Convergence of pointed non-compact metric measure spaces and
  stability of Ricci curvature bounds and heat flows }.
\newblock {\em {Proc. London Math. Soc.}}, 111(3):1071--1129, 2015.

\bibitem[GV20]{GyuVon20}
{Gy\"unesu, B.} and {Von Renesse, M..}
\newblock {Molecules as metric measure spaces with Kato-bounded Ricci
  curvature}.
\newblock {\em {Comptes Rendus. Mathématique}}, 358:595--602, 2020.

\bibitem[HR03]{HinRam03}
{Hino, M.} and {Ram{\'{i}}rez, J. A.}
\newblock {Small-Time Gaussian Behavior of Symmetric Diffusion Semigroups}.
\newblock {\em {Ann. Probab.}}, 31(3):1254--1295, 2003.

\bibitem[JKO98]{JorKinOtt98}
{Jordan, R}, {Kinderlehrer, D.}, and {Otto F.}
\newblock {The Variational Formulation of the Fokker--Planck Equation}.
\newblock {\em {SIAM Journal on Mathematical Analysis}}, 29:1--17, 1998.

\bibitem[KOT21]{KawOsaTan21}
{Kawamoto Y.}, {Osada H.}, and {Tanemura H.}
\newblock {Uniqueness of Dirichlet forms related to infinite systems of
  interacting Brownian motions}.
\newblock {\em {Potential Anal.}}, 55:639--676, 2021.

\bibitem[KOT22]{KawOsaTan22}
{Kawamoto Y.}, {Osada H.}, and {Tanemura H.}
\newblock {Infinite-dimensional stochastic differential equations and tail
  $\sigma$-fields II: the IFC condition}.
\newblock {\em {J. Math.~Soc.~Japan}}, 74:79--128, 2022.

\bibitem[KS09]{KilSto09}
{Killip, R.} and {Stoiciu, M.}
\newblock {Eigenvalue statistics for CMV matrices: from Poisson to clock via
  random matrix ensembles}.
\newblock {\em {Duke Math. J.}}, 146 (3):361--399, 2009.

\bibitem[KS21]{KopStu21}
{Kopfer, E.} and {Sturm, K-Th.}
\newblock {Functional inequalities for the heat flow on time-dependent metric
  measure spaces}.
\newblock {\em {J. London Math. Soc.}}, 104-2:926--955, 2021.

\bibitem[KT10]{KatTan10}
{Katori, M.} and {Tanemura, H.}
\newblock {Non-equilibrium dynamics of Dyson’s model with an infinite number
  of particles}.
\newblock {\em {Comm.~Math.~Phys.}}, 293(2):469--497, 2010.

\bibitem[Li15]{Li15}
H.~Li.
\newblock {Dimension-Free Harnack Inequalities on $\RCD(K, \infty)$ Spaces}.
\newblock {\em {J. Theoret. Probab.}}, 29:1280--1297, 2015.

\bibitem[{Meh}04]{Meh04}
{Mehta, M.L..}
\newblock {\em {Random Matrices, 3rd Edition}}.
\newblock {Amsterdam, Elsevier}, 2004.

\bibitem[{Mos}94]{Mos94}
{Mosco, U.}
\newblock {Composite Media and Asymptotic Dirichlet Forms}.
\newblock {\em {J. Funct. Anal.}}, 123:368--421, 1994.

\bibitem[MR90]{MaRoe90}
{Ma, Z.-M.} and {R\"ockner, M.}
\newblock {\em Introduction to the Theory of (Non-Symmetric) Dirichlet Forms}.
\newblock Springer, 1990.

\bibitem[MR00]{MaRoe00}
{Ma, Z.-M.} and {R{\"{o}}ckner, M.}
\newblock {Construction of Diffusions on Configuration Spaces}.
\newblock {\em {Osaka J.~Math.}}, 37:273--314, 2000.

\bibitem[MS20]{MurSav20}
{Muratori, M.} and {Savar\'e, G.}
\newblock {Gradient flows and Evolution Variational Inequalities in metric
  spaces. I: Structural properties}.
\newblock {\em {J. Funct. Anal.}}, 278 (4), 2020.

\bibitem[{Nak}14]{Nak14}
{Nakano, F.}
\newblock {Level statistics for one-dimensional Schrödinger operators and
  Gaussian beta ensemble}.
\newblock {\em {J. Stat. Phys.}}, 156(1):66--93, 2014.

\bibitem[NF98]{NagFor98}
{Nagao, T} and {Forrester, P. J.}
\newblock {Multilevel dynamical correlation functions for Dyson's Brownian
  motion model of random matrices}.
\newblock {\em {Physics Letters A}}, 247:801--850, 1998.

\bibitem[OO23]{OsaOsa23}
{Osada, H.} and {Osada, S.}
\newblock {Ergodicity of unlabeled dynamics of Dyson’s model in infinite
  dimensions}.
\newblock {\em {J. Math. Phys.}}, 64(4), 2023.

\bibitem[{Osa}96]{Osa96}
{Osada, H.}
\newblock {Dirichlet Form Approach to Infinite-Dimensional Wiener Processes
  with Singular Interactions}.
\newblock {\em {Comm.\ Math.\ Phys.}}, 176:117--131, 1996.

\bibitem[{Osa}12]{Osa12}
{Osada, H.}
\newblock {Infinite-dimensional stochastic differential equations related to
  random matrices}.
\newblock {\em {Prob.\ Theory Relat. Fields}}, 153(1):471--509, 2012.

\bibitem[{Osa}13]{Osa13}
{Osada, H.}
\newblock {Interacting Brownian Motions in Infinite Dimensions with Logarithmic
  Interaction Potentials}.
\newblock {\em {Ann.\ Probab.}}, 41(1):1--49, 2013.

\bibitem[OT20]{OsaTan20}
{Osada, H.} and {Tanemura, H.}
\newblock {Infinite-dimensional stochastic differential equations and tail
  $\sigma$-fields}.
\newblock {\em {Probab.\ Theory Relat.\ Fields}}, 177:1137--1242, 2020.

\bibitem[RS75]{ReeSim75}
{Reed, M.} and {Simon, B.}
\newblock {\em {Methods of Modern Mathematical Physics II -- Fourier Analysis,
  Self-Adjointness}}.
\newblock Academic Press, New York, London, 1975.

\bibitem[RS80]{ReeSim80}
{Reed, M.} and {Simon, B.}
\newblock {\em {Methods of Modern Mathematical Physics I -- Functional
  Analysis}}.
\newblock Academic Press, New York, London, 1980.

\bibitem[RS99]{RoeSch99}
{R{\"o}ckner, M.} and {Schied, A.}
\newblock {Rademacher's Theorem on Configuration Spaces and Applications}.
\newblock {\em {J.\ Funct.\ Anal.}}, 169(2):325--356, 1999.

\bibitem[{Sav}14]{Sav14}
{Savar{\'{e}}, G.}
\newblock {Self-Improvement of the Bakry--{\'{E}}mery Condition and Wasserstein
  Contraction of the Heat Flow in $\mathrm{RCD}(K,\infty)$ Metric Measure
  Spaces}.
\newblock {\em {Discr.\ Cont.\ Dyn.\ Syst.}}, 34(4):1641--1661, 2014.

\bibitem[SMM02]{StrMacMus02}
{Strauss, W.}, {Macheras, N.~D.}, and {Musia{\l}, K.}
\newblock {Liftings}.
\newblock In {Pap, E.}, editor, {\em {Handbook of Measure Theory: In two
  volumes}}, chapter~28, pages 1131--1184. {North Holland}, 2002.

\bibitem[{Spo}87]{Spo87}
{Spohn, H.}
\newblock {Interacting Brownian Particles: A Study of Dyson’s Model}.
\newblock {\em {Hydrodynamic Behavior and Interacting Particle Systems}}, pages
  151--179, 1987.

\bibitem[{Suz}24]{Suz23}
{Suzuki, K.}
\newblock {On The Ergodicity Of Interacting Particle Systems Under Number
  Rigidity}.
\newblock {\em {Probab. Theory and Relat. Fields}}, 188:583--623, 2024.

\bibitem[{Suz}04]{Suz24}
{Suzuki, K.}
\newblock {The Infinite Dyson Brownian Motion with beta=2 Does Not Have a
  Spectral Gap}.
\newblock {\em {Bulletin of the London Mathematical Society}}, 2024, DOI:
  \url{https://doi.org/10.1112/blms.13204}.

\bibitem[{Tsa}16]{Tsa16}
{Tsai, L.-C.}
\newblock {Infinite dimensional stochastic differential equations for Dyson’s
  model}.
\newblock {\em {Probability Theory and Related Fields}}, 166:801--850, 2016.

\bibitem[{Vil}09]{Vil09}
{Villani, C.}
\newblock {\em {Optimal transport, old and new}}, volume 338 of {\em
  {Grundlehren der mathematischen Wissenschaften}}.
\newblock {Springer-Verlag}, 2009.

\bibitem[VV09]{ValVir09}
{Valkó, B.} and {Virág, B.}
\newblock {Continuum limits of random matrices and the Brownian carousel}.
\newblock {\em {Invent. math.}}, 177(3):463--508, 2009.

\bibitem[{Wan}14]{Wan14}
{Wang. F-Y.}
\newblock {\em {Analysis for diffusion processes on Riemannian manifolds}},
  volume~18.
\newblock {World Scientific,}, 2014.

\end{thebibliography}
\end{document}